\documentclass[10pt]{article}
\usepackage[english]{babel}
\usepackage{url}
\usepackage{inputenc}
\usepackage{amsfonts}
\usepackage{amsmath}
\usepackage{empheq}
\usepackage{graphicx}
\graphicspath{{images/}}
\usepackage{parskip}
\usepackage{fancyhdr}
\usepackage{bbold}
\usepackage{amssymb}
\usepackage{vmargin}
\usepackage{array}
\usepackage{amsthm}
\usepackage[ruled,vlined]{algorithm2e}
\usepackage{hyperref}
\usepackage{caption}
\usepackage{float}
\usepackage{fourier-orns}
\usepackage{listings}
\usepackage{epsfig}

 \newtheorem{theorem}{Theorem}[section]

 \newtheorem{lemma}[theorem]{Lemma}
 \newtheorem{prop}[theorem]{Proposition}
 \newtheorem{remark}[theorem]{Remark}
 
 \newtheorem{assumption}[theorem]{Assumption}

\usepackage{listings}

\newcommand{\cha}{\mathbb{1}_{\Omega_0}}
\newcommand{\out}{\mathbb{1}_{\mathbb{R}^2\setminus \Omega_0}}

\newcommand{\om}{{\Omega_{0}}}

\newcommand{\mA}{{\mathbb{A}}}

\newcommand{\R}{{\mathbb{R}}}
\newcommand{\G}{{\mathbb{G}}}
\newcommand{\tp}{{\tilde{p}}}
\newcommand{\bydef}{\stackrel{\mbox{\tiny\textnormal{\raisebox{0ex}[0ex][0ex]{def}}}}{=}}
\newcommand{\gdag}{{\gamma^\dagger}}

\usepackage[shortlabels]{enumitem}
\mathtoolsset{showonlyrefs=true}

\usepackage[dvipsnames]{xcolor}

\title{Stationary  non-radial localized  patterns in the planar Swift-Hohenberg PDE: constructive proofs of existence
}

\author{
Matthieu Cadiot
\footnote{McGill University, Department of Mathematics and Statistics, 805 Sherbrooke Street West, Montreal, QC, H3A 0B9, Canada. {\tt matthieu.cadiot@mail.mcgill.ca}}
\and
Jean-Philippe Lessard \footnote{McGill University, Department of Mathematics and Statistics, 805 Sherbrooke Street West, Montreal, QC, H3A 0B9, Canada. {\tt jp.lessard@mcgill.ca}}
\and
Jean-Christophe Nave \footnote{McGill University, Department of Mathematics and Statistics, 805 Sherbrooke Street West, Montreal, QC, H3A 0B9, Canada. {\tt jean-christophe.nave@mcgill.ca}}
}

\date{}

\begin{document}

\maketitle

\begin{abstract}
In this paper, we present a methodology for establishing constructive proofs of existence of smooth, stationary, non-radial localized patterns in the planar Swift-Hohenberg equation. Specifically, given an approximate solution $u_0$, we construct an approximate inverse for the linearization around $u_0$, enabling the development of a Newton-Kantorovich approach. Consequently, we derive a sufficient condition for the existence of a unique localized pattern in the vicinity of $u_0$. The verification of this condition is facilitated through a combination of analytic techniques and rigorous numerical computations. Moreover, an additional condition is derived, establishing that the localized pattern serves as the limit of a family of periodic solutions (in space) as the period tends to infinity. The integration of analytical tools and meticulous numerical analysis ensures a comprehensive validation of this condition. To illustrate the efficacy of the proposed methodology, we present computer-assisted proofs for the existence of three distinct unbounded branches of periodic solutions in the planar Swift-Hohenberg equation, all converging towards a localized planar pattern, whose existence is also proven constructively. All computer-assisted proofs, including the requisite codes, are accessible on GitHub at \cite{julia_cadiot}.
\end{abstract}

\begin{center}
\end{center}

\begin{center}
\end{center}

\section{Introduction}

In this paper, we investigate the existence (and local uniqueness) of smooth,  stationary, non-radial localized patterns in the planar Swift-Hohenberg (SH) equation \cite{Swift_original}
\begin{equation}\label{eq : swift original}
      u_t = -\left((I_d+\Delta)^2u +  \mu u +  \nu_1 u^2 +  \nu_2u^3\right) \bydef - \mathbb{F}(u), \quad u = u(x,t), ~~ x \in \R^2,
\end{equation}
where $\mu >0$ and $(\nu_1, \nu_2) \in \R^2$ are given parameters.
Note that the sign of $\mu$ is essential in the analysis of this paper but $\nu_1$ and $\nu_2$ can be chosen freely.\\
The SH equation is a well-established partial differential equation (PDE) model for pattern formation which finds applications in fields as diverse as phase-field crystals \cite{phase-crystal}, magnetizable fluids \cite{GROVES20171} and nonlinear optics \cite{Odent:2016gli}. Its noteworthy feature of generating localized patterns, often in the form of spatially confined structures, offers valuable insights into the underlying dynamics and stability of complex systems. The existence and dynamics of localized patterns in \eqref{eq : swift original} have been extensively studied in the past decades (e.g. see \cite{Knobloch2008open_problems} or \cite{Knobloch2015spatial} for an introduction to the subject). Comprehensive mathematical analysis, complemented by numerical experiments, has played a pivotal role in revealing the complexities inherent in the pattern formation process within the SH equation \eqref{eq : swift original}. Notably, homoclinic snaking \cite{burke2007snakes, avitabile2010snake}, coupled with bifurcation theory \cite{radial2019Bramburger, burke2006localized, budd2005localized}, and careful numerical simulations, has significantly enhanced our understanding of the formation of symmetric planar patterns, such as hexagonal \cite{hexagon2008, hexagon2021lloyd}, radial \cite{radial2019Bramburger, radial2009, mccalla2013spots}, stripe \cite{hexagon2008} and square \cite{squareSakaguchi_1997} patterns. 
Moreover, leveraging the reversibility of the equation and its first integral, proofs of localized patterns can be derived under certain hypotheses \cite{ladder2009Beck}. Specifically, for $\mu$ small in \eqref{eq : swift original}, several existence results have been obtained using bifurcation arguments, allowing for the proof of existence of branches of patterns with $\mu \in (0, \mu^*)$, for some $\mu^*>0$, using the implicit function theorem or fixed-point theorems. Some examples of such proofs may be found in articles such as \cite{ladder2009Beck, existence2019sandstede, stability1997Mielke, radial2009, hexagon2008, radial2019Bramburger, hexagon2021lloyd}. Finally, without assuming $\mu$ small, an existence proof  of a radially symmetric planar pattern in \eqref{eq : swift original} was recently proposed in \cite{radial2023lessard} by solving a projected boundary value problem and using a rigorous enclosure of a local center-stable manifold. 

In general, establishing the existence of stationary patterns for PDEs defined on unbounded domains, {\em without} imposing assumptions on parameters or constraining symmetries (e.g. radial), is a notoriously difficult task. Notably, the analytical intricacies diverge significantly from the bounded case due to the loss of compactness in the resolvent of differential operators. The present paper addresses these challenges within the context of the SH equation \eqref{eq : swift original}, presenting a general (computer-assisted) method to constructively prove the existence of planar non-radial localized patterns. This result is, to the best of our knowledge, a new result in the pattern formation field. While the techniques and estimates presented in the present paper focus on the SH equation, it is important to emphasis that they are readily generalizable to a class of planar reaction-diffusion PDEs, as described by the assumptions in \cite{unbounded_domain_cadiot}.

Our methodology builds upon the framework established in \cite{unbounded_domain_cadiot}, and it is crucial to underscore that certain modifications are required to examine equation \eqref{eq : swift original} defined on $\R^2$, as elaborated later. The method first relies upon the availability of a numerical approximation, denoted as $u_0$. Such an approximation is supposed to have its support contained on a square $\om = (-d,d)^2$. Equivalently, $u_0$ can be represented by a Fourier series defined on $\om$.  Additionally, $u_0$ is required to belong to a Hilbert space of smooth functions on $\R^2$, exhibiting vanishing behavior at infinity. To meet this criterion, a specific Hilbert space, denoted as $H^l_{D_2}$, is introduced as a subset of $H^4(\mathbb{R}^2)$ (see \eqref{def : definition of Hl} for its specific definition). Elements in $H^l_{D_2}$ possess $D_2$-symmetry, signifying invariance under reflection about the $x$ and $y$ axes. This symmetry serves to isolate solutions by eliminating natural translation and rotation invariance. It is noteworthy that the constraint to $D_2$-symmetric solutions is not the exclusive means of isolating solutions (see Remark~\ref{rem : not D2 symmetric}). 
However, for the purposes of simplifying the analysis and reducing computational complexity, our focus here is on $D_2$-symmetry. Supposing $u_0 \in H^l_{D_2}$, the objective is to identify a solution $\tilde{u} \in H^l_{D_2}$ of equation  \eqref{eq : swift original} in proximity to $u_0$. This involves the construction of an approximate inverse $\mathbb{A}$ for the Fréchet derivative $D\mathbb{F}(u_0)$ and the formulation of a fixed-point operator $\mathbb{T}$ defined as $\mathbb{T}(u)=u-\mathbb{A}\mathbb{F}(u)$. Employing a Newton-Kantorovich approach, the aim is to demonstrate that $\mathbb{T}:\overline{B_r(u_0)} \to \overline{B_r(u_0)}$ is a contraction mapping on a closed ball $\overline{B_r(u_0)}$ centered at $u_0$. This, in turn, enables the conclusion that a unique solution to \eqref{eq : swift original} exists in $H^l_{D_2}$ close to $u_0$, as guaranteed by the Banach fixed-point theorem. Figure~\ref{fig : figure in the introduction} illustrates three distinct approximate solutions $u_0$ for which proofs of existence of localized patterns were successfully obtained via the approach just described. The specific details of these proofs are presented in Theorems~\ref{th : square pattern},~\ref{th : hexagonal pattern}~and~\ref{th : octogonal pattern}. Notably, the well-definedness and contractivity of $\mathbb{T}$ are rigorously verified throughout the explicit computation of various upper bounds, as detailed in Section~\ref{sec : computer assisted analysis}.

  \begin{figure}[h!]
  \centering
  \begin{minipage}{.33\textwidth}
   \centering
  \includegraphics[clip,trim=11cm 0cm 10cm 0cm,width=.92\textwidth]{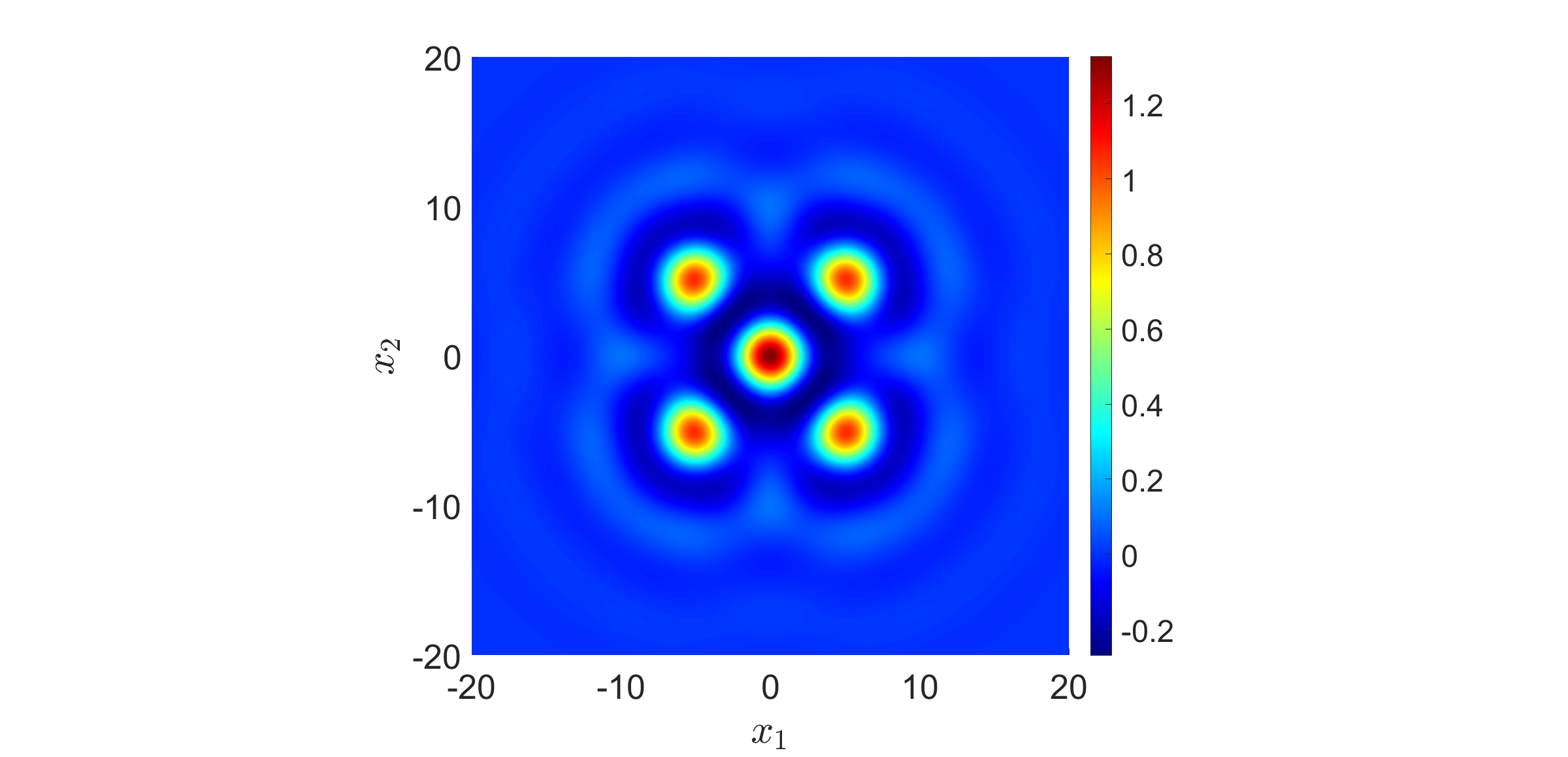}
  \end{minipage}%
  \begin{minipage}{.33\textwidth}
    \centering
   \includegraphics[clip,trim=11cm 0cm 10cm 0cm,width=.92\textwidth]{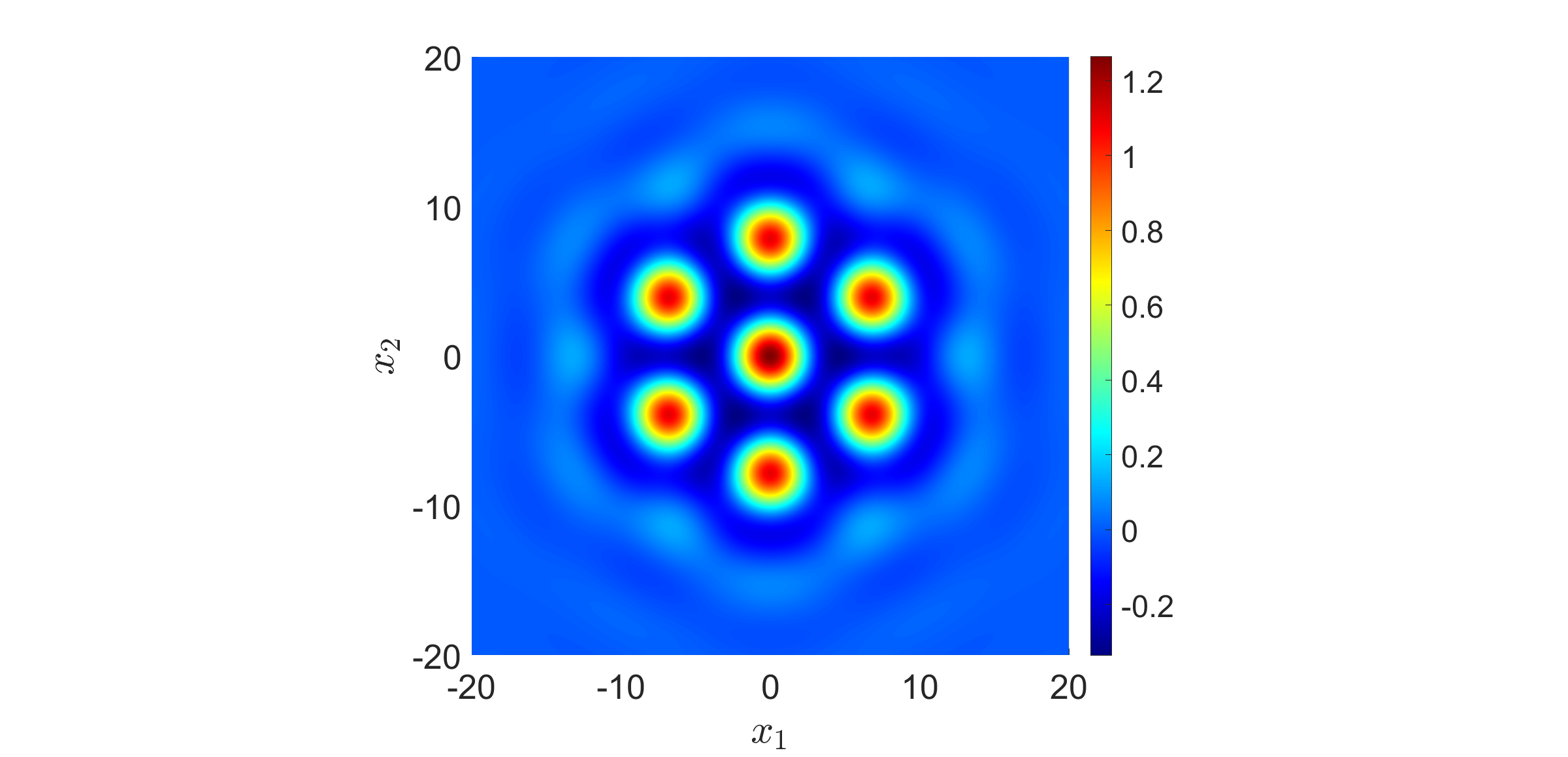}
  \end{minipage}
   \begin{minipage}{0.33\textwidth}
    \centering
   \includegraphics[clip,trim=11cm 0cm 10cm 0cm,width=.92\textwidth]{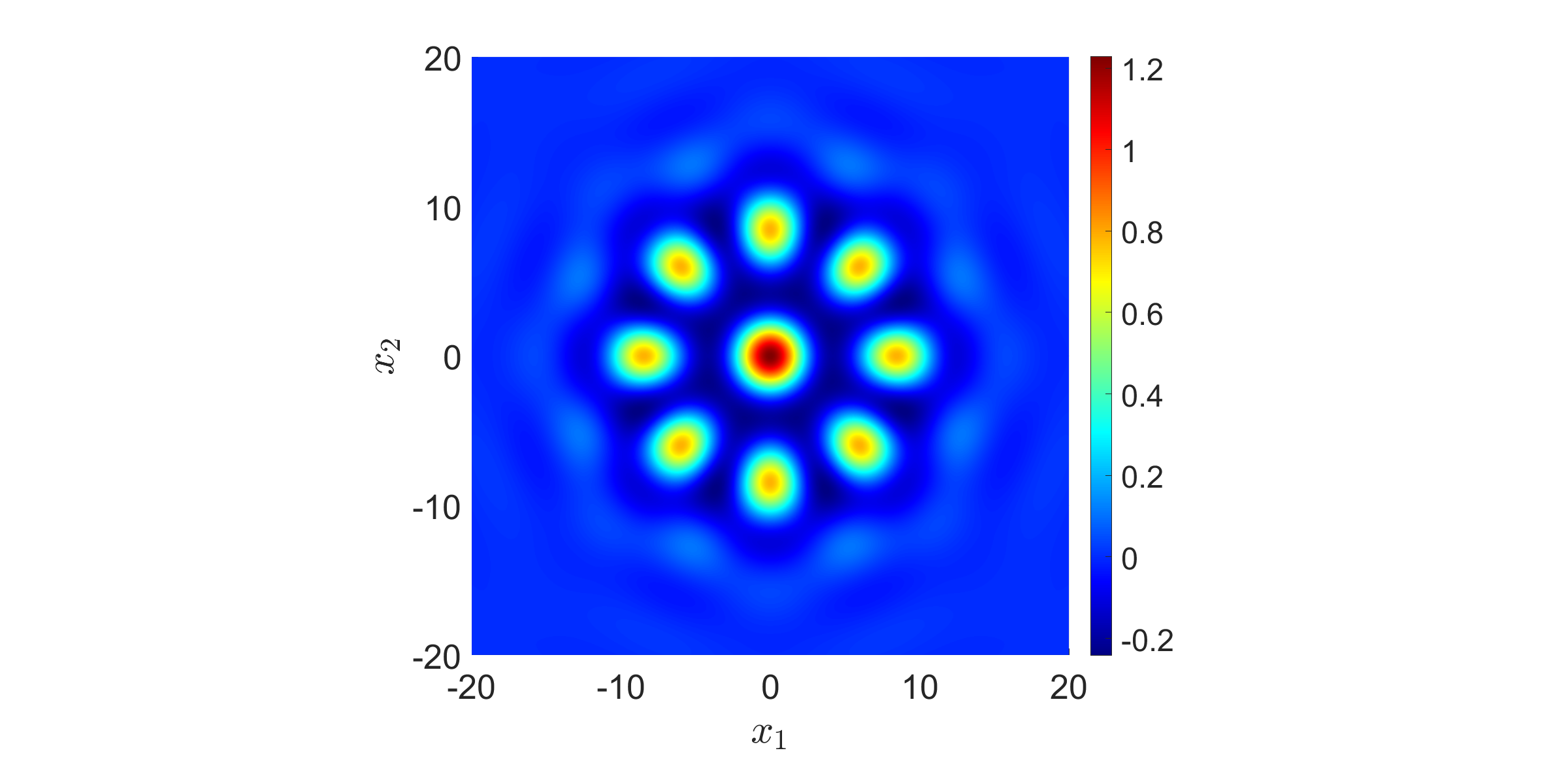}
  \end{minipage}
  \caption{Visualizations of the square pattern (L), hexagonal pattern (C), and octagonal pattern (R), respectively corresponding to Theorems~\ref{th : square pattern},~\ref{th : hexagonal pattern}~and~\ref{th : octogonal pattern}}
    \label{fig : figure in the introduction}
  \end{figure}
As previously mentioned, the application of the framework proposed in \cite{unbounded_domain_cadiot} to the present problem necessitates addressing several technical challenges. 
The methodology stipulates that the approximate solution $u_0$ should be constructed on the domain $\om$ through its Fourier coefficients representation $U_0$. Beyond the domain $\om$, $u_0$ is extended to the zero function. Practical implementation involves the numerical computation of the Fourier coefficients $U_0$, achieved in this paper utilizing the approach developed in \cite{Hill_2023}. It is noteworthy that due to the potential discontinuities at $\partial \om$, the constructed $u_0$ may not be inherently smooth. Specifically, the obtained Fourier coefficients must be projected into the kernel of a periodic trace operator to ensure the smoothness of $u_0$ (at least in $H^4(\R^2)$). The detailed construction of the approximate solution $u_0$ is elucidated in Section~\ref{sec : construction of u0}. Subsequently, our Newton-Kantorovich approach relies on explicit computations of certain upper bounds (i.e. $\mathcal{Y}_0, \mathcal{Z}_1$ and $\mathcal{Z}_2$ in Theorem~\ref{th: radii polynomial}). In particular, leveraging the techniques from \cite{unbounded_domain_cadiot}, we provide formulas for these bounds in the case of the SH equation \eqref{eq : swift original}. Once established, the explicit evaluation of these formulas is attained through rigorous numerical methods, enabling the verification of the existence of a localized pattern by confirming the condition \eqref{condition radii polynomial} in Theorem~\ref{th: radii polynomial}. Note that a particularly intricate challenge addressed in this study is the precise computation of an upper bound $\hat{C}_0$ on the supremum $C_0$ of a known smooth function $g : [0,\infty) \to \mathbb{C}$ (see \eqref{def : definition of function g}). While various methods, including integral and sum estimates, exist for such computations, analytic techniques often fail to provide a sharp bound, or at the very least, do not furnish means for verifying the sharpness of the bound. Consequently, we employ a rigorous computational approach to tackle this issue, the details of which are explained in Section~\ref{sec : computation of C_0}.

In the field of pattern formation, localized patterns have been observed to sometimes manifest as the limit of a family of periodic solutions as their period approaches infinity (see \cite{radial2009, existence2019sandstede,  champneys_2021} for instance and the  conjecture advanced in \cite{localized-rigorous} for the case of the quintic SH).  The authors of \cite{localized-rigorous} proved the existence of multiple 1D periodic solutions and observed that when parameterized by the period, they seemed to converge to a homoclinic connection (localized pattern on the real line). The present work provides the means to verify such a claim for a given planar localized pattern. Assuming the successful establishment of the existence of a localized pattern using the approach outlined earlier, we extend the findings from \cite{unbounded_domain_cadiot}. Specifically, we derive a condition on the bounds $\mathcal{Y}_0, \mathcal{Z}_1$ and $\mathcal{Z}_2$ under which an unbounded branch of (spatially) periodic solutions is obtained, converging to the localized pattern as the period tends to infinity. This phenomenon is exemplified and demonstrated in Theorems~\ref{th : square pattern},~\ref{th : hexagonal pattern}~and~\ref{th : octogonal pattern}, where we establish the existence of such a branch for three distinct localized patterns. This contribution represents, to the best of our knowledge, a novel result within the domain of localized solutions in semi-linear PDEs.

Before proceeding any further, it is worth mentioning that the use of computer-assisted proofs (CAPs) has by now established itself as an important tool in the analysis of nonlinear PDEs (e.g. refer to \cite{nakao_numerical, gomez_cap, jb_rigorous_dynmamics, koch_computer_assisted} and the book \cite{plum_numerical_verif}), especially in the analysis of \eqref{eq : swift original}. For instance, the novel approach of \cite{rigorous_global} considered \eqref{eq : swift original} on an interval and combined Conley index theory with CAPs of existence of steady states to build a model for the attractor consisting of stationary solutions and connecting orbits. In \cite{lessard-chaotic-rigorous}, a proof of existence of chaos in the form of symbolic dynamics in the stationary SH equation on the line was obtained by combining a CAP of a {\em skeleton} of periodic solutions, parabolic relations, braid theory and a topological forcing theorem. Shortly after, in \cite{MR2718657,MR3077902}, the authors considers the SH equation on 2D/3D rectangular domains with periodic boundary conditions, where they developed analytic estimates in weighted $\ell^\infty$ spaces of Fourier coefficients and used a Newton-Kantorovich type theorem to obtain constructive proofs of existence of steady states. Still on bounded domains, the recent construction \cite{MR4668616} of stable manifolds of equilibria in \eqref{eq : swift original} and the recent works \cite{MR4379799,berg2023validated,IVP_Takayasu} based on Fourier-Chebyshev expansions for solving initial value problems opened the door to rigorous computations of connecting orbits in \eqref{eq : swift original}. In the case of unbounded domains, a constructive proof of existence of a radial localized pattern in \eqref{eq : swift original} was recently proposed in \cite{radial2023lessard}. By studying the equation in polar coordinates, the 2D PDE transforms into an ordinary differential equation (ODE), for which a homoclinic connection at zero is computed via a rigorous enclosure of the center stable manifold, achieved through the use of the Lyapunov-Perron operator. Concurrently, \cite{Hill_2023} delved into the examination of the SH equation \eqref{eq : swift original} with polar coordinates, where their focus was on identifying solutions characterized by a finite expansion in the angle component, giving rise to a finite system of ODEs in the radial component. This system underwent rigorous resolution employing a finite-dimensional Newton-Kantorovich argument.

Finally, it is noteworthy to highlight that there is a gradual emergence of computer-assisted methodologies in investigating PDEs on unbounded domains. Indeed, as mentioned earlier, the loss of compactness in the resolvent of differential operators for PDEs defined on unbounded domains hinders significantly the analysis. Consequently, the development of CAPs for PDEs defined on unbounded domains requires special care. It is worth mentioning that for problems posed on the 1D line, the Parameterization Method, as exemplified in \cite{MR1976079,MR1976080}, provides a means to formulate a projected boundary value problem solvable through Chebyshev series or splines, as detailed in \cite{MR2821596,MR3741385}. While this methodology facilitates the constructive establishment of solutions and provide efficiently the asymptotic dynamics in the stable and unstable manifolds, it however lacks a generalization to ``fully" 2D PDEs, thereby precluding its applicability to the study of planar localized patterns. In \cite{plum_numerical_verif}, Plum et al. present a comprehensive methodology for proving weak solutions to second and fourth-order PDEs. Their approach relies on the rigorous control of the spectrum of the linearization around an approximate solution, incorporating a homotopy argument and the Temple-Lehmann-Goerisch method. Notably, this approach is applicable to unbounded domains, as demonstrated by the authors in establishing the existence of a weak solution to the planar Schrödinger equation. Within the same theoretical framework, Wunderlich, in \cite{plum_thesis_navierstokes}, successfully demonstrated the existence of a weak solution to the Navier-Stokes equations defined on an infinite strip with an obstacle. It is noteworthy to highlight that the approach presented in \cite{plum_numerical_verif} exclusively enables the verification of weak solutions on unbounded domains and does not necessarily provide regularity. In our current work, we adopt the framework proposed in \cite{unbounded_domain_cadiot} and employ it to formulate a general methodology for proving constructively the existence of strong solutions in \eqref{eq : swift original} in the form of planar localized stationary patterns.

The paper is organized as follows. Section~\ref{sec : formulation of the problem} introduces the problem's framework, provides a comprehensive exposition of the setup and introduces the definition of pertinent operators and spaces essential to our investigation. In Section~\ref{sec : computer assisted analysis}, a Newton-Kantorovich approach is introduced for the rigorous constructive proof of existence of localized patterns, outlining the methodology for obtaining and constructing an approximate solution within the space $H^l_{D_2}$. Additionally, explicit computations leading to bounds required by this approach are detailed. Finally, Section~\ref{sec : proof of existence of localized patterns} is devoted to the presentation of proofs regarding the existence of localized patterns. All computer-assisted proofs, including the requisite codes, are accessible on GitHub at \cite{julia_cadiot}.

\section{Set-up of the problem}\label{sec : formulation of the problem}

In this section, we recall some set-up developed in \cite{unbounded_domain_cadiot} and present it in the specific context of the planar Swift-Hohenberg equation. 
Recall the Lebesgue notation $L^2 = L^2(\mathbb{R}^2)$ and $L^2(\Omega_0)$ on a bounded domain $\om$ in $\R^2$. More generally, $L^p$ denotes the usual $p$ Lebesgue space on $\mathbb{R}^2$ associated to its norm $\| \cdot \|_{p}$. Moreover, given $s \in \R$, denote by $H^s \bydef H^s(\R^2)$ the usual Sobolev space on $\R^2$.  For a bounded linear operator $\mathbb{K} : L^2 \to L^2$, denote by $\mathbb{K}^*$ the adjoint of $\mathbb{K}$ in $L^2$. Moreover, if $v \in L^2$, denote by $\hat{v} \bydef \mathcal{F}(v)$ the Fourier transform of $v$, that is 
\[
\hat{v}(\xi) \bydef \int_{\mathbb{R}^2}v(x)e^{-2\pi i x \cdot \xi}dx,
\]
for all $\xi \in \mathbb{R}^2$, where $x \cdot \xi=x_1 \xi_1+ x_2 \xi_2$. Given $\xi \in \R^2$, denote $|\xi| \bydef \sqrt{\xi_1^2 + \xi_2^2}$ the usual Euclidean norm on $\R^2$.  Finally, given $u, v \in L^2$, we denote $u*v$ the convolution of $u$ and $v$.

We wish to prove the existence (and local uniqueness) of localized stationary solutions of the planar Swift-Hohenberg equation. Equivalently, we look for a real-valued $u$ such that
\begin{equation}\label{eq : swift_hohenberg}
     (I_d+\Delta)^2u +  \mu u + \nu_1 u^2 + \nu_2 u^3 =0 
\end{equation}
with $u(x) \to 0 \text{ as } |x| \to \infty$.
 Using the notations introduced in \cite{unbounded_domain_cadiot}, denote 
 $$\mathbb{L} \bydef (I_d+\Delta)^2 + \mu I_d,$$
  where $I_d$ represents the identity operator and $\Delta$ is the usual Laplacian. Moreover we define $l : \mathbb{R}^2 \to \mathbb{R}$ as 
\[
    l(\xi) \bydef  (1-|2\pi\xi|^2)^2 + \mu, \quad \text{for all } \xi \in \mathbb{R}^2.
\]
In other words, $l$ is the symbol associated to the differential operator $\mathbb{L}$, that is $\mathcal{F}\big(\mathbb{L}u\big)(\xi) = l(\xi) \hat{u}(\xi)$. \\
For the nonlinear part, we denote 
\[
\mathbb{G}(u) \bydef \G_2(u) + \G_3(u) =  \nu_1u^2 + \nu_2u^3
\]
where $\G_2(u) = \nu_1u^2$ and $\G_3(u) = \nu_2u^3$. Moreover, we have $\G_2(u) = \left(\G^1_{2,1}u\right)\left( \G^2_{2,1} u\right)$ where $\G^1_{2,1} \bydef \nu_1 I_d$ and $\G^2_{2,1} \bydef I_d$ and $\G_3(u) = \left(\G^1_{3,1}u\right)\left( \G^2_{3,1} u\right) \left(\G^3_{3,1} u\right)$ where $\G^1_{3,1} \bydef \nu_2 I_d$ and $\G^2_{3,1} = \G^3_{3,1} = I_d$, using the notations of \cite{unbounded_domain_cadiot}. In particular, we define $g^p_{i,k} : \R^2 \to \R$ as the Fourier transform of $\G^p_{i,k}$. More specifically, $g^1_{2,1}(\xi) = \nu_1$, $g^1_{3,1}(\xi) = \nu_2$ and $g^2_{2,1}(\xi) = g^2_{3,1}(\xi) = g^3_{3,1}(\xi) =1$ for all $\xi \in \R^2$. 

Equation \eqref{eq : swift_hohenberg} is then equivalent to the zero finding problem $\mathbb{F}(u) = 0$, with $u \to 0$ as $|x| \to \infty$, where
\[
\mathbb{F}(u) \bydef \mathbb{L}u + \mathbb{G}(u).
\]
We recall the assumptions from \cite{unbounded_domain_cadiot} for convenience.

\begin{assumption}\label{ass:A(1)}
Assume that $
|l(\xi)| >0$ for all $\xi \in \mathbb{R}^2$.
\end{assumption}

\begin{assumption}\label{ass : LinvG in L1}
For all $i \in \{2, 3\}$ and $p \in \{1, \dots, i\}$, define $g_{i}^p$ as
\[
\mathcal{F}\bigg(\mathbb{G}_{i}^p(u)\bigg)(\xi) = g_{i}^p(\xi)\hat{u}(\xi),
\quad \text{for all } \xi \in \mathbb{R}^2, \text{ and  assume that } ~\frac{g_{i,k}^p}{l} \in L^1.
\vspace{-.2cm}
\]
\end{assumption}
 
First, notice Assumption \ref{ass:A(1)} is satisfied as we assume $\mu>0$.  Moreover, as the functions $g^p_{i,k}$ are all constants, Assumption \ref{ass : LinvG in L1} is verified if and only if $\frac{1}{l} \in L^1$,  which is trivially satisfied as well. Therefore, the analysis derived in \cite{unbounded_domain_cadiot} is readily applicable to \eqref{eq : swift_hohenberg}.

Denote by $H^l$ the following Hilbert space 
\[
    H^l \bydef \left\{ u \in L^2: \int_{\mathbb{R}^2}|\hat{u}(\xi)|^2|l(\xi)|^2d\xi < \infty \right\}
\]
associated to its natural inner product $(\cdot,\cdot)_l$ and norm defined as
\begin{align}\label{def : inner product and norm on Hl}
    (u,v)_l \bydef \int_{\mathbb{R}^2}\hat{u}(\xi)\overline{\hat{v}(\xi)}|{l}(\xi)|^2d\xi ~~\text{ and }~~
     \|u\|_l \bydef \|\mathbb{L}u\|_2
\end{align}
for all $u,v \in H^l.$ Now, to obtain the well-definedness of the operator $\mathbb G : H^l \to L^2$, we need to ensure that $uv \in L^2$ and $uvw \in L^2$ for all $u, v, w \in H^l$. The next lemma provides such a result.
\begin{lemma}\label{lem : banach algebra}
Let $\kappa >0$ such that $\kappa \geq  \|\frac{1}{l}\|_2$. Then, for all $u, v, w \in H^l$,
\begin{equation}\label{eq : banach algebra}
    \|uw\|_2 \leq \frac{\kappa}{\mu} \|u\|_l \|w\|_l ~ \text{ and } ~ \|u v w\|_2 \leq \frac{\kappa^2}{\mu} \|u\|_l \|v\|_l \|w\|_l.
\end{equation}
\end{lemma}
\begin{proof}
By definition of $l$, we have $\max_{\xi \in \R^2} \frac{1}{|l(\xi)|} =\frac{1}{\mu}$. Then, similarly as what was achieved in Lemma 2.4 in \cite{unbounded_domain_cadiot}, one can easily prove that $\kappa$ satisfies \eqref{eq : banach algebra}.
\end{proof}

In practice, one needs to know explicitly (or at least have an upper bound) for the quantity  $\|\frac{1}{l}\|_2$ in order to use \eqref{eq : banach algebra}. This is achieved in the next proposition.
\begin{prop}
    \begin{equation}\label{eq : equality norm 2 of 1/l}
    \left\|\frac{1}{l}\right\|_2^2 = \frac{2\sqrt{\mu} + (1 + \mu)\left(2\pi - 2\arctan(\sqrt{\mu})\right)}{8\mu^{\frac{3}{2}}(1+\mu)}.
\end{equation}
\end{prop}
\begin{proof}
     We have,
\begin{align*}
    \left\|\frac{1}{l}\right\|_2^2 = \int_{\mathbb{R}^2} \frac{1}{\left(\mu +(1-|2\pi\xi|^2)^2\right)^2}d\xi = \int_{0}^\infty \frac{r}{\left(\mu+(1-r^2)^2\right)^2}dr
\end{align*}
using polar coordinates. Then, using standard integration techniques for rational functions (see \cite{gradshteyn2014table} for instance), we prove \eqref{eq : equality norm 2 of 1/l}. 
\end{proof}

Using Lemma \ref{lem : banach algebra}, we obtain that the operator $\mathbb{G}$ is smooth from $H^l$ to $L^2$. This implies that the zero finding problem $\mathbb{F}(u) =0$ is well defined on $H^l.$ The condition $u  \to 0$ as $|x| \to \infty$ is satisfied implicitly if $u \in H^l.$

Now supposing that $u$ is a solution of \eqref{eq : swift_hohenberg}, then any translation and rotation of $u$ is still a solution. Therefore, in order to isolate a particular solution in the set of solutions, we choose to look for solutions that are invariant under reflections about the $x$-axis and the $y$-axis. In other words, we restrict ourselves to $D_2$-symmetric solutions. Therefore, we define $H^l_{D_2}$ as the following Hilbert subspace of $H^l$ which takes into account these symmetries :
\begin{align}\label{def : definition of Hl}
    H^l_{D_2} \bydef \left\{ u \in H^l : u(x_1,x_2) = u(-x_1,x_2) = u(x_1, -x_2) \text{ for all } x_1,x_2 \in \mathbb{R} \right\}.
\end{align}
Similarly, denote by $L^2_{D_2}$ the Hilbert subspace of $L^2$ satisfying the $D_2$-symmetry.
In particular we notice that if $u \in H^l_{D_2}$, then $\mathbb{L}u \in L^2_{D_2}$ and $\mathbb{G}(u) \in L^2_{D_2}$, hence it is natural to define $\mathbb{L}$ and $\mathbb{G}$ as operators from $H^l_{D_2}$ to $L^2_{D_2}$. 

Finally, we look for solutions of the following problem
\begin{equation}\label{eq : f(u)=0 on Hl}
    \mathbb{F}(u) = 0 ~~ \text{ and } ~~ u \in H^l_{D_2}.
\end{equation}

As we look for classical solutions to \eqref{eq : swift_hohenberg}, we need to ensure that solutions to  \eqref{eq : f(u)=0 on Hl} are smooth. The next proposition provides such a result and, consequently, we may focus our analysis on the zeros of $\mathbb{F} : H^l_{D_2} \to L^2_{D_2}$ and obtain the regularity of the solution a posteriori. 

\begin{prop}\label{prop : regularity of the solution}
Let $u \in H^l_{D_2}$ such that $u$ solves \eqref{eq : f(u)=0 on Hl}. Then $u \in H^\infty(\mathbb{R}^2) \bigcap C^\infty(\R^2)$ and $u$ is a classical solution of \eqref{eq : f(u)=0 on Hl}.
\end{prop}
\begin{proof}
The proof is a direct consequence of Proposition 2.5 in \cite{unbounded_domain_cadiot}.
\end{proof}
Finally, denote by $\|\cdot\|_{l,2}$ the operator norm for any bounded linear operator between the two Hilbert spaces $H^l_{D_2}$ and $L^2_{D_2}$. Similarly denote by $\|\cdot\|_{l}$, $\|\cdot\|_2$ and $\|\cdot\|_{2,l}$ the operator norms for bounded linear operators on $H^l_{D_2} \to H^l_{D_2}$, $L^2_{D_2} \to L^2_{D_2}$ and $L^2_{D_2} \to H^l_{D_2}$ respectively.

\begin{remark}\label{rem : not D2 symmetric}
By construction, the space $H^l_{D_2}$ allows eliminating the translation and rotation invariance of the solutions. If one is interested in proving solutions that are not necessarily $D_2$-symmetric, one may use the set-up of Section 5 in \cite{unbounded_domain_cadiot}. Indeed, by appending  extra equations and the same number of unfolding parameters, the solution can be isolated again (see \cite{Lessard2021conserved,Lessard2021conserved_nbody} for instance).
\end{remark}

\subsection{Periodic Sobolev spaces}

In this section we recall some notations introduced in Section 2.4 of \cite{unbounded_domain_cadiot}. We define $\Omega_0 \bydef (-d,d)^2$  where $0<d <\infty$ is a fixed quantity.  Then, we define  $$\tilde{n} = (\tilde{n}_1,\tilde{n}_2)\bydef \left( \frac{n_1}{2d},\frac{n_2}{2d} \right) \in \mathbb{R}^2$$ for all $(n_1,n_2) \in \mathbb{Z}^2$. Similarly as in the continuous case, we want to restrict to Fourier series representing $D_2$-symmetric functions. Given a Fourier series $U = (u_n)_{n \in \mathbb{Z}^2}$ representing a $D_2$-symmetric function, $U$  satisfies
\begin{align}\label{def : equations for the symmetry D2}
    u_{n_1,n_2} = u_{-n_1,n_2} = u_{n_1,-n_2} \text{ for all } (n_1,n_2) \in \mathbb{Z}^2.
\end{align}
Therefore, we  restrict the indexing of $D_2$-symmetric functions to $\mathbb{N}_0^2$, where
\begin{align*}
    \mathbb{N}_0^2 \bydef \left(\mathbb{N} \cup\{0\}\right)^2,
\end{align*}
and construct the full series by symmetry if needed. In other words, $\mathbb{N}_0^2$ is the reduced set associated to the $D_2$-symmetry. 

Let  $(\alpha_n)_{n \in \mathbb{N}_0^2}$ be defined as 
\begin{align}\label{def : alpha_n}
    \alpha_n \bydef \begin{cases}
        1 &\text{ if } n=(0,0)\\
        2 &\text{ if } n_1n_2=0, \text{ but } n \neq (0,0)\\
        4 &\text{ if } n_1n_2 \neq 0
    \end{cases}
\end{align}
and let $\ell^p_{D_2}$ denote the following Banach space
\[
    \ell^p_{D_2} \bydef \left\{U = (u_n)_{n \in \mathbb{N}_0^2}: ~ \|U\|_p \bydef \left( \sum_{n \in \mathbb{N}_0^2} \alpha_n|u_n|^p\right)^\frac{1}{p} < \infty \right\}.
\]
Note that $\ell^p_{D_2}$ possesses the same sequences as the usual $p$ Lebesgue space for sequences indexed on $\mathbb{N}_0^2$ but with a different norm. For the special case $p=2$, $l^2_{D_2}$ is an Hilbert space on sequences indexed on $\mathbb{N}_0^2$ and we denote $(\cdot,\cdot)_2$ its inner product given by
\[
(U,V)_2 \bydef \sum_{n \in \mathbb{N}_0^2} \alpha_n u_n \overline{v_n}
\]
for all $U = (u_n)_{n \in \mathbb{N}_0^2}, V = (v_n)_{n \in \mathbb{N}_0^2} \in \ell^2_{D_2}.$ Moreover, for a bounded operator $K : \ell^2_{D_2} \to \ell^2_{D_2}$, $K^*$ denotes the adjoint of $K$ in $\ell^2_{D_2}$.  

The coefficients $(\alpha_n)_{n \in \mathbb{N}_0^2}$ arise naturally when switching from the usual Fourier basis in $e^{2\pi i \tilde{n}\cdot x}$ to the one in $\cos(2\pi \tilde{n}_1x_1)\cos(2\pi \tilde{n}_2x_2)$, which is specific to $D_2$-symmetric functions. Indeed, given $(u_n)_{n \in \mathbb{Z}^2}$ satisfying \eqref{def : equations for the symmetry D2}, we have
\begin{equation} \label{eq:Fourier_expansion}
    \sum_{n \in \mathbb{Z}^2}  u_n e^{2\pi i \tilde{n} \cdot x} = \sum_{n \in \mathbb{N}_0^2} \alpha_n u_n \cos(2\pi \tilde{n}_1x_1)\cos(2\pi \tilde{n}_2x_2)
\end{equation}
for all $(x_1,x_2) \in \R^2$. Now, similarly as what is done in Section 6 of \cite{unbounded_domain_cadiot}, we  define 
$\gamma : L^2_{D_2} \to \ell^2_{D_2}$
\begin{equation}\label{def : gamma}
    \left(\gamma(u)\right)_n \bydef  \frac{1}{|\om|}\int_\om u(x) e^{-2\pi i \tilde{n}\cdot x}dx
\end{equation}
for all $n \in \mathbb{N}_0^2$. Similarly, we define $\gamma^\dagger : \ell^2_{D_2} \to L^2_{D_2}$ as 
\begin{equation}\label{def : gamma dagger}
    \gamma^\dagger(U)(x) \bydef \cha(x) \sum_{n \in \mathbb{N}_0^2} \alpha_nu_n \cos(2\pi \tilde{n}_1x_1)\cos(2\pi \tilde{n}_2x_2)
\end{equation}
for all $x = (x_1,x_2) \in \R^2$ and all $U =\left(u_n\right)_{n \in \mathbb{N}_0^2} \in \ell^2_{D_2}$, where $\cha$ is the characteristic function on $\om$. Given $u \in L^2_{D_2}$, $\gamma(u)$ represents the Fourier coefficients indexed on $\mathbb{N}_0^2$ of the restriction of $u$ on $\om$. Conversely, given a sequence $U\in \ell^2_{D_2}$, $\gdag\left(U\right)$  is the function representation of $U$ in $L^2_{D_2}.$ In particular, notice that $\gdag\left(U\right)(x) =0$ for all $x \notin \om.$  
Then, recalling similar notations from \cite{unbounded_domain_cadiot}
\begin{align}
    L^2_{D_2,\om} &\bydef \left\{u \in L^2_{D_2} : \text{supp}(u) \subset \overline{\om} \right\}\\
   H_{D_2,\om}^l &\bydef \left\{u \in H^l_{D_2} : \text{supp}(u) \subset \overline{\om} \right\}.
\end{align}
Moreover, recall that $\mathcal{B}(L^2_{D_2})$ (respectively $\mathcal{B}(\ell^2_{D_2})$) denotes the space of bounded linear operators on $L^2_{D_2}$ (respectively $\ell^2_{D_2}$) and denote by $\mathcal{B}_\om(L^2_{D_2})$ the following subspace of $\mathcal{B}(L^2_{D_2})$
\begin{equation}\label{def : Bomega}
    \mathcal{B}_\om(L^2_{D_2}) \bydef \{\mathbb{K}_\om \in \mathcal{B}(L^2_{D_2}) :  \mathbb{K}_\om = \cha \mathbb{K}_\om \cha\}.
\end{equation}
Finally, define $\Gamma : \mathcal{B}(L^2_{D_2}) \to \mathcal{B}(\ell^2_{D_2})$ and $\Gamma^\dagger : \mathcal{B}(\ell^2_{D_2}) \to \mathcal{B}(L^2_{D_2})$ as follows
\begin{equation}\label{def : Gamma and Gamma dagger}
    \Gamma(\mathbb{K}) \bydef \gamma \mathbb{K} \gdag ~~ \text{ and } ~~  \Gamma^\dagger(K) \bydef \gamma^\dagger {K} \gamma 
\end{equation}
for all $\mathbb{K} \in \mathcal{B}(L^2_{D_2})$ and all $K \in \mathcal{B}(\ell^2_{D_2}).$

The maps defined above in \eqref{def : gamma}, \eqref{def : gamma dagger} and \eqref{def : Gamma and Gamma dagger} are fundamental in our analysis as they allow to pass from the problem on $\R^2$ to the one in $\ell^2_{D_2}$ and vice-versa. Furthermore, we show in the following lemma, which is proven in \cite{unbounded_domain_cadiot} using Parseval's identity, that this passage is actually an isometric isomorphism when restricted to the relevant spaces.

\begin{lemma}\label{lem : gamma and Gamma properties}
    The map $\sqrt{|\om|} \gamma : L^2_{D_2,\om} \to \ell^2_{D_2}$ (respectively $\Gamma : \mathcal{B}_\om(L^2_{D_2}) \to \mathcal{B}(\ell^2_{D_2})$) is an isometric isomorphism whose inverse is given by $\frac{1}{\sqrt{|\om|}} \gdag : \ell^2_{D_2} \to L^2_{D_2,\om}$ (respectively $\Gamma^\dagger :   \mathcal{B}(\ell^2_{D_2}) \to \mathcal{B}_\om(L^2_{D_2})$). In particular,
    \begin{align}\label{eq : parseval's identity}
        \|u\|_2 = \sqrt{\om}\|U\|_2 \text{ and } \|\mathbb{K}\|_2 = \|K\|_2
    \end{align}
    for all $u \in L^2_{D_2,\om}$ and $\mathbb{K}\in \mathcal{B}_\om(L^2_{D_2})$, and where $U \bydef \gamma(u)$ and ${K} \bydef \Gamma(\mathbb{K})$.
\end{lemma}

The above lemma not only provides a one-to-one correspondence between the elements in $L^2_{D_2,\om}$ (respectively $\mathcal{B}_\om(L^2_{D_2})$) and the ones in $\ell^2_{D_2}$ (respectively $\mathcal{B}(\ell^2_{D_2})$) but it also provides an identity on norms. This property is essential in our construction of an approximate inverse  in Section \ref{sec : construction of operator A}.

Now, we define the Hilbert space $X^l$ as 
\begin{align*}
    X^l \bydef \left\{ U = (u_n)_{n \in \mathbb{N}_0^2} : \|U\|_l < \infty \right\}
\end{align*}
where $X^l$ is associated to its inner product $(\cdot,\cdot)_l$ and norm $\|\cdot\|_l$ defined as 
\[
    (U,V)_l \bydef \sum_{n \in \mathbb{N}_0^2} \alpha_n u_n v_n^* |l(\tilde{n})|^2 ~~ \text{ and } ~~
    \|U\|_l \bydef \sqrt{(U,U)_l}
\]
for all $U=(u_n)_{n \in \mathbb{N}_0^2},~ V=(v_n)_{n \in \mathbb{N}_0^2} \in X^l.$

Denote by $L : X^l \to \ell^2_{D_2}$ and $G : X^l \to \ell^2_{D_2}$ the Fourier series representation of $\mathbb{L}$ and $\mathbb{G}$ respectively. More specifically, $L$ is represented by an infinite diagonal matrix with coefficients $\left(l(\tilde{n})\right)_{n\in \mathbb{N}_0^2}$ on the diagonal, that is 
\[
(LU)_n = l(\tilde{n})u_n
\]
for all $n \in \mathbb{N}_0^2$ and all $U = (u_n)_{n \in \mathbb{N}_0^2}$.\\ Then, the nonlinear part $G$ is given by $G(U) = \nu_1 U*U + \nu_2 U*U*U$, where $U*V \bydef  \gamma\left(\gamma^\dagger\left(U\right)\gamma^\dagger\left(V\right)\right)$ is defined as the discrete convolution (under the $D_2$-symmetry) for all $U = (u_n)_{n \in \mathbb{N}_0^2}$, and $V = (v_n)_{n \in \mathbb{N}_0^2} \in X^l$. In particular, notice that Young's convolution inequality is applicable and 
\begin{align}\label{eq : youngs inequality}
    \|U*V\|_2 \leq \|U\|_2\|V\|_1
\end{align}
for all $U \in \ell^2_{D_2}$ and all $V \in \ell^1_{D_2}$. 

 We define $F(U) \bydef LU + G(U)$ and introduce  
\begin{equation}\label{eq : F(U)=0 in X^l}
    F(U) =0 ~~ \text{ and } ~~ U \in X^l
\end{equation}
as the periodic equivalent on $\om$ of \eqref{eq : f(u)=0 on Hl}. 

\begin{remark}
    In terms of group theory, $\mathbb{N}^2_0$ is the reduced set associated to the group of symmetries $D_2.$ Moreover, given $n \in \mathbb{N}_0^2$, $\alpha_n$ is the size of the orbit associated to $n.$
\end{remark}

\section{Computer-assisted analysis}\label{sec : computer assisted analysis}

In this section we present our computer-assisted approach to obtain the proofs of existence of localized patterns of \eqref{eq : swift_hohenberg}. More specifically, we first expose the numerical construction of the approximate solution $u_0 \in H^l_{D_2}$, such that supp$(u_0) \subset \overline{\om}$, and its associated Fourier series representation $U_0 \in X^l$. The construction is based on the theory developed in \cite{unbounded_domain_cadiot} (Section 4.1) combined with the numerical analysis derived in \cite{Hill_2023}. Then, following the set-up introduced in \cite{unbounded_domain_cadiot}, we provide the required technical details for the specific case of the Swift-Hohenberg equation. 

Let us first fix $N, N_0 \in \mathbb{N}$ such that $N_0 > N$, where $N_0$ represents the size of our Fourier series approximation and $N$ the size of the operator approximation.  Moreover, given $\mathcal{N} \in \mathbb{N}$, we introduce the projection operators from \cite{unbounded_domain_cadiot} 
 \begin{align}\label{def : projection operators}
    \left(\pi^\mathcal{N}(V)\right)_n  =  \begin{cases}
          v_n,  & n \in I^\mathcal{N} \\
              0, &n \notin I^\mathcal{N}
    \end{cases} ~~ \text{ and } ~~
     \left(\pi_\mathcal{N}(V)\right)_n  =  \begin{cases}
          0,  & n \in I^\mathcal{N} \\
              v_n, &n \notin I^\mathcal{N}
    \end{cases}
 \end{align}
   for all $n \in \mathbb{N}_0^2$ and $V = (v_n)_{n \in  \mathbb{N}_0^2} \in \ell^2_{D_2}$, where $I^\mathcal{N} \bydef \{n \in \mathbb{N}_0^2 : 0 \le n_1, n_2 \leq \mathcal{N}\}$. 
    In particular $U_0$ is chosen such that $U_0 = \pi^{N_0} U_0$, meaning that $U_0$ only has a finite number of non-zero coefficients ($U_0$ may be seen as a vector).

\begin{remark}
The use of two sizes of numerical truncation $N$ and $N_0$ allows us to avoid numerical-memory limitations. More specifically, we represent the numerical operators (such as $B^N$ defined in Section \ref{sec : construction of operator A}) on a truncation of size $N$, which can be different than the truncation of size $N_0$ that we use for sequences (such as $U_0$ for instance). Since operators are more memory consuming than sequences, it  makes sense to choose $N_0 > N$ in practice. As a consequence, we develop the analysis of the bounds $\mathcal{Y}_0,$ $\mathcal{Z}_1,$ and $\mathcal{Z}_2$ with different values for $N_0$ and $N$ such that $N_0 > N$.  
\end{remark}

\subsection{Construction of \boldmath$u_0$}\label{sec : construction of u0}

The analysis developed in \cite{unbounded_domain_cadiot} is based on the construction of a fixed point operator around an approximate solution $u_0 \in H^l_{D_2}$ such that supp$(u_0) \subset \overline{\om}$. This point constitutes one of the main challenge of this work. To answer this problem, we use the approach developed in Section 4 of \cite{unbounded_domain_cadiot}. Specifically,  we need to compute a Fourier series $U_0 \in X^l$ having a function representation on $\om$ with a zero trace of order 4. Finally, we require $U_0$ to be finite-dimensional, that is $U_0 = \pi^{N_0}U_0$. In other terms, $U_0$ has a finite number of non-zero coefficients. This last point is required to perform a computer-assisted proof as $u_0$ will possess a representation on the computer.

Following the set-up of \cite{Hill_2023}, we first study the Swift-Hohenberg equation \eqref{eq : swift_hohenberg} in its radial form on the disk $D_R \subset \mathbb{R}^2$ centered at zero and of radius $R>0$, that is
\begin{equation}\label{def : radial swift hohenberg}
  \left(I_d + \partial_{rr} + \frac{\partial_r}{r} +   \frac{\partial_{\theta \theta}}{r^2}\right)^2v + \mu v + \nu_1 v^2 + \nu_2 v^3 = 0  
\end{equation}
and we look for an approximate solution of the form
\begin{equation}\label{def : form of the approximate solution}
   v(\theta,r) = \sum_{n =0}^{N_1} v_n(r) \cos(s n\theta) 
\end{equation}
where $v_n : (0, R) \to \mathbb{R}$, $N_1 \in \mathbb{N}$ and $s \in \mathbb{N}$ is a parameter determining the symmetry we want our approximate solution to have (e.g. $s=6$ for hexagonal patterns). Plugging the ansatz \eqref{def : form of the approximate solution} into \eqref{def : radial swift hohenberg}, we notice that the functions $(v_n)$ satisfy a system of ODEs given in Equation (2.6) in \cite{Hill_2023}. Using a Galerkin projection, we obtain a system of $N_1+1$ ODEs with $N_1+1$ unknown radial functions. 
Therefore, instead of solving a PDE on a 2D bounded domain, we reduce the problem to solving a system of ODEs on the interval $(0, R)$. As we look for solutions with the $D_2$-symmetry, we impose Neumann's boundary conditions at $0$. 

Then we represent each $v_n$ on a grid defined on $(0, R)$ and solve the system of ODEs using a finite-difference scheme combined with the solver \textit{fsolve} on Matlab. After convergence of \textit{fsolve}, we construct a grid on $D_R$ and obtain an approximate solution at the points of the grid. We use this construction in order to construct a Fourier series representation of the function. In other words, we need to compute the Fourier coefficients
\[
\frac{1}{|\om|}\int_\om v(x)e^{-2\pi i\tilde{n} \cdot x}dx
\]
for all $n \in \mathbb{N}_0^2.$
Supposing that $v$ decreases fast enough to $0$ and that $R$ and $d$ are large enough, then
\[
\frac{1}{|\om|}\int_\om v(x)e^{-2\pi i\tilde{n}\cdot x}dx \approx \frac{1}{|\om|}\int_{D_R} v(x)e^{-2\pi i\tilde{n} \cdot x}dx.
\]
Using a numerical quadrature (trapezoidal rule), we estimate the Fourier series using the values of the function on the disk and obtain a first sequence of Fourier coefficients $\tilde{U}_0$ such that $\tilde{U}_0 = \pi^{N_0} \tilde{U}_0$. To gain precision, we consider $\tilde{U}_0$ as an initial guess for Newton's method applied to the Galerkin projection ${F}^{N_0} : \R^{(N_0+1)^2} \to \R^{(N_0+1)^2}$ defined as 
 \begin{align*}
     \left({F}^{N_0}(V)\right)_n  \bydef 
         \left(F(\iota^{N_0}(V))\right)_n, \qquad n \in I^{N_0}
 \end{align*}
where $\iota^{N_0} : \R^{(N_0+1)^2} \to \pi^{N_0}\ell^2$ is the natural inclusion. Once Newton's method has reached a desired tolerance, we obtain an improved approximation which we still denote by $\tilde{U}_0$. In this process, we choose a number of Fourier coefficients $N_0 \in \mathbb{N}$ big enough in order for the last coefficients of $\tilde{U}_0$ to be of the order of machine precision. 

At this point, $\tilde{U}_0$ represents a $D_2$-symmetric function $\tilde{u}_0$ in $L^2(\Omega_0)$ and by extending $\tilde{u}_0$ by zero outside of $\om$, we obtain a function in $L^2_{D_2}$, but not necessarily in $H^l_{D_2}.$ To fix this lack of regularity we use the approach presented in Section 4.1 of \cite{unbounded_domain_cadiot}. More specifically, we need to ensure that $\tilde{u}_0$ has a null trace of order 4 so that its extension by zero becomes a function in $H^4(\R^2)$. Notice first that because $\tilde{u}_0$ is smooth on $\om$ and has a cosine series representation of the form \eqref{eq:Fourier_expansion}, then its first and third order normal derivatives are automatically zero on $\partial \om.$ Therefore, it remains to ensure that $\tilde{u}_0$ and its second order normal derivative vanish on $\partial \om.$

Let $m \in \{1,2\}$ and define $X^{N_0,m}$ as the following vector space 
\begin{equation}
    X^{N_0,m} \bydef \{(u_n)_{\mathbb{N}_0^m} : u_{n} =0 \text{ for all } |n|_\infty > N_0\}
\end{equation}
where $|n|_\infty = \max_{i \in \{1, \dots, m\}} |n_i|$ for all $n \in \mathbb{N}_0^m.$ In particular, notice that $\tilde{U}_0 \in  X^{N_0,2}$ by construction. Now, let $H^{N_0,m}$ be defined as follows 
\begin{align*}
    H^{N_0,1}&\bydef \left\{u \in L^2\left((-d,d)\right): ~  u(x) = u_0 + {2}\sum_{n=1}^{N_0}u_{n}\cos\left(\frac{\pi n x}{d}\right) \text{ with } (u_n)_{n \in \mathbb{N}_0} \in \ell^2(\mathbb{N}_0)  \right\}\\
    H^{N_0,2}&\bydef \left\{u \in L^2\left(\om\right): ~  u(x) = \sum_{n \in I^{N_0}} u_{n}\alpha_{n}\cos\left(\frac{\pi n_1 x_1}{d}\right) \cos\left(\frac{\pi n_2 x_2}{d}\right) \text{ with } (u_n)_{n \in \mathbb{N}_0^2} \in \ell^2_{D_2} \right\},
\end{align*}
where $\ell^2(\mathbb{N}_0) \bydef \left\{(u_n)_{n \in \mathbb{N}_0}, ~ \sum_{n \in \mathbb{N}_0}|u_n|^2 < \infty\right\}.$ Now, let $\widehat{\mathcal{T}} : H^{N_0,2} \to \left(H^{N_0,1}\right)^4$ be defined as 
\begin{align*}
    \widehat{\mathcal{T}}(u) \bydef \begin{pmatrix}
        u(d,\cdot)\\
        \partial_x^2 u(d,\cdot)\\
        u(\cdot,d)\\
        \partial_y^2 u(\cdot,d)
    \end{pmatrix},
\end{align*}
which is a trace operator of order 4 on $H^{N_0,2}$. Note that, using the periodicity of the elements in $H^{N_0,2}$, it is sufficient to evaluate at $x_1 = d$ or $x_2 =d$ in order to control the whole trace on $\partial \om$.
Then, $\widehat{\mathcal{T}}$ has a representation $\mathcal{T} : X^{N_0,2} \to \left(X^{N_0,1}\right)^4$  given by
\begin{align}\label{eq : formula_trace}
 \mathcal{T}(U) &\bydef \begin{pmatrix}
     \mathcal{T}_{1,0}(U)\\
        \mathcal{T}_{1,2}(U)\\
        \mathcal{T}_{2,0}(U)\\
        \mathcal{T}_{2,2}(U)
 \end{pmatrix}
 \end{align}
where
\begin{equation}
 \left(\mathcal{T}_{1,j}(U)\right)_{n_2} \bydef \displaystyle\sum_{n_1 =0}^{N_0}(-1)^{n_1} \alpha_n u_{n_1,n_2}\big(\frac{\pi n_1}{d}\big)^j   \text{ and } ~  \left(\mathcal{T}_{2,j}(U)\right)_{n_1} \bydef \displaystyle\sum_{n_2 =0}^{N_0} (-1)^{n_2} \alpha_n u_{n_1,n_2}\big(\frac{\pi n_2}{d}\big)^j
\end{equation}
for all $(n_1,n_2) \in \mathbb{N}_0^2$. In particular, if $\mathcal{T}(U) = 0$, then the function representation of $U$ on $\om$ has a null trace of order 4 on $\partial \om$. Moreover, notice that $\mathcal{T}$ has a $4(N_0+1)$ by $(N_0+1)^2$ matrix representation where $\mathcal{T}_{i,j} : \R^{(N_0+1)^2} \to \R^{N_0+1}$. We abuse notation and identify $\mathcal{T}$ by its matrix representation.

Recall that the trace operator is not full rank when defined on a polygon (we refer the interested reader to \cite{trace_polynomials} for a complete study of the trace operator on polygons and polyhedra). In particular, compatibility conditions have to be added in order to ensure the smoothness at the vertices of $\om.$ Indeed, let $u \in H^{N_0,2}$ and denote 
\[
    v_{0} \bydef u(\cdot,d), \quad 
    w_{0} \bydef u(d,\cdot), \quad 
    v_2 \bydef \partial^2_yu(\cdot,d) \quad \text{and} \quad
    w_2 \bydef \partial^2_xu(d,\cdot).
\]
Then, using \cite{trace_polynomials}, the compatibility conditions read
\begin{equation}\label{eq : compatibility condutions}
    v_0(d) = w_0(d), \quad
    v_2(d) = w_2(d), \quad
    v_0''(d) = v_2(d), \quad \text{and} \quad
    w_0''(d) = w_2(d).
\end{equation}
In particular, since $H^{N_0,2} \subset C^\infty(\om)$, \cite{trace_polynomials} provides that $\widehat{\mathcal{T}} : H^{N_0,2}  \to \tilde{H}$ is surjective, where 
\begin{align*}
    \tilde{H} \bydef \left\{(v_0,w_0,v_2,w_2) \in \left(H^{N_0,1}\right)^4 : ~ v_0,w_0,v_2,w_2 \text{ satisfy } \eqref{eq : compatibility condutions}\right\}.
\end{align*}
This implies that $\mathcal{T}$ has a 4-dimensional cokernel.  Since we wish to build a projection into the kernel of $\mathcal{T}$, we need to build a matrix $M : \R^{(N_0+1)^2}\to \R^{4N_0}$ having the same kernel as $\mathcal{T}$ but being full rank. In fact, the matrix $M$ can be obtained numerically. Practically, one can remove $4$ rows from $\mathcal{T}$ and denote $M$ the obtained matrix. To verify that $M$ is indeed full rank, one can compute the singular values of $MM^*$ using interval arithmetic (\cite{Moore_interval_analysis,julia_interval} for instance) and prove that $MM^*$ is invertible. If that is the case, it means that $M$ is full rank and that ${\rm Ker}(M)={\rm Ker}(\mathcal{T})$, where ${\rm Ker}(\cdot)$ denotes the kernel. Assuming we are able to obtain such a matrix $M$, we define $D$ to be the diagonal matrix with entries $\left(\frac{1}{l(\tilde{n})}\right)_{n \in I^{N_0}}$ on the diagonal and  we build a projection  $U_0$ of $\tilde{U}_0$ in the kernel of ${\mathcal{T}}$ defined as
\begin{equation}\label{eq : projection in X^k_0}
    U_0  \bydef \tilde{U}_0 - DM^*(MDM^*)^{-1}M\tilde{U}_0.
\end{equation}
We abuse notation in the above equation as $U_0$ and $\tilde{U}_0$ are seen as vectors in $\R^{(N_0+1)^2}.$ Note that the matrix $D$ allows to build a projection while imposing a decay in, at least, $\frac{1}{l(\tilde{n})}$. One might chose a different matrix to impose more or less decay.
In practice, this construction is made rigorous using interval arithmetic.
Finally, letting $u_0 \bydef \gamma^\dagger(U_0)$, we have that $u_0$ satisfies the $D_2$-symmetry with $u_0 \in H^4(\R^2),$ $supp(u_0) \subset \overline{\om},$. Noticing that $H^l = H^4(\R^2)$ by equivalence of norms (since $l(\xi) = \mathcal{O}(|\xi|^4)$), we obtain that 
$u_0 \in H^l_{D_2,\om}$.

In the rest of this paper, we assume that $u_0 \in H^l_{D_2, \om}$ and $U_0 \in X^l$ satisfy
\begin{align}\label{eq : definition of u0}
    u_0 \bydef \gamma^\dagger(U_0)  \in H^l_{D_2,\om} ~ \text{ and } ~ U_0 = \pi^{N_0} U_0.
\end{align}

\subsection{Newton-Kantorovich approach}\label{sec : newton kantorovich}

In this section we expose our computer-assisted approach, which is based on  Newton-Kantorovich arguments. More specifically, the zeros of  \eqref{eq : f(u)=0 on Hl} are turned into fixed points of some contracting operator $\mathbb{T}$ defined below.  We define 
\begin{equation} \label{def : v0}
    v_0 \bydef 2\nu_1u_0 + 3\nu_2u_0^2 ~ \text{ and } ~ V_0 \bydef \gamma\left(v_0\right) \in X^l,
\end{equation}
where $u_0$ satisfies \eqref{eq : definition of u0}.
In particular, notice that $D\mathbb{G}(u_0)u = {v}_0 u$ for all $u \in L^2$. Recall that $V_0$ is by definition the sequence of Fourier coefficients of $v_0$ on $\om$.

 We want to prove that there exists $r>0$ such that $\mathbb{T} : \overline{B_r(u_0)} \to \overline{B_r(u_0)}$ defined as
\[
\mathbb{T}(u) \bydef u - \mA\mathbb{F}(u)
\]
is well defined and is a contraction, where ${B_r(u_0)} \subset H^l_{D_2}$ is the open ball of radius $r$ centered at $u_0$. In order to determine a possible value for $r>0$ that would provide the contraction and the well-definedness of $\mathbb{T}$, a standard Newton-Kantorovich type theorem is derived. In particular, we want  to build $\mathbb{A} : L^2_{D_2} \to H^l_{D_2}$, $\mathcal{Y}_0,$ $\mathcal{Z}_1,$ and $\mathcal{Z}_2 >0$ in such a way that the hypotheses of the following Theorem~\ref{th: radii polynomial} are satisfied.

\begin{theorem}[\bf Localized patterns]\label{th: radii polynomial}
Let $\mathbb{A} : L^2_{D_2} \to H^l_{D_2}$ be a bounded linear operator. Moreover, let $\mathcal{Y}_0, \mathcal{Z}_1$ be non-negative constants and let $\mathcal{Z}_2 : (0, \infty) \to [0,\infty)$ be a non-negative function such that
  \begin{align}
    \|\mathbb{A}\mathbb{F}(u_0)\|_l & \le \mathcal{Y}_0\label{def : Y0}\\
    \|I_d - \mathbb{A}D\mathbb{F}(u_0)\|_{l} &\le \mathcal{Z}_1\label{def : Z1}\\
    \|\mathbb{A}\left({D}\mathbb{F}(v) - D\mathbb{F}(u_0)\right)\|_l &\le \mathcal{Z}_2(r)r, ~~ \text{for all } v \in \overline{B_r(u_0)} \text{ and all } r>0.\label{def : Z2}
\end{align}  
If there exists $r>0$ such that
\begin{equation}\label{condition radii polynomial}
    \frac{1}{2}\mathcal{Z}_2(r)r^2 - (1-\mathcal{Z}_1)r + \mathcal{Y}_0 <0 \text{ and } \mathcal{Z}_1 + \mathcal{Z}_2(r)r < 1,
 \end{equation}
then there exists a unique $\tilde{u} \in \overline{B_r(u_0)} \subset H^l_{D_2}$ such that $\mathbb{F}(\tilde{u})=0$. 
\end{theorem}

\begin{proof}
First, using Theorem 3.5 in \cite{unbounded_domain_cadiot}, we obtain that both $\mathbb{A} : L^2_{D_2} \to H^l_{D_2}$ and $D\mathbb{F}(u_0) : H^l_{D_2} \to L^2_{D_2}$ have a bounded inverse if $\mathcal{Z}_1<1$ (which is the case if \eqref{condition radii polynomial} is satisfied). 
The rest of the proof can be found in Theorem 2.15 in  \cite{van2021spontaneous}.
\end{proof}

In practice,  $\mathbb{F}(u_0)$ is supposedly small in norm if $u_0$ is a good approximation of a solution. The bound $\mathcal{Y}_0$ controls the accuracy of this approximation. Since the construction of $u_0 \in H^l_{D_2,\om}$ has already been described in Section \ref{sec : construction of u0}, it remains to compute the operator $\mathbb{A} : L^2_{D_2} \to H^l_{D_2}$  approximating the inverse of $D\mathbb{F}(u_0)$. We provide a detailed presentation of its construction in Section \ref{sec : construction of operator A} below. Once $u_0$ and $\mathbb{A}$ are constructed, we need to determine $\mathcal{Y}_0, \mathcal{Z}_1$ and $\mathcal{Z}_2$.
Section \ref{sec : computation of the bounds} focuses on building these quantities in order to make use of Theorem \ref{th: radii polynomial}.

\subsection{The operator \boldmath$\mathbb{A}$\unboldmath}\label{sec : construction of operator A}

In this section, we focus our attention on the construction of $\mathbb{A} : L^2_{D_2} \to H^l_{D_2}$. Specifically, we recall the construction exposed in Section 3 from \cite{unbounded_domain_cadiot}.

We begin by constructing numerically, using  floating point arithmetic, an approximate inverse for $\pi^N DF(U_0)L^{-1}\pi^N$ that we denote $B^N$. By construction, $B^N$ is a matrix that we naturally extend to a bounded linear operator on $\ell^2_{D_2}$ such that $B^N = \pi^N B^N \pi^N$. Using this matrix, we define the bounded linear operator $\mathbb{B} : L^2_{D_2} \to L^2_{D_2}$ as
\begin{equation}\label{def : the operator B}
    \mathbb{B} \bydef \out + \Gamma^\dagger(\pi_N + B^N).
\end{equation}
Using the operator $\mathbb{B}$, we can finally define the operator $\mathbb{A} : L^2_{D_2} \to H^l_{D_2}$ as 
\begin{equation}\label{def : the operator A}
    \mathbb{A} \bydef \mathbb{L}^{-1} \mathbb{B}.
\end{equation}
We refer the interested reader to the Section 3 of \cite{unbounded_domain_cadiot} for the justification of such a construction. 
In particular, using the fact that  $\mathbb{L}$ is an isometric isomorphism between $H^l_{D_2}$ and $L^2_{D_2}$ (cf. \eqref{def : inner product and norm on Hl}), we obtain that $\mathbb{A} : L^2_{D_2} \to H^l_{D_2}$ is well defined as a bounded linear operator. Moreover, $\mathbb{A}$ is completely determined by $B^N$, which is chosen numerically. This implies that the computations associated to $\mathbb{A}$ can be conducted rigorously using the arithmetic on intervals. In particular, using Lemma 3.4 in \cite{unbounded_domain_cadiot}, we have 
\begin{equation}\label{eq : equality norm A and BN}
\|\mathbb{A}\|_{2,l} = \|\mathbb{B}\|_2 = \max\left\{1,\|B^N\|_2\right\}.
\end{equation}
In practice, if \eqref{condition radii polynomial} holds, then the bound $\mathcal{Z}_1$ (satisfying \eqref{def : Z1} in Theorem \ref{th: radii polynomial}) satisfies $\mathcal{Z}_1 <1$, and hence $\|I_d - \mathbb{A}D\mathbb{F}(u_0)\|_{l}<1$.
From there, Theorem 3.5 in \cite{unbounded_domain_cadiot} provides that both $\mathbb{A} : L^2_{D_2} \to H^l_{D_2}$ and $D\mathbb{F}(u_0) : H^l_{D_2} \to L^2_{D_2}$ have a bounded inverse, which, in such a case, justifies that $\mathbb{A}$ can be considered as an approximate inverse of $D\mathbb{F}(u_0)$. Having  determined the operator $\mathbb{A}$, the remaining task consists of  presenting the computation of the bounds $\mathcal{Y}_0,$ $\mathcal{Z}_1,$ and $\mathcal{Z}_2$, which we now do.

\subsection{Computation of the bounds}\label{sec : computation of the bounds}

Throughout this section, we use the following notations. Given $u \in L^\infty$ and $U \in \ell^1$, we denote by 
\begin{equation}\label{eq : multiplication operator L2}
    \mathbb{M}_u \colon L^2 \to L^2: v \mapsto \mathbb{M}_uv \bydef uv \quad \text{ and } \quad M_U \colon \ell^2_{D_2} \to \ell^2_{D_2}:
                    V \mapsto M_{U}V \bydef U*V
\end{equation}
 the linear multiplication operator associated to $u$ and the linear discrete convolution operator associated to $U$, respectively.


We begin  by determining the bound $\mathcal{Y}_0$ satisfying \eqref{def : Y0}, which can be computed explicitly using 
 Lemma 4.11 in \cite{unbounded_domain_cadiot}. We recall the aforementioned lemma for convenience.  
\begin{lemma}\label{lem : bound Y_0}
Let $\mathcal{Y}_0 >0$ be such that 
\begin{equation}\label{def : upper bound Y0}
     |\om|^{\frac{1}{2}}\left(\left\|B^NF(U_0)\right\|_{2}^2 + \left\|(\pi^{N_0}-\pi^N)LU_0 + (\pi^{3N_0}-\pi^N) G(U_0)\right\|_{2}^2 \right)^{\frac{1}{2}} \le \mathcal{Y}_0.
\end{equation}
Then \eqref{def : Y0} holds, that is $\|\mathbb{A}\mathbb{F}(u_0)\|_l \leq \mathcal{Y}_0.$
\end{lemma}

Then, we show that the bound $\mathcal{Z}_2$ can be obtained explicitly thanks to computations on finite-dimensional objects.

\begin{lemma}\label{lem : bound Z_2}
Let $U_0 \bydef (u_n)_{n \in \mathbb{N}_0^2}$ be the Fourier series representation of $u_0$ and consider $\kappa >0$ such that $\kappa \geq \|\frac{1}{l}\|_2$. Moreover, let $r>0$ and let $\mathcal{Z}_2(r) >0$ be such that 
\begin{equation}\label{def : Z2 in lemma}
    3\nu_2\frac{\kappa^2}{\mu} \max \left\{1,\|B^N\|_{2}\right\} r + \frac{\kappa}{\mu} \max\left\{ 2|\nu_1|, \left(\|B^NM_V^2(B^N)^*\|_{2}+\|V\|_1^2\right)^{\frac{1}{2}}\right\} \le \mathcal{Z}_2(r),
\end{equation}
where $V \bydef (v_n)_{n \in \mathbb{N}_0^2} = (2\nu_1\delta_{n}+6\nu_2u_n)_{n \in \mathbb{N}_0^2}$ and $\delta_n$ is the Kronecker symbol. Moreover, $M_V$ is the discrete convolution operator associated to $V$ as defined in \eqref{eq : multiplication operator L2}.
Then $\mathcal{Z}_2(r)$ satisfies \eqref{def : Z2}.
\end{lemma}

\begin{proof}
Let $v \in \overline{B_r(u_0)}$. Since ${D}\mathbb{F}(v) - D\mathbb{F}(u_0) = {D}\mathbb{G}(v) - D\mathbb{G}(u_0)$ and $\|u\|_l = \|\mathbb{L}u\|_2$ for all $u \in H^l$, 
\begin{align*}
    \|\mathbb{A}\left({D}\mathbb{F}(v) - D\mathbb{F}(u_0)\right)\|_l &= \|\mathbb{L}\mathbb{A}\left({D}\mathbb{F}(v) - D\mathbb{F}(u_0)\right)\|_{l,2}\\
   & = \|\mathbb{B}\left({D}\mathbb{G}(v) - D\mathbb{G}(u_0)\right)\|_{l,2}.
\end{align*}
Now let  $w \bydef v-u_0 \in \overline{B_r(0)} \subset H^l_{D_2}$ (in particular $\|w\|_l \leq r$). Then, 
\[
D\mathbb{G}(v) - D\mathbb{G}(u_0) = 2\nu_1(\mathbb{M}_{u_0} + \mathbb{M}_w) + 3\nu_2(\mathbb{M}_{u_0}+\mathbb{M}_w)^2 - 2\nu_1\mathbb{M}_{u_0} -3\nu_2\mathbb{M}_{u_0}^2 = 2\nu_1\mathbb{M}_w + 6\nu_2\mathbb{M}_{u_0}\mathbb{M}_w + 3\nu_2\mathbb{M}_w^2,
\]
where $\mathbb{M}_w$ is the multiplication operator associated to $w$, as defined in \eqref{eq : multiplication operator L2}.
In particular, we obtain
\[
\|\mathbb{B}({D}\mathbb{G}(v) - D\mathbb{G}(u_0))\|_{l,2} \leq 3|\nu_2|\|\mathbb{B}\|_{2}\|\mathbb{M}_w^2\|_{l,2} + \|\mathbb{B}(2\nu_1 I_d + 6\nu_2 \mathbb{M}_{u_0})\|_{2}\|\mathbb{M}_w\|_{l,2}.
\]
Moreover, using Lemma \ref{lem : banach algebra}, we get
\begin{equation}
    \|\mathbb{w}\|_{l,2} = \sup_{\|h\|_l=1} \|w h \|_2 
    \leq \frac{\kappa}{\mu}\|w\|_l  \leq \frac{\kappa}{\mu} r.
\end{equation}
Similarly,
\begin{equation}
    \|\mathbb{w}^2\|_{l,2} \leq \frac{\kappa^2}{\mu}\|w\|_l^2 \leq \frac{\kappa^2}{\mu}r^2.
\end{equation}
Furthermore, using \eqref{eq : equality norm A and BN}, we have
\[
\|\mathbb{B}\|_2 = \max\left\{1, \|B^N\|_2\right\}.
\]
Then, notice that
\[
\|\mathbb{B}(2\nu_1 I_d + 6\nu_2 \mathbb{M}_{u_0})u\|_{2}^2 = \|2\nu_1 \mathbb{1}_{\mathbb{R}^2\setminus\Omega_0}u\|^2_2 + \|\Gamma^\dagger(\pi_N + B^N)(2\nu_1\cha +6\nu_2\mathbb{M}_{u_0})u\|_2^2
\]
for all $u \in L^2_{D_2}$ as $\mathbb{B} = \out + \Gamma^\dagger(\pi_N + B^N)$. Then, because $u_0 = \gamma^\dagger(U_0)$ we have $\mathbb{M}_{u_0} = \Gamma^\dagger(M_{U_0})$. This implies that
\[
    \Gamma^\dagger(\pi_N + B^N)(2\nu_1\cha +6\nu_2\mathbb{M}_{u_0}) = \Gamma^\dagger((\pi_N + B^N)(2\nu_1I_d +6\nu_2M_{U_0})) 
\]
and therefore
\[
\|\mathbb{B}(2\nu_1 I_d + 6\nu_2 \mathbb{M}_{u_0})u\|_{2} \leq \max\left\{2|\nu_1|, \|(\pi_N + B^N)(2\nu_1I_d + 6\nu_2M_{U_0})\|_2\right\}.
\]
At this point we focus our attention on $\|(\pi_N + B^N)(2\nu_1I_d +6\nu_2M_{U_0})\|_{2}.$
Notice that the operator $(2\nu_1 I_d + 6\nu_2M_{U_0})$ can be seen as a discrete convolution operator associated to $V = (v_n)_{n \in \mathbb{N}_0^2} = (2\nu_1\delta_{n}+6\nu_2u_n)_{n \in \mathbb{N}_0^2}$. Therefore, 
\[
    \left\|(\pi_N + B^N)M_{V}\right\|_{2}^2 \leq \left\|B^NM_{V}\right\|_{2}^2 + \left\|\pi_NM_{V}\right\|_{2}^2 \leq  \left\|B^NM_{V}\right\|_{2}^2 + \|V\|_1^2
\]
where we used  \eqref{eq : youngs inequality}. Therefore, using the properties of the adjoint, we obtain that $\|B^NM_{V}\|_{2}^2  = \|B^NM_{V}^2(B^N)^*\|_2$ as $M_{V} = M_{V}^*$ and consequently
\[
   \left\|(\pi_N + B^N)M_{V}\right\|_{2} \leq \left( \|B^NM_{V}^2(B^N)^*\|_2+\|V\|_1^2\right)^{\frac{1}{2}}. 
\]
This concludes the proof.
\end{proof}

\begin{remark}
    Note that $B^NM_{V}^2(B^N)^*$ can be seen as a matrix and consequently, in the previous lemma, the quantity $\|B^NM_{V}^2(B^N)^*\|_2$ can be obtained rigorously on the computer (cf. \cite{julia_cadiot}).
\end{remark}

Recall that $D\mathbb{G}(u_0) = \mathbb{M}_{v_0}$, where $v_0$ is defined in \eqref{def : v0}. In particular, $v_0$ has a Fourier series representation $V_0 \in X^l$. Then, recall that we fixed $N$ and $N_0$ to be the size of our Fourier series truncations for operators and sequences respectively. In particular, we have that $U_0 = \pi^{N_0} U_0$ by construction and therefore $V_0 = \pi^{2N_0}V_0$. Now, let
\begin{align}\label{eq : v0N}
    v_0^N \bydef \gamma^\dagger(V_0^N) \in L^\infty(\R^2) \cap L^2_{D_2}~~ \text{ where } ~~ V_0^N \bydef \pi^{2N}V_0.
\end{align}
Denote by $\mathbb{M}_{v_0}^N$ and $M_{V_0}^N$ the operators built from $v_0^N$ and $V_0^N$ via \eqref{eq : multiplication operator L2}. The next result provides an explicit lower bound for $\mathcal{Z}_1$.

\begin{lemma}\label{lem : computation of Z1}
Let $\mathcal{Z}_{u}$ and $Z_1$ be bounds satisfying
\begin{equation} \label{def : Z1 periodic and Zu}
\begin{aligned}
\|\left(\mathbb{L}^{-1}-\Gamma^\dagger(L^{-1})\right) \mathbb{M}_{v_0}^N  \|_2 & \le \mathcal{Z}_{u} \\
\|I_d - (B^N + \pi_N )(I_d + M_{V_0}^N L^{-1})\|_{2} & \le Z_1.
\end{aligned}
\end{equation}
Then defining $\mathcal{Z}_1$ as 
\begin{equation}\label{def : definition of Z1 theoretical}
   \mathcal{Z}_1 \bydef Z_1 + \max\{1, \|B^N\|_{2}\} \left( \mathcal{Z}_{u} +  \frac{1}{\mu}\|V_0-V_0^N\|_1 \right),
\end{equation}
 we get that \eqref{def : Z1} holds, that is $ \|I_d - \mathbb{A}D\mathbb{F}(u_0)\|_l \leq \mathcal{Z}_1.$
\end{lemma}

\begin{proof}
First, using that $\|u\|_l = \|\mathbb{L}u\|_2$,
\begin{equation}
    \|I_d - \mathbb{A}D\mathbb{F}(u_0)\|_l 
    =
    \|I_d - \mathbb{B}D\mathbb{F}(u_0)\mathbb{L}^{-1}\|_2
    = \|I_d - \mathbb{B}- \mathbb{B}D\mathbb{G}(u_0)\mathbb{L}^{-1}\|_2.
\end{equation}
Then, using triangle inequality,
\begin{equation}\label{eq : proof lemma 3.4 eq 1}
    \|I_d - \mathbb{B}- \mathbb{B}D\mathbb{G}(u_0)\mathbb{L}^{-1}\|_2 \leq    \|I_d - \mathbb{B}- \mathbb{B}\mathbb{M}_{v_0}^N\mathbb{L}^{-1}\|_2 +   \| \mathbb{B}\left(\mathbb{M}_{v_0}^N - D\mathbb{G}(u_0)\right)\mathbb{L}^{-1}\|_2.
\end{equation}
Now, the first term of \eqref{eq : proof lemma 3.4 eq 1}, namely $ \|I_d - \mathbb{B}- \mathbb{B}\mathbb{M}_{v_0}^N\mathbb{L}^{-1}\|_2$, can be bounded using the analysis developed in \cite{unbounded_domain_cadiot}. Specifically, using the proof of Theorem 3.5 from \cite{unbounded_domain_cadiot}, we get
\begin{align}\label{eq : proof lemma 3.4 step 2}
\nonumber
    &\|I_d - \mathbb{B}- \mathbb{B}\mathbb{M}_{v_0}^N\mathbb{L}^{-1}\|_2 \\ \nonumber
    & \leq \|I_d - \mathbb{B}- \mathbb{B}\mathbb{M}_{v_0}^N\Gamma(L^{-1})\|_2 + \|\mathbb{B}\mathbb{M}_{v_0}^N\left(\Gamma^\dagger(L^{-1}) - \mathbb{L}^{-1}\right)\|_2\\ \nonumber
    &\leq \|I_d - (\pi_N + B^N)(I_d + M_{V_0}^N L^{-1})\|_{2} +  \max\left\{1, \|B^N\|_{2}\right\} \|\left(\Gamma^\dagger(L^{-1})^*-(\mathbb{L}^{-1})^*\right)(\mathbb{M}_{v_0}^N)^*\|_2\\ 
    &\leq  Z_1 +\max\left\{1, \|B^N\|_{2}\right\} \|\left(\Gamma^\dagger(L^{-1})^*-(\mathbb{L}^{-1})^*\right)(\mathbb{M}_{v_0}^N)^*\|_2,
\end{align}
where we used \eqref{eq : parseval's identity} on the second inequality.
To bound the second term of \eqref{eq : proof lemma 3.4 eq 1}, we use that $\|\mathbb{L}^{-1}\|_2 \leq \frac{1}{\mu}$ as $l(\xi) \geq \mu$ for all $\xi \in \R^2$, and get
\begin{align}\label{eq : proof lemma 3.4 step 3}
     \| \mathbb{B}\left(\mathbb{M}_{v_0}^N - D\mathbb{G}(u_0)\right)\mathbb{L}^{-1}\|_2 &\leq \frac{\| \mathbb{B}\|_2}{\mu} \|\mathbb{M}_{v_0}^N - D\mathbb{G}(u_0)\|_2
     \leq \frac{\max\{1, \|B^N\|_{2}\}}{\mu}\|V_0^N-V_0\|_1
\end{align}
where we used \eqref{eq : youngs inequality} for the last step. Combining \eqref{eq : proof lemma 3.4 step 2} and \eqref{eq : proof lemma 3.4 step 3}, 
\begin{equation}
    \|I_d - \mathbb{A}D\mathbb{F}(u_0)\|_l 
    \leq Z_1  + \max\{1, \|B^N\|_{2}\} 
    \left( \|\left(\Gamma^\dagger(L^{-1})^*-(\mathbb{L}^{-1})^*\right)(\mathbb{M}_{v_0}^N)^*\|_2 +\frac{1}{\mu} \|V_0-V_0^N\|_1 \right).
\end{equation}
Notice that $(\mathbb{L}^{-1})^* = \mathbb{L}^{-1}$ as $l$ is real-valued. The same argument applies to $\Gamma^\dagger\left(L^{-1}\right)$ and we get that $\Gamma^\dagger\left(L^{-1}\right) = \Gamma^\dagger\left(L^{-1}\right)^*$. Moreover, $(\mathbb{M}_{v_0}^N)^* = \mathbb{M}_{v_0}^N$ since $v_0$ is real-valued. This concludes the proof.
\end{proof}

\begin{remark}
Note that the bound obtained in \eqref{def : definition of Z1 theoretical} is slightly less sharp than the one presented in Theorem 3.5 in \cite{unbounded_domain_cadiot}. Indeed, in the previous lemma, we applied the triangle inequality in \eqref{eq : proof lemma 3.4 eq 1} and chose to work with $v_0^N$ instead of $v_0$, which yields an extra error term, namely $\frac{\max\{1, \|B^N\|_{2}\}}{\mu}\|V_0-V_0^N\|_1$. In practice, this manipulation can be useful for the CAP to be efficient in terms of computer memory.  Since the number of non-zero coefficients in $V_0^N$ is smaller than the one of $V_0$, using the operator $M_{V_0}^N$ is in fact more memory efficient than using $M_{V_0}$. Moreover, if $V_0$ has a fast decay, then $\|V_0-V_0^N\|_1$ will be negligible and the quality of the CAP will not be affected. 
\end{remark}

 \begin{remark}
     In the previous lemma, we choose to work with the quantity $\|\left(\mathbb{L}^{-1}-\Gamma^\dagger(L^{-1})\right) \mathbb{M}_{v_0}^N  \|_2$ instead of $\|\mathbb{M}_{v_0}^N \left(\mathbb{L}^{-1}-\Gamma^\dagger(L^{-1})\right)  \|_2$ in order to be able to apply the analysis derived in \cite{unbounded_domain_cadiot}. Indeed, for technical reasons which can be found in \cite{unbounded_domain_cadiot}, one can notice that it is easier to work with the former rather than the later.
 \end{remark}

In order to obtain an explicit expression for $\mathcal{Z}_1$, we need to compute an upper bound for $Z_1$ and $\mathcal{Z}_{u}$  defined in \eqref{def : Z1 periodic and Zu}.
$\mathcal{Z}_{u}$ comes from the unboundedness part of the problem. More particularly, it depends on how  good an approximation $\Gamma^\dagger(L^{-1})$ is for $\mathbb{L}^{-1}$. We will see in Lemma \ref{lem : lemma Zu}  that $\mathcal{Z}_{u}$ is exponentially decaying with the size of $\om$, which is itself given by $d.$

Note that $Z_1$ is the usual term one has to compute during the proof of a periodic solution using a standard Newton-Kantorovich approach (see \cite{van2021spontaneous} for instance). Lemma \ref{lem : usual term periodic Z1} provides the details for such an analysis. In particular, it is fully determined by vector and matrix norm computations.

\begin{lemma}\label{lem : usual term periodic Z1}
Let $Z^N_1$ and $Z_1$ be such that
\begin{equation}\label{def : upper bound periodic Z1}
\begin{aligned}
     \left(\|\pi^N - B^N(I_d + M_{V_0}^N L^{-1})\pi^{3N}\|_2^2 + \|(\pi^{3N}-\pi^N) M_{V_0}^N L^{-1}\pi^N\|_2^2\right)^{\frac{1}{2}} & \le Z_1^N\\
    \left((Z_1^N)^2 + \|V_0^N\|_1^2\max_{n \in \mathbb{N}_0^2\setminus I^N}\frac{1}{|l(\tilde{n})|^2}\right)^{\frac{1}{2}}  &\leq Z_1.
\end{aligned}
\end{equation}
Then we have $\left\|I_d - (B^N + \pi_N)\left(I_d + M_{V_0}^NL^{-1}\right)\right\|_l \leq Z_1$.
   \end{lemma}
   \begin{proof}
       The proof can be found in \cite{unbounded_domain_cadiot}.
   \end{proof}

To compute the bound $\mathcal{Z}_1$, it remains to compute $\mathcal{Z}_u$. This bound is the one requiring the most analysis. We present its computation in the next section.   

\subsection{Computation of \boldmath$\mathcal{Z}_{u}$\unboldmath}\label{ssec : computation Zu}

Denote 
    \begin{equation}\label{def : definition of f0}
        f_0 \bydef \mathcal{F}^{-1}\left(\frac{1}{l}\right).
    \end{equation}
In particular, provided we are able to compute explicitly $C_0$ and $a>0$ such that
\begin{align}\label{def : definition of C0 and a}
   |f_0(x)| \leq C_0e^{-a|x|_1} 
\end{align}
for all $x \in \R^2$, where $|x|_1 = |x_1| + |x_2|$, then Theorem 3.7 in \cite{unbounded_domain_cadiot} provides an explicit upper bound on $\left(\|\out \left(f_0^2* (v_0^N)^2\right)\|_1\right)^{\frac{1}{2}}$ depending on $a, C_0, d$ and the Fourier coefficients of $v_0^N$. This quantity then allows to compute an explicit value for the bound $\mathcal{Z}_u$ (see \eqref{eq:proof_E_i} and \eqref{eq : identity Linv and f0}).

Consequently, we now focus our attention on computing explicitly $C_0$ and $a>0$ satisfying \eqref{def : definition of C0 and a}. We begin by computing $a$ and prove the existence of $C_0$. The explicit computation of $C_0$ is then addressed in Section \ref{sec : computation of C_0}.

\begin{prop}\label{prop : computation_of_f}
Fix $\mu>0$. Let $a>0$ and $b \in \mathbb{C}$ be defined as 
\begin{equation}\label{def : definition of a and b}
    a \bydef \frac{\sqrt{-1+\sqrt{1+\mu}}}{2} ~~ \text{ and } ~~ b \bydef \sqrt{2}a - i \frac{\sqrt{\mu}}{2\sqrt{2}a}.
\end{equation}
In particular, $b = (1+\mu)^{\frac{1}{4}}e^{i\theta}$,  where $\theta \in (-\frac{\pi}{2},0)$ is defined as 
\begin{equation}\label{def : definition of theta}
    \theta \bydef -\arctan\left(\frac{\sqrt{\mu}}{4a^2}\right).
\end{equation}
%
Then, defining  $C_0$ and $\beta$ as 
\begin{equation}\label{def : definition C0 as a sup}
   C_0 \bydef \sup_{r \in [0,\beta]}  e^{\sqrt{2}ar}\left|\frac{1}{2i\sqrt{\mu}}\bigg(K_0(br) - K_0(\overline{b}r) \bigg)\right| < \infty~~ \text{ and } ~~ \beta \bydef \frac{\pi}{2\sqrt{2}a\theta^2},
\end{equation}
we obtain that $C_0$ satisfies \eqref{def : definition of C0 and a} and 
\[
C_0 \geq f_0(0) =  \frac{-\theta}{\sqrt{\mu}}.
\]

\end{prop}

\begin{proof}
We first notice that the function $\xi \mapsto \frac{1+|\xi|}{l(\xi)}$ is in $L^1$, so $f_0$ is continuously differentiable  on $\mathbb{R}^2$. 
Letting $\xi \in \R^2$ and $r \bydef |\xi|$, we have
\begin{align*}
   \frac{1}{l(r)} = \frac{1}{\mu+  (1-r^2)^2} = \frac{1}{2i\sqrt{\mu}}\bigg(\frac{1}{r^2-1-i\sqrt{\mu}} - \frac{1}{r^2 -1 + i\sqrt{\mu}} \bigg).
\end{align*}
Now using \cite{poularikas2018transforms} (Section 9.3), we know that the Fourier transform of a radially symmetric function equals its Hankel transform of order zero in polar coordinates. Therefore, using the Hankel transform tables in  \cite{poularikas2018transforms} (Section 9.11)  and noticing that $b^2 = -1-i\sqrt{\mu}$ and $(\overline{b})^2 = -1 + i\sqrt{\mu}$ we obtain that 
\begin{align}\label{eq : equality for f0}
  f_0(x) =  \mathcal{F}^{-1}\left( \frac{1}{l} \right)(x) = \frac{1}{2i\sqrt{\mu}}\bigg(K_0(b|x|) - K_0(\overline{b}|x|) \bigg).
\end{align}

We know from \cite{watson_bessel} (Chapter 6.15) that if Re$(z) > 0$, then
\begin{align*}
    K_0(z) \bydef \displaystyle\int_0^\infty e^{-z \cosh(t)}dt.
\end{align*}
Since $b = \sqrt{2}a - i\frac{\sqrt{\mu}}{2\sqrt{2}a}$, then Re$(rb)$ = Re$(r\overline{b})$ = $\sqrt{2}ra >0$ for all $r>0$. Therefore we obtain that
\begin{align*}
    |K_0(br) - K_0(\overline{b}r)|  &= \bigg|\displaystyle\int_0^\infty e^{-\sqrt{2}a r  \cosh(t)} \left( e^{i\sqrt{\mu}\frac{\cosh(t)}{2\sqrt{2}a}} - e^{-i\sqrt{\mu}\frac{\cosh(t)}{2\sqrt{2}a}} \right) dt\bigg|\\
     &= \bigg|2i\displaystyle\int_0^\infty e^{-\sqrt{2}a r  \cosh(t)} \sin\left(\frac{\sqrt{\mu}\cosh(t)}{2\sqrt{2}a}\right) dt\bigg|\\
    &\leq 2\displaystyle\int_0^\infty e^{-\sqrt{2}ar \cosh(t)} dt\\
    &= 2 K_0(\sqrt{2}ar).
\end{align*}
But then using \cite{gaunt2017inequalities}, we know that 
\begin{equation}\label{ineq : assymptotics}
    K_0(\sqrt{2}ar) \leq \sqrt{\frac{\pi}{2\sqrt{2}ar}}e^{-\sqrt{2}ar}
\end{equation}
for all $r >0$. Finally, using the smoothness of $f_0$ and the fact that $|x| \geq \frac{|x_1|+|x_2|}{\sqrt{2}} = \frac{|x|_1}{\sqrt{2}}$, we obtain that  $C_0< \infty$ and satisfies \eqref{def : definition of C0 and a}.

Now, using \cite{watson_bessel}, we have 
\[
K_0(z) = -\ln\left(\frac{z}{2}\right) - C_{Euler} + \mathcal{O}(z)
\]
for $|z|$ small, where $C_{Euler}$ is Euler-Mascheroni's constant. Moreover, since $a>0$ by definition, then the principal value of the argument of $b$  is given by $\theta <0$ in \eqref{def : definition of theta}. In particular, this implies that 
\begin{align*}
    f_0(x) = \frac{1}{2i\sqrt{\mu}}\left(K_0(b|x|) - K_0(\overline{b}|x|)\right) =  \frac{1}{2i\sqrt{\mu}}\left(K_0(b|x|) - K_0(\overline{b}|x|)\right) = \frac{-\ln(\frac{|x|b}{2}) + \ln(\frac{|x|\overline{b}}{2})} {2i\sqrt{\mu}} + \mathcal{O}(|x|)
\end{align*}
for $|x|$ small. But using that $\ln(\frac{|x|b}{2}) = \ln(\frac{|xb|}{2}) + i \theta$ and $\ln(\frac{|x|\overline{b}}{2}) = \ln(\frac{|xb|}{2}) - i \theta$ for all $|x| >0$, we obtain that
\begin{align*}
    f_0(0) = \frac{-\theta}{\sqrt{\mu}} >0.
\end{align*}
By definition of $C_0$, it is clear that $C_0 \geq \frac{|\theta|}{\sqrt{\mu}}$. Moreover, using \eqref{ineq : assymptotics}, we have that 
\begin{align*}
    e^{\sqrt{2}ar}\left|\frac{1}{2i\sqrt{\mu}}\bigg(K_0(br) - K_0(\overline{b}r) \bigg)\right| \leq \frac{1}{\sqrt{\mu}}\sqrt{\frac{\pi}{2\sqrt{2}ar}}
\end{align*}
for all $r>0.$ Consequently, given $r\geq \beta$, then the proof of the proposition follows from observing that 
\[
    e^{\sqrt{2}ar}\left|\frac{1}{2i\sqrt{\mu}}\bigg(K_0(br) - K_0(\overline{b}r) \bigg)\right| \leq \frac{1}{\sqrt{\mu}}\sqrt{\frac{\pi}{2\sqrt{2}a \beta}} = \frac{-\theta}{\sqrt{\mu}} \leq C_0. \qedhere
\]
\end{proof}

\begin{prop}\label{prop : inequality depending on constants Ci}
Let $C_0$ and $a$ be defined as in Proposition \ref{prop : computation_of_f} and let $V_0^N$ be the Fourier coefficients of $v_0^N$ on $\om$ (\eqref{eq : v0N}). 
Moreover, let  ${E_1}, {E_{1,2}}$ and ${E_2}$ be sequences in $\ell^2_{D_2}$ defined by
\begin{equation} \label{def : definition of the sequences Ei}
\begin{aligned}
    {E_1} &\bydef \gamma(\cha(x)\cosh(2ax_1))\\
    {E_{1,2}} &\bydef \gamma(\cha(x)\cosh(2ax_1)\cosh(2ax_2)) 
    \\
    {E_2} &\bydef \gamma(\cha(x)\cosh(2ax_2)),
\end{aligned}
\end{equation}
where we abuse notation in the above definitions and consider the argument of $\gamma$ to be a function in $L^2_{D_2}.$
    Moreover, let $C_1(d), C_{12}(d)$ and $ C_2(d)$ be non-negative constants defined by  
    
\begin{equation}\label{def : definition of the constants Ci}
    {\small
    \begin{aligned}
 C_1(d) &\bydef 4\left(\frac{2ad+1+e^{-2ad}}{a^2} + e^{-2ad}\left(4d + \frac{e^{-2ad}}{a}\right) + \left(\frac{1+e^{-2ad}}{a} + 2d
    \right)\frac{2e^{-1}+1}{a(1-e^{-ad})}\right) \\
     & \qquad + ~   \frac{4(2e^{-1}+1)^2}{a^2(1-e^{-ad})^2} +\frac{2}{a}\left( \frac{1+e^{-2ad}}{a} + 2d + e^{-2ad}(4d+\frac{e^{-2ad}}{a})+ \frac{2e^{-1}+1}{a(1-e^{-ad})}\right)\\
     C_{12}(d) &\bydef  8\left(2d + \frac{1}{2a}\right) \left(2d + \frac{1+e^{-2ad}}{2a} + \left(2d + \frac{3+e^{-2ad}}{2a}\right)\frac{1}{1-e^{-ad}} +  \frac{4e^{-1} + 1 +e^{-2ad}}{2a(1-e^{-ad})^2}\right)\\
      C_2(d) &\bydef \frac{2}{a} \left[\frac{1+e^{-2ad}}{a} + 2d + e^{-2ad}(4d + \frac{e^{-2ad}}{a}) + \frac{(2e^{-1}+e^{-2ad})}{a(1-e^{-ad})} \right].
    \end{aligned}
    }
     \end{equation}
   Now, let $u \in L^2_{D_2}$ such that $\|u\|_2=1$ and define $v \bydef v^N_0 u$. Then
   \begin{align}
    &\hspace{-1.5cm}  \sum_{n \in \mathbb{N}_0^2, n \neq 0} \alpha_n  \int_{\mathbb{R}^2\setminus (\om \cup (\om +2dn))} \int_\om \int_\om e^{-a|y-x|_1}e^{-a|y-2dn-z|_1}|v(x)v(z)|dxdzdy\nonumber 
    \\
    \leq & ~  e^{-4ad}|\om| \left(V_0^N,V_0^N*\left[C_1(d)E_1+ C_{12}(d)E_{1,2} + C_{2}(d)E_2\right]\right)_2.
    \label{eq : proven in the appendix}
\end{align}
\end{prop}

\begin{proof}
The proof is presented in Appendix~\ref{appendix}.
\end{proof}

Given $C_0$ and $a$ as defined in Proposition \ref{prop : computation_of_f} satisfying \eqref{def : definition of C0 and a}, we can now compute an upper bound for $\mathcal{Z}_{u}$ (defined in \eqref{def : Z1 periodic and Zu}) in terms of $C_0$ and $a$.

\begin{lemma}\label{lem : lemma Zu}
Let $C_0$ and $a$ be defined as in Proposition \ref{prop : computation_of_f} and let $V_0^N$ be the Fourier coefficients of $v_0^N$ on $\om$ (\eqref{eq : v0N}). 
Moreover, let  ${E_1}, {E_{1,2}}$ and ${E_2}$ be the sequences in $\ell^2_{D_2}$ defined in \eqref{def : definition of the sequences Ei} and let $C_1(d), C_{12}(d), C_2(d)$ be the non-negative constants defined in \eqref{def : definition of the constants Ci}. We have that if $\mathcal{Z}_{u,1}$ and $\mathcal{Z}_{u,2}$ are bounds satisfying
\begin{align}\label{def : upper bound Zu1 Zu2}
    (\mathcal{Z}_{u,1})^2 \ge &\frac{C_0^2e^{-2ad}|\om|}{a^2}\left(V_0^N,V_0^N*\pi^{4N}(E_1+E_2)\right)_2\\ \nonumber
    (\mathcal{Z}_{u,2})^2
     \geq & (\mathcal{Z}_{u,1})^2 + e^{-4ad}C_0^2|\om| (V_0^N,V_0^N*[C_1(d)\pi^{4N}(E_1)+ C_{12}(d)\pi^{4N}(E_{1,2}) + C_{2}(d)\pi^{4N}(E_2)])_2,
\end{align}
then $\mathcal{Z}_u \bydef \left((\mathcal{Z}_{u,1})^2 + (\mathcal{Z}_{u,2})^2\right)^\frac{1}{2}$  satisfies \eqref{def : Z1 periodic and Zu}, that is $\|\left(\mathbb{L}^{-1}-\Gamma^\dagger(L^{-1})\right) \mathbb{M}_{v_0}^N  \|_2 \le \mathcal{Z}_u$.
\end{lemma}

\begin{proof}
Let $u \in L^2_{D_2}$ such that $\|u\|_2 = 1$ and let us denote $v\bydef  v_0^N u.$ By construction, $v \in L^2_{D_2}$ and supp$(v) \subset \om$. First, note that 
\begin{equation} \label{eq:proof_E_i}
    \|\left(\mathbb{L}^{-1}-\Gamma^\dagger(L^{-1})\right) \mathbb{M}_{v_0}^N u \|_2^2 =  \|\out \mathbb{L}^{-1} v \|_2^2 +  \|\cha \left(\mathbb{L}^{-1}-\Gamma^\dagger(L^{-1})\right) v \|_2^2
\end{equation}
since $\Gamma^\dagger(L^{-1}) = \cha \Gamma^\dagger(L^{-1}) \cha.$ By definition of $f_0$ in \eqref{def : definition of f0}, the first term in \eqref{eq:proof_E_i} is given by
\begin{align}\label{eq : identity Linv and f0}
    \|\out \mathbb{L}^{-1} v \|_2^2 = \|\out f_0*v\|_2^2.
\end{align}
Moreover, combining Theorem 3.7 in \cite{unbounded_domain_cadiot} and Proposition \ref{prop : computation_of_f}, we obtain that
\begin{align*}
     \|\out f_0*v\|_2^2 \leq \frac{C_0^2e^{-2ad}}{a^2}\int_{\om}v_0^N(x)^2 e_0(x) dx,
\end{align*}
where $e_0(x) \bydef e^{2ad} - e^{2ad}\prod_{k=1}^2\left(1 -e^{-2ad}\cosh(2ax_k)\right)$. In particular, notice that $e_0(x) \geq 0$ for all $x \in \om$. Moreover, straightforward computations lead to 
\begin{align*}
    e_0(x) \leq \cosh(2ax_1) + \cosh(2ax_2)
\end{align*}
for all $x \in \om$.
 Therefore, using Parseval's identity, we get
\begin{align}\label{eq : Zu1 in lemma}
     \|\out f_0*v\|_2^2 \leq \frac{C_0^2e^{-2ad}|\om|}{a^2}(V^N_0,V^N_0*(E_1+E_2))_2.
\end{align}
Now, since $V^N_0 = \pi^{2N} V_0^N$ by definition (\eqref{eq : v0N}),  we have that
\begin{align*}
    \left(V_0^N,V_0^N*(E_1+E_2)\right)_2 = \left(V_0^N,\pi^{2N}\left(V_0^N*(E_1+E_2)\right)\right)_2.
\end{align*}
Then by definition of the discrete convolution we get 
\begin{align}\label{eq : finite number elements E0}
    \pi^{2N}\left(V_0^N*(E_1+E_2)\right) = \pi^{2N}\left(V_0^N*\pi^{4N}((E_1+E_2))\right).
\end{align} 
This implies that 
\begin{align}\label{eq : ineq Zu1 in lemma}
     \|\out f_0*v\|_2^2 \leq \left(\mathcal{Z}_{u,1}\right)^2.
\end{align}
To bound the second term 
of \eqref{eq:proof_E_i}, the proof of Theorem 3.7 in \cite{unbounded_domain_cadiot} provides that 
\begin{align}\label{eq : first inequality Zu2 from article 1}
    \|\cha \left(\mathbb{L}^{-1}-\Gamma^\dagger(L^{-1})\right) v \|_2^2 \leq  (\mathcal{Z}_{u,1})^2 +  \sum_{n \in \mathbb{Z}^2_*} \int_{\mathbb{R}^2\setminus (\om \cup (\om +2dn))} \left|\mathbb{L}^{-1} v(y) \mathbb{L}^{-1}  v(y-2dn)\right|   dy.
\end{align}
First, notice that  $\mathbb{L}^{-1}v \in H^l_{D_2}$ since $v \in L^2_{D_2}$. Then, let $n \in \mathbb{Z}^2_*$. Using the change of variable $(y_1, y_2) \mapsto (-y_1,y_2)$ and the $D_2$-symmetry, we get
\begin{align}\label{eq : proof_symmetries_1}
\nonumber
   &\hspace{-1cm} \int_{\mathbb{R}^2\setminus (\om \cup (\om +2dn))} \left|\mathbb{L}^{-1} v(y) \mathbb{L}^{-1}  v(y-2dn) \right|  dy\\\nonumber
   =    & ~ \int_{\mathbb{R}^2\setminus (\om \cup (\om +2dn))} \left|\mathbb{L}^{-1} v(-y_1,y_2) \mathbb{L}^{-1}  v(-y_1-2dn_1,y_2-2dn_2)  \right| dy\\
     = & ~ \int_{\mathbb{R}^2\setminus (\om \cup (\om +2dn))} \left|\mathbb{L}^{-1} v(y) \mathbb{L}^{-1}  v(y_1+2dn_1,y_2-2dn_2)  \right| dy.
\end{align}
Similarly,
\begin{align}\label{eq : proof_symmetries_2}
\nonumber
   &\hspace{-1.5cm}  \int_{\mathbb{R}^2\setminus (\om \cup (\om +2dn))}\left| \mathbb{L}^{-1} v(y) \mathbb{L}^{-1}  v(y-2dn)  \right| dy \\
     = & ~ \int_{\mathbb{R}^2\setminus (\om \cup (\om +2dn))} \left|\mathbb{L}^{-1} v(y) \mathbb{L}^{-1}  v(y_1-2dn_1,y_2+2dn_2) \right|  dy.
\end{align}
Therefore, combining \eqref{eq : proof_symmetries_1} and \eqref{eq : proof_symmetries_2}, we get
\begin{align}\label{eq : using symmetry Z1}
\nonumber
    & \hspace{-1cm}  \sum_{n \in \mathbb{Z}^2_*} \int_{\mathbb{R}^2\setminus (\om \cup (\om +2dn))} \left|\mathbb{L}^{-1} v(y) \mathbb{L}^{-1}  v(y-2dn) \right|  dy \\
    = & ~ \sum_{n \in \mathbb{N}_0^2, n \neq 0} \alpha_n \int_{\mathbb{R}^2\setminus (\om \cup (\om +2dn))} \left|\mathbb{L}^{-1} v(y) \mathbb{L}^{-1}  v(y-2dn) \right|  dy,
\end{align}
where $\alpha_n$ is given in \eqref{def : alpha_n}. Recall that $\mathbb{L}^{-1}v = f_0*v$ by definition of $f_0$ in \eqref{def : definition of f0}. Consequently, using Proposition~\ref{prop : computation_of_f}, we obtain
\begin{align}\label{eq : f0 inequality in proof lemma 36}
    \left|\mathbb{L}^{-1}v(x)\right| = \left|\left(f_0*v\right)(x)\right| = \left|\int_{\om} f_0(x-y)v(y)dy\right| \leq C_0 \int_{\om} e^{-a|x-y|_1}|v(y)|dy,
\end{align}
for all $x \in \R^2.$ Then, combining \eqref{eq : using symmetry Z1} and \eqref{eq : f0 inequality in proof lemma 36}, we get
\begin{align}\label{eq : Zu2 in proof with exponential ineq}
    & \hspace{-1.2cm}  \sum_{n \in \mathbb{Z}^2_*} \int_{\mathbb{R}^2\setminus (\om \cup (\om +2dn))} \left|\mathbb{L}^{-1} v(y) \mathbb{L}^{-1}  v(y-2dn) \right|  dy \\\nonumber
    \leq & ~ C_0^2\sum_{n \in \mathbb{N}_0^2, n \neq 0} \alpha_n   \int_{\mathbb{R}^2\setminus (\om \cup (\om +2dn))} \int_\om \int_\om e^{-a|y-x|_1}e^{-a|y-2dn-z|_1}|v(x)v(z)|dxdzdy.
\end{align}
Now, using Proposition~\ref{prop : inequality depending on constants Ci}, we get
\begin{align*}
    &\hspace{-1.5cm}  \sum_{n \in \mathbb{N}_0^2, n \neq 0} \alpha_n  \int_{\mathbb{R}^2\setminus (\om \cup (\om +2dn))} \int_\om \int_\om e^{-a|y-x|_1}e^{-a|y-2dn-z|_1}|v(x)v(z)|dxdzdy\nonumber 
    \\
    \leq & ~  e^{-4ad}|\om| \left(V_0^N,V_0^N*\left[C_1(d)E_1+ C_{12}(d)E_{1,2} + C_{2}(d)E_2\right]\right)_2.
\end{align*}
    Moreover, using  \eqref{eq : finite number elements E0},
    we obtain that \begin{align*}
        &\hspace{-1.5cm}  \left(V_0^N,V_0^N*\left[C_1(d)E_1+ C_{12}(d)E_{1,2} + C_{2}(d)E_2\right]\right)_2\\
        = &~ \left(V_0^N,V_0^N*\left[C_1(d)\pi^{4N}\left(E_1\right)+ C_{12}(d)\pi^{4N}\left(E_{1,2}\right) + C_{2}(d)\pi^{4N}\left(E_2\right)\right]\right)_2.
    \end{align*} 
 Consequently, we obtain that 
 \begin{align}\label{eq : Zu2 in lemma proof}
      \|\cha \left(\mathbb{L}^{-1}-\Gamma^\dagger(L^{-1})\right) v \|_2^2 \leq \left(\mathcal{Z}_{u,2}\right)^2.
 \end{align}
 Finally, combining \eqref{eq:proof_E_i}, \eqref{eq : ineq Zu1 in lemma} and \eqref{eq : Zu2 in lemma proof} concludes the proof.
\end{proof}


\begin{remark}\label{rem : Z11 = Z12}
In practice, one has that $\mathcal{Z}_{u,2} \approx \mathcal{Z}_{u,1}$. Consequently, one can estimate the required size of the domain $\om$ by taking $d$ to be large enough so that $\sqrt{2}\mathcal{Z}_{u,1} < \frac{1}{2}$ (having in mind the condition $\mathcal{Z}_1 <1$ in Theorem \ref{th: radii polynomial}). Once $d$ is fixed, we can take the number $N_0$ of Fourier series coefficients large enough in order for $\mathcal{Y}_0$ to be small. More specifically, we determine the required number of coefficients so as to obtain a sharp approximation $U_0$. Once $N_0$ is fixed and $U_0$ is obtained (using the construction of Section \ref{sec : construction of u0}), a large enough $N$ may then be chosen so as to guarantee $Z_1$ is small enough. These heuristics provide a strategy for the practical choice of $d$, $N$ and $N_0$.
\end{remark}

Lemma~\ref{lem : lemma Zu} provides an explicit formula for computing $\mathcal{Z}_u$ given the constants $C_0$ and $a$ defined in Proposition~\ref{prop : computation_of_f}. However, an explicit value for $C_0$ still needs to be computed. Note that an upper bound $\hat{C}_0$ for $C_0$ is actually sufficient in the computation of $\mathcal{Z}_u$. Consequently, we present in the next section a computer-assisted approach to compute a sharp upper bound $\hat{C}_0$ for $C_0$.
The computation of a sharp constant $\hat{C}_0$ is of major importance in our analysis since the lower bound for $\mathcal{Z}_u$ depends linearly on $\hat{C}_0$ (cf. Lemma~\ref{lem : lemma Zu}). Consequently, having a sharp constant $\hat{C}_0$ can help obtain $\mathcal{Z}_1 <1$, where $\mathcal{Z}_1$ is given in \eqref{def : definition of Z1 theoretical}. This last condition is duly required in \eqref{condition radii polynomial} in order for our computer-assisted approach to be applicable.

\subsubsection{Computation of an upper bound for \boldmath$C_0$}\label{sec : computation of C_0}
In this section, we present a strategy based on computer-assisted proofs in which we provide a rigorous representation for $f_0(x)$ on the computer for all $x \in \R^2$ such that $|x| \leq \beta$, with $\beta$ defined in \eqref{def : definition C0 as a sup}.  This allows to verify \eqref{def : definition of C0 and a} for all $|x| \leq \beta$ rigorously on the computer. Then, using \eqref{def : definition C0 as a sup}, we can compute an upper bound for $C_0.$ In particular, all computational aspects are implemented in Julia (cf. \cite{julia_fresh_approach_bezanson}) via the package RadiiPolynomial.jl (cf. \cite{julia_olivier}) which relies on the package IntervalArithmetic.jl (cf. \cite{julia_interval}) for rigorous computations. The specific algorithmic details complementing this article may be found at \cite{julia_cadiot}.

In this section, we define $g : [0,\infty) \to \mathbb{C}$ as 
\begin{align}\label{def : definition of function g}
    g(r) \bydef \frac{ e^{\sqrt{2}ar}}{2i\sqrt{\mu}}\bigg(K_0(br) - K_0(\overline{b}r) \bigg)
\end{align}
Using \eqref{eq : equality for f0}, we want to study $g$ in order to compute an upper bound for $C_0.$ More specifically, we get from \eqref{def : definition C0 as a sup} that $|g(r)| \le C_0$ for all $r \in [0,\beta]$. Our goal is to compute an explicit and computable upper bound $\hat{C}_0$ for $|g|$.

 In practice, we begin by studying the graph of $g$ and obtaining a numerical upper bound for $C_0$, that we denote $C_{0,num}$.  The constant $C_{0,num}$ is not a rigorous upper bound but it provides a  numerical approximation that will be useful in computing $\hat{C}_0$. In particular, $C_{0,num}$ should be close to $C_0$,  as in practice, it is obtained by evaluating numerically $g$ on a fine grid of $[0,\beta]$ and then taking the maximum of the obtained evaluations. Now fix $\delta >0$ a numerical error tolerance, and let
\begin{equation}\label{eq : definition C0hat}
    \hat{C}_0 \bydef \delta + C_{0,num}.
\end{equation}
Note that if $\delta$ is too small, then it might be difficult to use interval arithmetic to prove that $\hat{C}_0$ is actually an upper bound. On the other hand, if $\delta$ is too big, then $\hat{C}_0$ might be far from $C_0.$ 

Since the modified Bessel function $K_0$ is singular at zero, our first objective is to obtain an explicit (and computable) representation for $e^{-\sqrt{2}ar}g(r)$ for all $r\geq 0$ in order to be able to evaluate $g$ point-wise. 
\begin{prop}\label{prop : K_0 series}
Let $r \geq 0$ and let $\theta$ be defined in \eqref{def : definition of theta}, then
\begin{equation}\label{eq : series of K0}
    \frac{K_0(rb) - K_0(r\overline{b})}{2i\sqrt{\mu}} = \frac{1}{\sqrt{\mu}} \sum_{k=0}^\infty \frac{(r|b|)^{2k}}{4^k(k!)^2}\left(\psi(k+1)\sin(2k\theta) - \ln\left(\frac{|rb|}{2}\right)\sin(2k\theta) - \theta \cos(2k\theta)   \right).
\end{equation}
\end{prop}

\begin{proof}
Using \cite{watson_bessel}, we know that 
\[
K_0(z) = \sum_{k=0}^\infty \frac{\psi(k+1)- \ln\left(\frac{z}{2}\right)}{4^k(k!)^2}z^{2k}
\]
for all $z \in \mathbb{C}$, where $\psi$ is the digamma function. In particular, since $\ln(\overline{z}) = \overline{\ln(z)}$ for all $z \in \mathbb{C}$, we obtain that $K_0(\overline{z}) = \overline{K_0(z)}$ for all $z \in \mathbb{C}$. Consequently, given $r \geq 0$, we have
\begin{equation}\label{eq : interval eps1 eps2 0}
    K_0(rb) - K_0(r\overline{b}) = 2i\text{Imag}(K_0(rb)) = 2i\text{Imag}\left(\sum_{k=0}^\infty \frac{\psi(k+1)-\ln\left(\frac{rb}{2}\right)}{4^k(k!)^2}(rb)^{2k}\right).
\end{equation}

Then notice that 
\begin{equation}\label{eq : interval eps1 eps2 1}
  \text{Imag}\left(\psi(k+1)(rb)^{2k}\right) = (r|b|)^{2k}\psi(k+1)\sin(2k\theta).  
\end{equation}
Moreover, using that $\ln(\frac{rb}{2}) = \ln(\frac{r|b|}{2}) + i \theta$, we obtain
\begin{equation}\label{eq : interval eps1 eps2 2}
    \text{Imag}\left(\ln\left(\frac{rb}{2}\right)(rb)^{2k}\right) = (r|b|)^{2k}\left( \ln\left(\frac{|rb|}{2}\right)\sin(2k\theta) + \theta \cos(2k\theta)   \right).
\end{equation}
We conclude the proof combining \eqref{eq : interval eps1 eps2 0}, \eqref{eq : interval eps1 eps2 1} and \eqref{eq : interval eps1 eps2 2}.
\end{proof}

Since $K_0$ is singular at $0$, we separate the analysis at zero and the one away from zero. For $r$ close to zero, we control theoretically how far $e^{-\sqrt{2}ar}g(r)$ is to $g(0) = \frac{-\theta}{\sqrt{\mu}}$ (cf. Proposition \ref{prop : computation_of_f}), which we now present. 
\begin{prop}\label{MVT}
Let $\theta$ be defined in \eqref{def : definition of theta}, then for all $r \in [0,\frac{1}{|b|}]$,
\begin{equation}
        \left|\frac{K_0(rb) - K_0(r\overline{b})}{2i\sqrt{\mu}}  +\frac{\theta}{\sqrt{\mu}}\right| \leq \frac{(e^{\frac{1}{4}}-1)|b|(4 + e^{-1} +  |\theta|)}{\sqrt{\mu}} r.
\end{equation}
\end{prop}
\begin{proof}
Using \eqref{eq : series of K0}, we have  
\begin{align}
\frac{K_0(rb) - K_0(r\overline{b})}{2i\sqrt{\mu}} &= -\frac{\theta}{\sqrt{\mu}} \\
& \quad + \frac{r^2|b|^2}{\sqrt{\mu}} \sum_{k=1}^\infty \frac{(r|b|)^{2(k-1)}}{4^k(k!)^2}\left(\psi(k+1)\sin(2k\theta) - \ln(\frac{|rb|}{2})\sin(2k\theta) - \theta \cos(2k\theta) \right).
\end{align}
Since $|x \ln(x)| \leq e^{-1}$ for all $x \in [0,1]$ and $r|b| \le 1$ (indeed $r \leq \frac{1}{|b|}$), we get 
\[
\left|\frac{K_0(rb) - K_0(r\overline{b})}{2i\sqrt{\mu}}  +\frac{\theta}{\sqrt{\mu}}\right| \leq \frac{r|b|}{\sqrt{\mu}} \sum_{k=1}^\infty \frac{\psi(k+1) + e^{-1} + |\theta|}{4^k(k!)^2}.
\]
Now notice that $\psi(k+1) = \sum_{n=1}^k\frac{1}{n} - C_{Euler} \leq \frac{1}{2} + \ln(k)$ for all $k \geq 1$, where $C_{Euler} \approx 0.58$ is the Euler–Mascheroni constant.  Therefore, 
\begin{equation}\label{eq : digamma ineq}
   \frac{\psi(k+1)}{k} \leq \frac{1}{2k} + \frac{\ln(k)}{k} \leq \frac{1}{2} + e \leq 4 
\end{equation}
for all $k \geq 1$. Therefore,
\begin{align*}
\left|\frac{K_0(rb) - K_0(r\overline{b})}{2i\sqrt{\mu}}  +\frac{\theta}{\sqrt{\mu}}\right| &\leq \frac{r|b|}{\sqrt{\mu}} \sum_{k=1}^\infty \frac{4 + e^{-1} + |\theta| }{4^k k!} \\
&\leq \frac{r|b|(4 + e^{-1} + |\theta| )}{\sqrt{\mu}} \sum_{k=1}^\infty \frac{1}{4^k k!}= \frac{(e^{\frac{1}{4}}-1)r|b|(4 + e^{-1} + |\theta| )}{\sqrt{\mu}},
\end{align*}
where we used that $
\sum_{k=1}^\infty \frac{1}{4^k k!} = e^{\frac{1}{4}}-1.
$
\end{proof}

Let $\epsilon >0$ be defined as 
\begin{align*}
    \epsilon \bydef \min\left\{\frac{1}{|b|},~ \frac{\sqrt{\mu}\delta}{(e^{\frac{1}{4}}-1)|b|(4+e^{-1}+|\theta|)}\right\}.
\end{align*}
Then the previous Proposition \ref{MVT} provides that
\begin{equation}\label{eq : C0hat for 0 epsilon}
    |g(r)| \leq \left|g(r) + \frac{\theta e^{\sqrt{2}ar}}{\sqrt{\mu}}\right| - \frac{\theta e^{\sqrt{2}ar}}{\sqrt{\mu}} \leq e^{\sqrt{2}ar}\left(\delta - \frac{\theta}{\sqrt{\mu}}\right),
\end{equation}
for all $r \in [0, \epsilon]$.
In particular, if $\hat{C}_0 \geq e^{\sqrt{2}a\epsilon}\left(\delta - \frac{\theta}{\sqrt{\mu}}\right)$, then $|g(r)| \leq \hat{C}_0$ for all $r \in [0,\epsilon].$

The value of $\epsilon >0$ being fixed, it remains to verify that $|g(r)| \leq \hat{C}_0$ for all $r \in [\epsilon, \beta]$. Specifically, our goal is to provide a computer-assisted strategy to verify that $|g(r)| \leq \hat{C}_0$ for all $r \in [\epsilon, \beta]$. In particular, our strategy is based on the use of interval arithmetic (e.g. see \cite{Moore_interval_analysis,MR2807595}). Specifically, given an interval $I \subset [\epsilon,\beta]$, we want to compute an upper bound for the set $g\left(I\right)$ using rigorous numerics. To achieve such a goal, we need to provide a representation of $g$ which is compatible with the computer.

Since we already possess an explicit representation for $e^{-\sqrt{2}ar}g(r)$ given in \eqref{eq : series of K0}, we consider a finite truncation of the sum and control the tail uniformly. The following Proposition \ref{prop : remainder of K0} provides a uniform bound on the tail.
\begin{prop}\label{prop : remainder of K0}
Let $N \in \mathbb{N}$ be big enough so that $\frac{(\beta|b|)^{2(N+1)}}{N!} \leq 1$. Moreover, letting 
\[
C_{log} \bydef \displaystyle\max\left\{ \left|\ln\left(\frac{\epsilon|b|}{2}\right)\right|, ~ \left|\ln\left(\frac{\beta|b|}{2}\right)\right|\right\},
\]
then
\begin{align}\label{ineq : remainder of series of K0}
\nonumber
     & \hspace{-1cm} \bigg|\frac{K_0(rb) - K_0(r\overline{b})}{2i\sqrt{\mu}} - \frac{1}{\sqrt{\mu}} \sum_{k=0}^N \frac{(r|b|)^{2k}}{4^k(k!)^2}\left(\left(\psi(k+1)-\ln\left(\frac{|rb|}{2}\right)\right)\sin(2k\theta) - \theta \cos(2k\theta)   \right)\bigg| \\
    \leq   ~~&\frac{\left(4+ C_{log} +|\theta|\right)}{3\sqrt{\mu}4^N(N+1)!}
\end{align}
for all $r \in [\epsilon,\beta]$, where $\theta$ is defined in \eqref{def : definition of theta}.
\end{prop}
\begin{proof}
Let $f(k,r) \bydef \frac{(r|b|)^{2k}}{4^k(k!)^2}\left(\psi(k+1)\sin(2k\theta) - \ln(\frac{|rb|}{2})\sin(2k\theta) - \theta \cos(2k\theta)   \right) $, then
\begin{align*}
    \frac{1}{\sqrt{\mu}}\sum_{k=N+1}^\infty |f(k,r)| &\leq  \frac{1}{\sqrt{\mu}} \sum_{k=N+1}^\infty \frac{(r|b|)^{2k}}{4^k(k!)^2}\left(\psi(k+1)+ \left|\ln\left(\frac{|rb|}{2}\right)\right| +|\theta|\right)\\
     &\leq \frac{1}{\sqrt{\mu}} \sum_{k=N+1}^\infty \frac{(\beta|b|)^{2k}}{4^k(k!)^2}\left(\psi(k+1)+ C_{log} +|\theta|\right),
\end{align*}
where we used that $\left|\ln\left(\frac{|rb|}{2}\right)\right| \leq C_{log}$ since $\ln$ is monotone. Moreover, using \eqref{eq : digamma ineq}, we get
\begin{align*}
    \frac{1}{\sqrt{\mu}}\sum_{k=N+1}^\infty |f(k,r)| &\leq \frac{\left(4+ C_{log} + |\theta|\right)}{\sqrt{\mu}} \sum_{k=N+1}^\infty \frac{k(\beta|b|)^{2k}}{4^k(k!)^2}.
\end{align*}
Finally, as $\frac{(\beta|b|)^{2(N+1)}}{N!} \leq 1$ by assumption, we obtain that $\frac{k(\beta|b|)^{2k}}{(k!)} \leq 1$ for all $k \geq N+1$ and so
\[
    \frac{1}{\sqrt{\mu}}\sum_{k=N+1}^\infty |f(k,r)| \leq \frac{\left(4+ C_{log} + |\theta|\right)}{\sqrt{\mu}(N+1)!} \sum_{k=N+1}^\infty \frac{1}{4^k} = \frac{\left(4+ C_{log} + |\theta|\right)}{3\sqrt{\mu}4^N(N+1)!}. \qedhere
\]
\end{proof}
Now, consider a decomposition of the interval $[\epsilon,\beta]$  as $[\epsilon,\beta] = \displaystyle\bigcup_{n =0}^p I_n$ where $p \in \mathbb{N}$ and $(I_n)_{0 \leq n \leq p}$ is a sequence of intervals. Given a bounded interval $I \subset \R$ and a function $h$ continuous on $\overline{I}$, we define $h(I) \bydef \left\{h(x) : x \in I\right\}$.
Then, given $N \in \mathbb{N}$ big enough so that $\frac{(\beta|b|)^{2(N+1)}}{N!} \leq 1$ and combining \eqref{eq : series of K0} and \eqref{ineq : remainder of series of K0}, we obtain that
\[
\sup\left(\left|g(I_n)\right|\right) \leq \sup\left(\left|g^N(I_n)\right|\right) + \frac{\left(4+ C_{log} + |\theta|\right)}{3\sqrt{\mu}4^N(N+1)!} \sup\left(e^{\sqrt{2}aI_n}\right)
\]
for all $n \in \{0,1,\dots,p\}$, where 
\[
g^N(r) \bydef e^{\sqrt{2}ar}\frac{1}{\sqrt{\mu}} \sum_{k=0}^N \frac{(r|b|)^{2k}}{4^k(k!)^2}\left(\left(\psi(k+1)-\ln\left(\frac{|rb|}{2}\right)\right)\sin(2k\theta) - \theta \cos(2k\theta)   \right).
\]
Now, upper bounds for both $\sup\left(|g^N(I_n)|\right)$ and $\sup\left(e^{\sqrt{2}aI_n}\right)$ can be computed thanks to the arithmetic on intervals for each $n \in \{0, 1, \dots, p\}$. This is achieved using the package IntervalArithmetics on Julia \cite{julia_interval}. In particular,  we verify that 
\begin{align}\label{eq : C0hat for bounded domain}
    \sup\left(\left|g^N(I_n)\right|\right) + \frac{\left(4+ C_{log} + |\theta|\right)}{3\sqrt{\mu}4^N(N+1)!} \sup\left(e^{\sqrt{2}aI_n}\right) \leq \hat{C}_0
\end{align}
for all $n \in \{0, 1, \dots, p\}$, which implies that $|g(r)| \leq \hat{C}_0$ for all $r \in [\epsilon, \beta]$. 

Consequently, combining \eqref{eq : C0hat for 0 epsilon} and \eqref{eq : C0hat for bounded domain}, we verify that $|g(r)| \leq \hat{C}_0$ for all $r \in [0,\beta]$. Using Proposition \ref{prop : computation_of_f}, this implies that 
\[
C_0 \leq \hat{C}_0.
\]
In practice, $C_0$ allows computing the bound $\mathcal{Z}_u$ (defined in \eqref{def : definition of Z1 theoretical}). Specifically, $\mathcal{Z}_u$ is linear in $C_0$ (cf. Lemma \eqref{lem : lemma Zu}). 
 Since the hypotheses of Theorem \ref{th: radii polynomial} require an upper bound $\mathcal{Z}_1 = Z_1 + \mathcal{Z}_u$, we derive the abstract computation of the bound $\mathcal{Z}_u$ with $C_0$, which is theoretical, in order to improve readability. From the point of view of the computer-assisted proof, the constant $\hat{C}_0$, which is a rigorous upper bound for $C_0$, allows to compute a rigorous upper bound for $\mathcal{Z}_u$. Moreover, Lemmas \ref{lem : bound Y_0}, \ref{lem : bound Z_2}, \ref{lem : usual term periodic Z1} and \ref{lem : lemma Zu} allow computing the bounds of Theorem \ref{th: radii polynomial} rigorously with IntervalArithmetic.jl \cite{julia_interval}. Once these bounds are computed, we determine the smallest value of $r>0$ for which \eqref{condition radii polynomial} is satisfied. This provides a computer-assisted proof of existence and uniqueness in the ball $\overline{B_r(u_0)}$. 




\subsection{Proof of a branch of periodic solutions}\label{ssec : proof of periodic solutions}

Since our analysis is based on Fourier series, one notices many similarities with the computer-assisted proofs of periodic solutions using a Newton-Kantorovich approach (see \cite{van2021spontaneous}). More specifically, the bounds $\mathcal{Y}_0$, $\mathcal{Z}_1$ and $\mathcal{Z}_2$ have corresponding bounds $Y_0$, $Z_1$ and $Z_2$ associated to the periodic problem on $\om.$ In addition, we derive a condition under which a proof of a localized pattern using Theorem~\ref{th: radii polynomial} implies a proof of existence of a branch of periodic solution converging to the localized pattern as the period tends to infinity.  In practice, this condition is easily satisfied if $\mathcal{Z}_u$ is small enough (namely $\max\{1, \|B^N\|_{2}\}\mathcal{Z}_u \ll Z_1$) and if $d$ is large enough (this point is quantified in Lemma~\ref{lem : riemann sum} below).

Let $q \in [d, \infty]$ and define $\Omega(q) \bydef (-q,q)^2$.
Then, define $\gamma_q : L^2_{D_2} \to \ell^2_{D_2}$ and $\gamma_q^\dagger : \ell^2_{D_2} \to L^2_{D_2}$  as 
\begin{align}
    \left(\gamma_q(u)\right)_n &\bydef \frac{1}{|\Omega(q)|}\int_{\Omega(q)} u(x) e^{-i\frac{\pi}{q} n\cdot x}dx \\
    \gamma^\dagger_q(U)(x) &\bydef \mathbb{1}_{\Omega(q)}(x) \sum_{n \in \mathbb{N}_0^2} \alpha_nu_n \cos\left(\frac{\pi n_1x_1}{q}\right)\cos\left(\frac{\pi n_2x_2}{q}\right)
\end{align}
for all $n \in \mathbb{N}_0^2$ and  for all $x = (x_1,x_2) \in \R^2$.  Moreover, define
\begin{align}
    L^2_q &\bydef \left\{u \in L^2_{D_2} ~:~ \text{supp}(u) \subset \overline{\Omega(q)} \right\}.
\end{align}
Moreover, denote by $\mathcal{B}_{\Omega(q)}(L^2_{D_2})$ the following subspace of $\mathcal{B}(L^2_{D_2})$
\begin{equation}\label{def : Bomega with period q}
    \mathcal{B}_{\Omega(q)}(L^2_{D_2}) \bydef \left\{\mathbb{K}_{\Omega(q)} \in \mathcal{B}(L^2_{D_2}) ~:~  \mathbb{K}_{\Omega(q)} = \mathbb{1}_{\Omega(q)} \mathbb{K}_{\Omega(q)} \mathbb{1}_{\Omega(q)}\right\}.
\end{equation}
Finally, define $\Gamma : \mathcal{B}(L^2_{D_2}) \to \mathcal{B}(\ell^2_{D_2})$ and $\Gamma^\dagger : \mathcal{B}(\ell^2_{D_2}) \to \mathcal{B}(L^2_{D_2})$ as follows
\begin{align}\label{def : Gamma and Gamma dagger with period 2q}
    \Gamma_q(\mathbb{K}) \bydef \gamma_q \mathbb{K} \gdag_q ~~ \text{ and } ~~  \Gamma^\dagger_q(K) \bydef \gamma^\dagger_q K \gamma_q 
\end{align}
for all $\mathbb{K} \in \mathcal{B}(L^2_{D_2})$ and all $K \in \mathcal{B}(\ell^2_{D_2}).$ 

Now, let $L_q$ be the diagonal infinite matrix with entries $(l(\frac{n}{2q}) )_{n \in \mathbb{N}^2_0}$ on the diagonal. In other words, $L_q$ is the Fourier coefficients representation of $\mathbb{L}$ on $\Omega(q)$ with periodic boundary conditions. This allows defining the Hilbert space $X^l_q$ as 
\begin{align*}
    X^l_q \bydef \left\{U  \in \ell^2_{D_2},~ \|U\|_{l,q} \bydef \sqrt{|\Omega(q)|}\|L_q U\|_2 < \infty \right\}.
\end{align*} 
Denote $G_q(U) \bydef \nu_1 U*U + \nu_2 U*U*U$ for all $U \in X^l_q$, where $U*V = \gamma_q\left(\gamma_q^\dagger\left(U\right)\gamma_q^\dagger\left(V\right)\right)$.
Now, we define the following zero finding problem
\begin{align}\label{eq : zero finding period 2q}
    {F}_q(U) \bydef {L}_q U + {G}_q(U) =0
\end{align}
which is equivalent to looking for periodic solutions of period $2q$ for the stationary Swift-Hohenberg equation \eqref{eq : swift_hohenberg}. When $q=+\infty$, we obtain the Fourier transform of the problem on $\R^2$ given in \eqref{eq : f(u)=0 on Hl}.

We want to prove that there exists a unique solution to \eqref{eq : zero finding period 2q} in $\overline{B_r(\gamma_q(u_0))} \subset X^l_{q}$ (for some $r>0$) using the Newton-Kantorovich approach presented in Section \ref{sec : newton kantorovich}.

\begin{theorem}[\bf Family of periodic solutions]\label{th : radii periodic}
Let $u_0 \in H^l_{D_2}$ and $\mathbb{A} : L^2_{D_2} \to H^l_{D_2}$ be defined in \eqref{eq : definition of u0} and \eqref{def : the operator A}.
Moreover, let $\mathcal{Y}_0$, $(Z_1,\mathcal{Z}_u)$ and $\mathcal{Z}_1$ be the bounds satisfying \eqref{def : Y0}, \eqref{def : Z1 periodic and Zu} and \eqref{def : definition of Z1 theoretical}, respectively. Assume that $\hat{\kappa}$ is a constant satisfying
\begin{align}\label{eq : condition kappa}
    \hat{\kappa} \geq \sup_{q \in [d,\infty)} \frac{1}{\sqrt{|\Omega(q)|}}\left(\sum_{n \in \mathbb{Z}^2} \frac{1}{l(\frac{n}{2q})^2}\right)^\frac{1}{2} 
\end{align}
and let $\widehat{\mathcal{Z}}_1$ and $\widehat{\mathcal{Z}}_2(r)$ be  bounds satisfying
\begin{align}
    \widehat{\mathcal{Z}}_1 &\geq \mathcal{Z}_1 + \max\{1, \|B^N\|_{2}\}\mathcal{Z}_{u}\label{def : Z1 chapeau}\\
     \widehat{\mathcal{Z}}_2(r) &\geq 3\nu_2\frac{\hat{\kappa}^2}{\mu} \max \left\{1,\|B^N\|_{2}\right\} r +  \frac{\hat{\kappa}}{\mu}\max\left\{ 2|\nu_1|, ~\left(\|M_{V}(B^N)^*\|_{2}^2+\|V\|_1^2\right)^{\frac{1}{2}}\right\} \label{def : Z2 chapeau}
\end{align}
for all $r>0$. Finally, define $\widehat{\mathcal{Y}}_0 \bydef \mathcal{Y}_0$.
If there exists $r>0$ such that  
\begin{align}\label{eq : radii condition periodic}
    \frac{1}{2}\widehat{\mathcal{Z}}_2(r)r^2 - (1-\widehat{\mathcal{Z}}_1) + \widehat{\mathcal{Y}}_0 <0 \text{ and } \widehat{\mathcal{Z}}_1 + \widehat{\mathcal{Z}}_2(r)r < 1,
\end{align}
then there exists a smooth curve 
\[
\left\{\tilde{u}(q) : q \in [d,\infty]\right\} \subset C^\infty(\R^2)
\]
such that $\tilde{u}(q)$ is a $D_2$-symmetric periodic solution to \eqref{eq : swift_hohenberg} with period $2q$ in both variables. In particular, $\tilde{u}(\infty)$ is a localized pattern on $\R^2.$ Moreover, 
\begin{align}\label{eq : periodic version of uq}
    \tilde{u}(q)(x) = \sum_{n \in \mathbb{N}_0^2} \left(\tilde{U}_q\right)_n \alpha_n \cos(2\pi \tilde{n}_1x_1)\cos(2\pi \tilde{n}_2x_2)
\end{align}
for all $x \in \R^2$, where
$\tilde{U}_q \in \overline{B_{r}(\gamma_q(u_0))} \subset X^l_q$ solves \eqref{eq : zero finding period 2q} for all $q \in [d,\infty].$ 
\end{theorem}
\begin{proof}
Let $q \in [d, \infty]$, then we want to prove that there exists a unique solution to \eqref{eq : zero finding period 2q} in $\overline{B_r(\gamma_q(u_0))} \subset X^l_q$ using the Newton-Kantorovich approach presented in Section \ref{sec : newton kantorovich}. Let us define $U_q \in X^l_q$ as $U_q \bydef \gamma_q(u_0)$, then we need to construct an approximate inverse for $D{F}(U_q) : X^l_q \to \ell^2_{D_2}$. Using the construction introduced in Section \ref{sec : construction of operator A}, we define 
\begin{align*}
    {A}_q \bydef {L}_q^{-1}B_q ~~\text{ and }~~ B_q \bydef \Gamma_q\left(\mathbb{1}_{\Omega(q)\setminus \om} +  \Gamma^\dagger_q\left(\pi_N + B^N\right) \right).
\end{align*}
In particular, using the proof of Lemma \ref{lem : bound Z_2}, we have
\begin{align*}
    \|B_q\|_{2} \leq  \max\{1, \|B^N\|_{2}\}.
\end{align*}
Then, combining Parseval's identity and \eqref{def : the operator A}, we have
\begin{align*}
    \|{A}_q {F}_q(U_q)\|_{l,q} = \sqrt{|\Omega(q)|}\|{B}_q {F}_q(U_q)\|_{2} =  \|\mathbb{B}\mathbb{F}(u_0)\|_{L^2(\R^2)} = \|\mathbb{A} \mathbb{F}(u_0)\|_l \leq \mathcal{Y}_0
\end{align*}
as $\gamma^\dagger_q\left(U_q\right) = u_0$ and supp$(u_0) \subset \overline{\om}$ by definition. 
Then, using that $q \geq d$ combined with Lemma~\ref{lem : lemma Zu},  we get
\begin{align*}
    \|\left(\mathbb{L}^{-1}-\Gamma_q^\dagger\left(L_q^{-1}\right)\right)\mathbb{M}_{v_0}^N\|_2 \leq 
    \mathcal{Z}_u,
\end{align*}
which yields
\begin{equation}\label{eq : trinagle ineq not sharp}
    \|\left(\Gamma_q^\dagger\left(L_q^{-1}\right)-\Gamma_d^\dagger\left(L_d^{-1}\right)\right)\mathbb{M}_{v_0}^N\|_2 \le  \|\left(\mathbb{L}^{-1}-\Gamma_q^\dagger\left(L_q^{-1}\right)\right)\mathbb{M}_{v_0}^N\|_2 +  \|\left(\mathbb{L}^{-1}-\Gamma_d^\dagger\left(L_d^{-1}\right)\right)\mathbb{M}_{v_0}^N\|_2
    \leq 2 \mathcal{Z}_u.
\end{equation}
Therefore, using the proof of Lemma \ref{lem : computation of Z1},
\begin{align*}
    \|I_d - {A}_qD{F}_q(U_q)\|_{l,q} &\leq 
    Z_1 + 2\max\{1, \|B^N\|_{2}\}\mathcal{Z}_{u} + \max\{1, \|B^N\|_{2}\}\|V_0-V_0^N\|_1\\
    &= \mathcal{Z}_1 + \max\{1, \|B^N\|_{2}\}\mathcal{Z}_{u}.
\end{align*}
Finally, in a similar fashion as what was achieved in Lemma \ref{lem : banach algebra}, we obtain that
\begin{align*}
    \|U*V\|_{2} \leq \frac{1}{\sqrt{|\Omega(q)|}}\left(\sum_{n \in \mathbb{Z}^2} \frac{1}{l(\frac{n}{2q})^2}\right)^\frac{1}{2} \|U\|_{l,q}\|V\|_{l,q}
\end{align*}
for all $U, V \in X^l_{q}.$ Moreover, using the proof of Lemma \ref{lem : bound Z_2}, we obtain 
\begin{align*}
    \|{A}_q\left(D{F}_q(U_q) - D{F}_q(V)\right)\|_{l,q} \leq \widehat{\mathcal{Z}}_2(r)r
\end{align*}
for all $V \in \overline{B_r(U_q)} \subset X^l_{q}$. Consequently, if \eqref{condition radii polynomial} is satisfied, then using Theorem 4.6 in \cite{unbounded_domain_cadiot}, there exists a unique solution $\tilde{U}_q$ to \eqref{eq : zero finding period 2q} in $\overline{B_r(U_q)}$ for all $q \in [d,\infty]$. Equivalently, we obtain that $\tilde{u}(q)$, defined in \eqref{eq : periodic version of uq}, is a periodic solution to \eqref{eq : swift_hohenberg} with period $2q$ in both variables.
 Moreover, since $F_q\left(\tilde{U}_q\right) =0$, we have that
 \begin{align*}
    \tilde{U}_q = - L^{-1}_q G(\tilde{U}_q).
 \end{align*}
 Using a bootstrapping argument, we obtain that $\tilde{U}_q$ decays quicker than any algebraic power, which implies that $\tilde{u}(q) \in C^\infty(\R^2).$ Moreover, notice that if \eqref{eq : radii condition periodic} is satisfied for some $r>0$, then \eqref{condition radii polynomial} is satisfied for the same $r$. Consequently, Theorem \ref{th: radii polynomial} implies that $\tilde{u}(\infty)$ is a localized pattern on $\R^2.$

Now, Theorem 4.6 in \cite{unbounded_domain_cadiot} provides that $D{F}_q({U}_q) : X^l_{q} \to \ell^2_{D_2}$ has a bounded inverse for all $q \in [d,\infty]$. Moreover, ${F}_q$ is smooth on $X^l_{q}$ and is also smooth with respect to $q$. Consequently, the implicit function theorem implies that there exists a smooth curve 
$
\left\{\tilde{u}(q): q \in [d,\infty]\right\} \subset C^\infty(\R^2)
$
such that $\tilde{u}(q)$ is a periodic solution to \eqref{eq : swift_hohenberg} with period $2q$ in both variables. 
\end{proof}

The previous theorem provides the existence of an unbounded branch of periodic solutions to \eqref{eq : zero finding period 2q}, provided that the condition \eqref{condition radii polynomial} is satisfied for the newly defined bounds $\widehat{\mathcal{Y}}_0, \widehat{\mathcal{Z}}_1, \widehat{\mathcal{Z}}_2$.  First of, notice that $\widehat{\mathcal{Z}}_1$ might not seem sharp as $q \to d$ or $q \to \infty$ in the proof of Theorem \ref{th : radii periodic} (mostly because of \eqref{eq : trinagle ineq not sharp}). However,  it is very convenient to compute because $\mathcal{Z}_1$, $\|B^N\|_2$ and $\mathcal{Z}_u$ are already known quantities. Moreover, since  $\widehat{\mathcal{Z}}_1 = \mathcal{Z}_1 + \max\{1, \|B^N\|_{2}\}\mathcal{Z}_{u}$,  one has $\widehat{\mathcal{Z}}_1 \approx \mathcal{Z}_1$ if $\mathcal{Z}_{u}$ is negligible compare to $\mathcal{Z}_1$. This situation happens in particular if the quantity $d$ is big enough  (cf. Lemma \ref{lem : lemma Zu}). Consequently, $\widehat{\mathcal{Z}}_1$ is simple to compute and is actually sharp when $d$ is big enough.   Moreover, notice that 
    \[
        \frac{1}{|\Omega(q)|}\sum_{n \in \mathbb{Z}^2}\frac{1}{l(\frac{n}{2q})^2} = \frac{1}{4q^2}\sum_{n \in \mathbb{Z}^2}\frac{1}{l(\frac{n}{2q})^2}
    \]
     is a Riemann sum. In particular, it implies that $\frac{1}{4q^2}\sum_{n \in \mathbb{Z}^2}\frac{1}{l(\frac{n}{2q})^2} = \|\frac{1}{l}\|_2^2 + \mathcal{O}(\frac{1}{q})$ for all $q \in [d,\infty)$. We prove this statement in the next lemma and we compute an explicit value for $\hat{\kappa}$ satisfying \eqref{eq : condition kappa}.

\begin{lemma}\label{lem : riemann sum}
    Let $\hat{\kappa} >0$ be defined as 
    \begin{align}\label{def : kappa chapeau}
         \hat{\kappa}^2 \bydef \frac{2\sqrt{\mu} + (1 + \mu)\left(2\pi - 2\arctan(\sqrt{\mu})\right)}{8\mu^{\frac{3}{2}}(1+\mu)} +\frac{ 2\pi^2 }{d} \left(\frac{3^{\frac{3}{4}}}{\mu^{\frac{7}{4}}} + \frac{3}{\mu^\frac{5}{2}}\right).
    \end{align}
    Then, $\hat{\kappa}$ satisfies \eqref{eq : condition kappa}.
\end{lemma}
\begin{proof}
Let $n \in \mathbb{N}_0^2$, $m \bydef n + (1,1)$ and let $g(\xi) \bydef \frac{1}{l(\xi)^2}$ for all $\xi \in \R^2$. Then
\begin{align} \label{eq : mean value for g}
    \left| \int_{\frac{n_1}{2q}}^{\frac{m_1}{2q}}\int_{\frac{n_2}{2q}}^{\frac{m_2}{2q}} g(\xi)d\xi - \frac{1}{4q^2}g\left(\frac{m}{2q}\right)\right| &= \left|\int_{\frac{n_1}{2q}}^{\frac{m_1}{2q}}\int_{\frac{n_2}{2q}}^{\frac{m_2}{2q}}\left(g(\xi)- g\left(\frac{m}{2q}\right)\right)d\xi\right|\\
&= \left|\int_{\frac{n_1}{2q}}^{\frac{m_1}{2q}}\int_{\frac{n_2}{2q}}^{\frac{m_2}{2q}}\int_0^1 \nabla g\left(t\frac{m}{2q} + (1-t)\xi\right)\cdot\left(\frac{m}{2q}-\xi\right) dt d\xi \right|.
\nonumber
\end{align}
Now, we have
\begin{align*}
    l(\xi) = \mu + (1-|2\pi\xi|^2)^2 \geq \frac{\mu}{2} + \frac{1}{2}|2\pi \xi|^4
\end{align*}
 and 
\begin{align*}
     \partial_{\xi_i}l(\xi) = 4(2\pi)^2 \xi_i\left(|2\pi\xi|^2-1\right) 
\end{align*}
for $i \in \{1,2\}$
 and for all $\xi \in \R^2$. Now, notice that $|2\pi\xi| \left||2\pi\xi|^2-1\right| \leq 1 + |2\pi\xi|^3$ for all $\xi \in \R^2.$ Consequently, we have
\begin{align*}
     |\partial_{\xi_i}g(\xi)| = 2\frac{|\partial_{\xi_i}l(\xi)|}{l(\xi)^3} \leq 
     \frac{8 (2\pi)^2 |\xi_i|\left||2\pi\xi|^2-1\right| }{(\mu + (1-|2\pi\xi|^2)^2)^3} \leq 32 \pi  \left( \frac{|2\pi \xi|^3}{(\mu + |2\pi \xi|^4)^3} + \frac{1}{(\mu + |2\pi \xi|^4)^3} \right)
\end{align*}
for all $\xi \in \R^2$. Now, one can easily prove that 
\begin{align*}
    \frac{x^3}{\mu+x^4} \leq \frac{(3\mu)^{\frac{3}{4}}}{4\mu}
\end{align*}
for all $x \geq 0$. Therefore, defining $\tilde{g} : \R^2 \to (0,\infty)$ as 
\begin{align*}
    \tilde{g}(\xi) \bydef 32\pi  \left(\frac{(3\mu)^{\frac{3}{4}}}{4\mu} \frac{1}{(\mu + |2\pi \xi|^4)^2} + \frac{1}{(\mu + |2\pi \xi|^4)^3}\right)
\end{align*}
we obtain that $|\partial_{\xi_i}g(\xi)| \leq  \tilde{g}(\xi)$ for all $\xi \in \R$ and all $i \in \{1,2\}$.
In particular, notice that $\tilde{g}(\xi)$ is decreasing with $|\xi|.$ Now,  given $\xi \in \left(\frac{n_1}{2q}, ~\frac{m_1}{2q} \right)\times \left(\frac{n_2}{2q}, ~\frac{m_2}{2q} \right)$ and $t \in (0, ~1)$, we have
\begin{align}
\nonumber
    \left|\nabla g\left(t\frac{m}{2q} + (1-t)\xi\right)\cdot\left(\frac{m}{2q}-\xi\right)\right| &\leq \frac{1}{2q} \left|\partial_{\xi_1}g\left(t\frac{m}{2q} + (1-t)\xi\right)\right| + \frac{1}{2q} \left|\partial_{\xi_2}g\left(t\frac{m}{2q} + (1-t)\xi\right)\right|\\
    &\leq \frac{1}{q} \tilde{g}\left(t\frac{m}{2q} + (1-t)\xi\right) \leq \frac{\tilde{g}(\xi)}{q},\label{eq : ineq for g}
\end{align}
where the last inequality follows from $\tilde{g}(\xi)$ decreasing in $|\xi|$. Combining \eqref{eq : mean value for g} and \eqref{eq : ineq for g}, we get  
\begin{align}\label{eq : inequality g on In}
    \left| \int_{\frac{n_1}{2q}}^{\frac{m_1}{2q}}\int_{\frac{n_2}{2q}}^{\frac{m_2}{2q}} g(\xi)d\xi - \frac{1}{4q^2}g\left(\frac{m}{2q}\right)\right| \leq  \frac{1}{q}\int_{\frac{n_1}{2q}}^{\frac{m_1}{2q}}\int_{\frac{n_2}{2q}}^{\frac{m_2}{2q}} \tilde{g}(\xi) d\xi. 
\end{align}
Now, given $n = (n_1,n_2) \in \mathbb{N}_0^2$, define $I(n) \bydef \left(\frac{n_1}{2q}, ~\frac{n_1+1}{2q} \right)\times \left(\frac{n_2}{2q}, ~\frac{n_2+1}{2q} \right)$. Then, combining \eqref{eq : inequality g on In} with the fact that  $g(\xi) \to 0$ and $\tilde{g}(x) \to 0$ as $|\xi| \to \infty$, we have
\begin{align*}
     \left|\frac{1}{4q^2}\sum_{n \in \mathbb{N}_0^2}\frac{1}{l(\frac{n}{2q})^2} - \int_{[0,\infty)^2}g(\xi)d\xi\right|& = \left|\frac{1}{4q^2}\sum_{n \in \mathbb{N}_0^2} \left(g\left(\frac{n}{2q}\right) - \int_{I(n)} g(\xi) d\xi\right)\right| \\
     &\leq \frac{1}{q} \sum_{n \in \mathbb{N}_0^2} \int_{I(n)} \tilde{g}(\xi) d\xi = \frac{1}{q}\int_{[0,\infty)^2} \tilde{g}(\xi) d\xi.
\end{align*}
Now since $l$, $g$ and $\tilde{g}$ are $D_2$-symmetric functions we obtain that
\[
\frac{1}{4q^2}\sum_{n \in \mathbb{Z}^2}\frac{1}{l(\frac{n}{2q})^2} \leq \frac{1}{q^2}\sum_{n \in \mathbb{N}_0^2}\frac{1}{l(\frac{n}{2q})^2} \leq \|\frac{1}{l}\|_2^2 + \frac{1}{q}\int_{\R^2} \tilde{g}(\xi) d\xi.
\]
 Using some standard results on integration of rational functions (see \cite{gradshteyn2014table} for instance), we get
\begin{align*}
    \int_{\R^2} \tilde{g}(\xi) d\xi  = { 32\pi } \left(\frac{(3\mu)^{\frac{3}{4}}}{4\mu}\frac{\pi}{4\mu^\frac{3}{2}} + \frac{3\pi}{16\mu^\frac{5}{2}}\right).
\end{align*}  
We conclude the proof using \eqref{eq : equality norm 2 of 1/l}.
\end{proof}
The previous lemma provides that $\hat{\kappa} \approx \|\frac{1}{l}\|_2$ if $d$ is big enough. In particular, we obtain that $\widehat{\mathcal{Z}}_2(r) \approx \mathcal{Z}_2(r)$ in that situation. 

Consequently, if one is able to compute the quantities $\mathcal{Y}_0$, $\mathcal{Z}_1$ and $\mathcal{Z}_2$ required for the proof of a localized pattern, then, modulo the straightforward computation of $\hat{\kappa}$ in \eqref{eq : condition kappa}, $\widehat{\mathcal{Y}}_0$, $\widehat{\mathcal{Z}}_1$ and $\widehat{\mathcal{Z}}_2$ are obtained without additional analysis. Moreover, in practice, if $d$ is big enough and $\mathcal{Z}_u$ is small enough, then proving the unbounded branch of periodic solutions has the same level of difficulty as the proof of the localized pattern itself.
Finally, the Newton-Kantorovich approach used in the proof of Theorem~\ref{th : radii periodic} provides a uniform control, given by $u_0$, on the branch of solutions.


\section{Constructive proofs of existence of localized patterns}\label{sec : proof of existence of localized patterns}

In this section, we provide the results of our computer-assisted proofs. More specifically, we prove the existence of three different localized patterns, namely the ``square", the ``hexagonal" and the ``octagonal" ones. Note that the symmetries of the patterns are not proven, that is we do not prove that the patterns possess the $D_4, D_6$ or $D_8$ symmetries. The names are only informative. However, the $D_2$-symmetry is obtained by construction as we prove solutions in $H^l_{D_2}.$ 
Note that it would not be difficult to prove the $D_4$ symmetry since $D_4$-sequences can be represented thanks to a reduced set on the Fourier coefficients (similarly as $D_2$ for which the reduced set is $\mathbb{N}_0^2$). However, such a reduced set cannot be readily obtained for proving the $D_6$ and $D_8$ symmetries since the $D_6$ and $D_8$ groups possess "non-cartesian" elements. For such groups, the present approach has to be adapted.

Numerically we obtained three candidates that we denote $u_s, u_h$ and $u_o$ informatively as they correspond to ``square", ``hexagonal" and ``octagonal" symmetries respectively (see Figures \ref{fig : square}, \ref{fig : hexagone} and \ref{fig : octogone}). Each candidate is represented through a Fourier series using the construction of Section~\ref{sec : construction of u0}. For each case, we perform a computer-assisted proof based on Theorem~\ref{th: radii polynomial}. The bounds $\mathcal{Y}_0, \mathcal{Z}_1$ and $\mathcal{Z}_2$ are computed numerically using the code available at \cite{julia_cadiot}. More specifically, we choose $N_0 = 130$, $N = 90$ for the computations of Section~\ref{sec : computation of the bounds} and we compute matrix norms of size $91^2 = 8281.$ For each case, we verify rigorously that the condition \eqref{condition radii polynomial} of Theorem~\ref{th: radii polynomial} is satisfied for some $r_0 >0.$ This implies the existence of a unique solution of \eqref{eq : swift_hohenberg} in $\overline{B_{r_0}(u_k)}$ where $k \in \{s, h, o\}$.

Moreover, by proving the existence of a localized pattern, we prove simultaneously the existence of a branch of periodic solutions using Theorem \ref{th : radii periodic}. In particular, the branch limits the localized pattern as the period tends to infinity. This result has been conjectured in \cite{localized-rigorous} in the context of the 1D quintic SH equation as the authors observed a continuum of periodic solutions, parameterized by the period, that limits to a localized pattern as the period goes to infinity. 

For each computer-assisted proof, we provide the parameters at which the proof is obtained as well as the radius of contraction $r_0$ (cf. Theorem \ref{th: radii polynomial}). We expose the results in the three Theorems \ref{th : square pattern}, \ref{th : hexagonal pattern} and \ref{th : octogonal pattern} below.  In particular, all computational aspects are implemented in Julia (cf. \cite{julia_fresh_approach_bezanson}) via the package RadiiPolynomial.jl (cf. \cite{julia_olivier}) which relies on the package IntervalArithmetic.jl (cf. \cite{julia_interval}) for rigorous interval arithmetic computations. The specific algorithmic details complementing this article can be found at \cite{julia_cadiot}.

\begin{theorem}[\bf The square pattern]\label{th : square pattern}
Let $\mu = 0.27$, $\nu_1 = -1.6$ and $\nu_2 = 1$. Let $r_0 \bydef 1.16 \times 10^{-5}$, then there exists a unique solution $\tilde{u}$ to \eqref{eq : swift_hohenberg} in $\overline{B_{r_0}(u_s)} \subset H^l_{D_2}$ and we have that $\|\tilde{u}-u_s\|_l \leq r_0$.
In addition, there exists a smooth curve 
\[
\left\{\tilde{u}(q) : q \in [d,\infty]\right\} \subset C^\infty(\R^2)
\]
such that $\tilde{u}(q)$ is a periodic solution to \eqref{eq : swift_hohenberg} with period $2q$ in both directions. In particular, $\tilde{u}(\infty) = \tilde{u}$ is a localized pattern on $\R^2.$ 
\end{theorem}

\begin{proof}
    Following the notations of Section \ref{sec : computer assisted analysis}, let us fix $N=90$, $N_0 =130$ and $d = 70$. Then, we construct $u_0 = \gamma^\dagger(U_0)$ as in \eqref{eq : definition of u0} and define $u_s \bydef u_0$, where the construction process is detailed in Section \ref{sec : construction of u0}. Once $U_0$ is fixed, we construct $B^N$ using the approach described in Section \ref{sec : construction of operator A}. In particular, we prove that 
\begin{align}
    \|B^N\|_2 \leq 31.6,
\end{align}
which implies that $\|\mathbb{A}\|_{2,l} \leq 31.6$ using \eqref{eq : equality norm A and BN}. Using Lemma \ref{lem : riemann sum}, we start by proving that 
\[
\hat{\kappa} \bydef 5.62
\]
satisfies \eqref{eq : condition kappa}.
    This allows us to compute the upper bounds introduced in Section \ref{sec : computation of the bounds}. In particular, using \cite{julia_cadiot}, we define 
    \begin{align*}
     \mathcal{Y}_0 \bydef 9.58\times 10^{-6}, ~~
     Z_1 \bydef 0.027 ~~ \text{ and } ~~
     \widehat{\mathcal{Z}}_2(r) \bydef 2988r + 315.85
     \end{align*}
    for all $r>0$ and prove that $  {\mathcal{Y}}_0$, $Z_1$ and $\widehat{\mathcal{Z}}_2$ satisfy \eqref{def : upper bound Y0}, \eqref{def : upper bound periodic Z1} and \eqref{def : Z2 chapeau} respectively. Then, using the approach presented in Section \ref{sec : computation of C_0}, we prove that \eqref{def : definition of C0 and a} is satisfied for $C_0 \bydef 2.6$.
    In particular, define $\mathcal{Z}_{u,1} \bydef 1.5852\times 10^{-3}$ and $\mathcal{Z}_{u,2} \bydef 1.5867\times 10^{-3}$, then we prove that \eqref{def : upper bound Zu1 Zu2} is satisfied and~$\mathcal{Z}_u~\bydef~\left(\left(\mathcal{Z}_{u,1} \right)+ \left(\mathcal{Z}_{u,2} \right)^2\right)^{\frac{1}{2}}$     satisfies \eqref{def : Z1 periodic and Zu}. Consequently, defining $\widehat{\mathcal{Z}}_1$ as 
    \begin{align*}
        \widehat{\mathcal{Z}}_1 \bydef 0.17,
    \end{align*}
    we obtain that \eqref{def : Z1 chapeau} is satisfied. Finally, we prove that $r_0 = 1.16\times 10^{-5} $ satisfies \eqref{eq : radii condition periodic}. We conclude the proof using Theorem \ref{th : radii periodic}.
\end{proof}

   \begin{figure}[H]
  \centering
  \begin{minipage}{0.9\textwidth}
   \centering
   \epsfig{figure=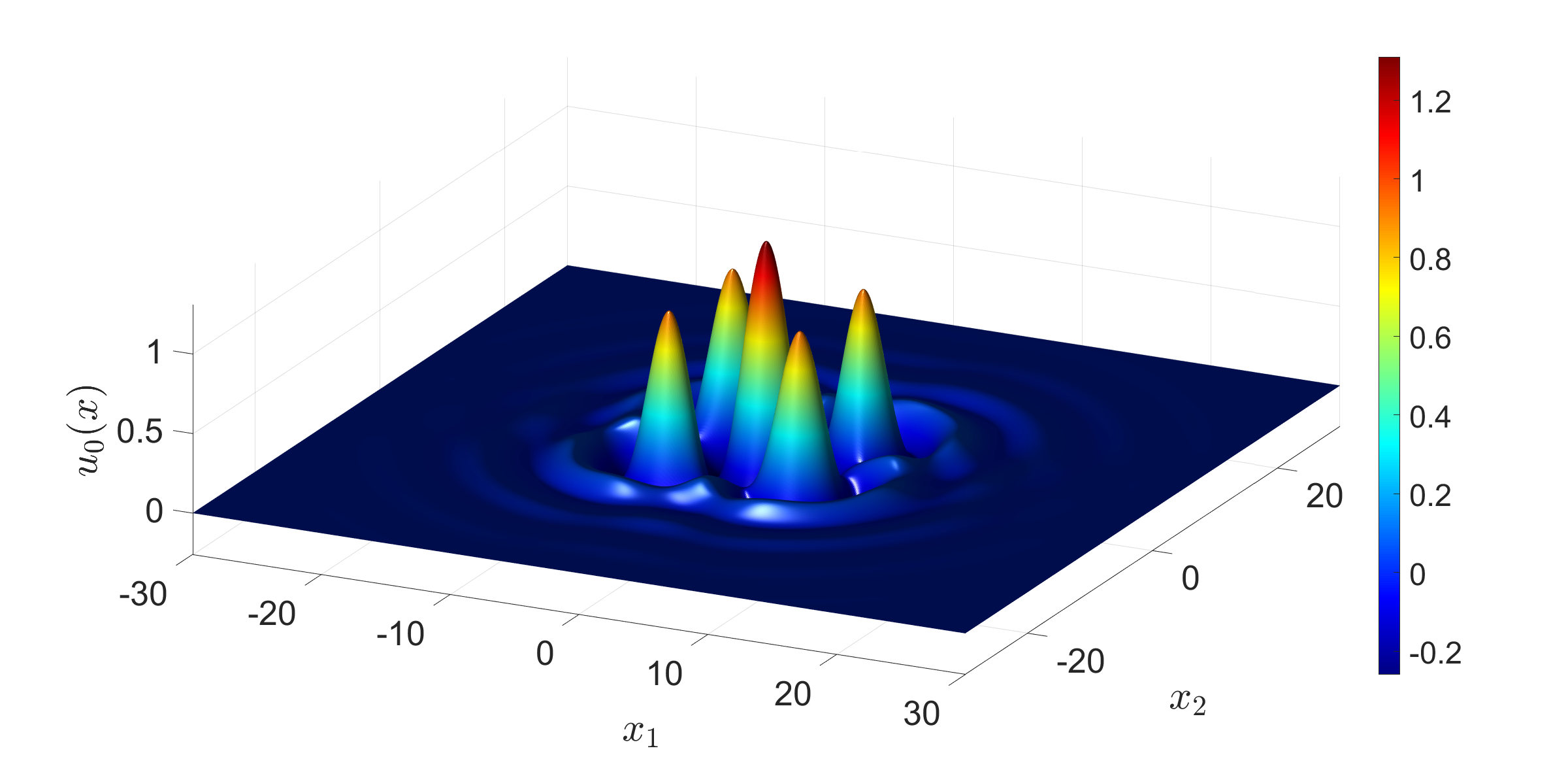,width=\linewidth}
  \end{minipage}%
  \caption{Square pattern constructed on $(-70,70)^2$ and represented on $(-30,30)^2$}
    \label{fig : square}
  \end{figure}

\begin{theorem}[\bf The hexagonal pattern]\label{th : hexagonal pattern}
Let $\mu = 0.32$, $\nu_1 = -1.6$ and $\nu_2 = 1$. Let $r_0 \bydef 8.21 \times 10^{-6}$, then there exists a unique solution $\tilde{u}$ to \eqref{eq : swift_hohenberg} in $\overline{B_{r_0}(u_h)} \subset H^l_{D_2}$ with $\|\tilde{u}-u_h\|_l \leq r_0$. In addition, there exists a smooth curve 
\[
\left\{\tilde{u}(q) : q \in [d,\infty]\right\} \subset C^\infty(\R^2)
\]
such that $\tilde{u}(q)$ is a periodic solution to \eqref{eq : swift_hohenberg} with period $2q$ in both directions. In particular, $\tilde{u}(\infty) = \tilde{u}$ is a localized pattern on $\R^2.$
\end{theorem}

\begin{proof}
    The proof is obtained similarly as the one of Theorem \ref{th : square pattern}. In particular, we define
    \begin{align*}
        \hat{\kappa}&\bydef 4.86\\
        \mathcal{Y}_0 \bydef 7.57\times 10^{-6},~~
        \widehat{\mathcal{Z}}_1 &\bydef 0.078 ~~ \text{ and } ~~
        \widehat{\mathcal{Z}}_2(r) \bydef 1464r + 197.8
    \end{align*}
     for all $r>0$ and prove that $ \hat{\kappa},  {\mathcal{Y}}_0$, $\widehat{\mathcal{Z}}_1$ and $\widehat{\mathcal{Z}}_2$ satisfy \eqref{eq : condition kappa}, \eqref{def : upper bound Y0}, \eqref{def : Z1 chapeau} and \eqref{def : Z2 chapeau} respectively.
\end{proof}

 \begin{figure}[H]
 \centering
 \begin{minipage}{.9\textwidth}
   \centering
   \epsfig{figure=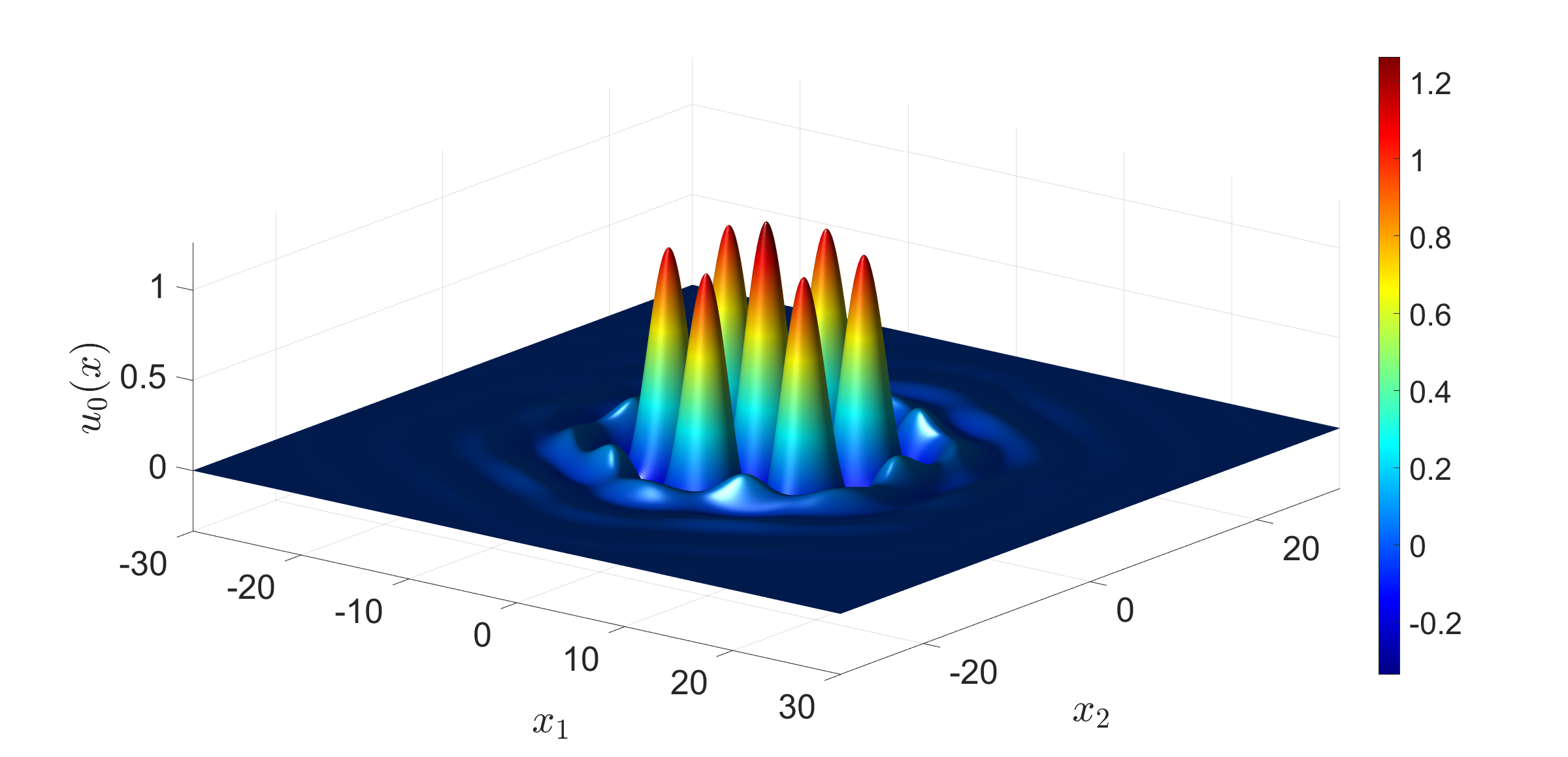,width=\linewidth}
 \end{minipage}%
 \caption{Hexagonal pattern constructed on $(-70,70)^2$ and represented on $(-30,30)^2$}
 \label{fig : hexagone}
 \end{figure}


\begin{theorem}[\bf The octogonal pattern]\label{th : octogonal pattern}
Let $\mu = 0.27$, $\nu_1 = -1.6$ and $\nu_2 = 1$. Let $r_0 \bydef 3.86 \times 10^{-5}$, then there exists a unique solution $\tilde{u}$ to \eqref{eq : swift_hohenberg} in $\overline{B_{r_0}(u_o)} \subset H^l_{D_2}$ with $\|\tilde{u}-u_o\|_l \leq r_0$. In addition, there exists a smooth curve 
\[
\left\{\tilde{u}(q) : q \in [d,\infty]\right\} \subset C^\infty(\R^2)
\]
such that $\tilde{u}(q)$ is a periodic solution to \eqref{eq : swift_hohenberg} with period $2q$ in both directions. In particular, $\tilde{u}(\infty) = \tilde{u}$ is a localized pattern on $\R^2.$
\end{theorem}

\begin{proof}
    The proof is obtained similarly as the one of Theorem \ref{th : square pattern}. In particular, we define 
    \begin{align*}
        \hat{\kappa} &\bydef 5.43\\
        \mathcal{Y}_0 \bydef 2.7\times 10^{-5}, ~~
        \widehat{\mathcal{Z}}_1 &\bydef 0.286 ~~ \text{ and } ~~
        \widehat{\mathcal{Z}}_2(r) \bydef 10402r + 907.8
    \end{align*}
     for all $r>0$ and prove that $\hat{\kappa},  {\mathcal{Y}}_0$, $\widehat{\mathcal{Z}}_1$ and $\widehat{\mathcal{Z}}_2$ satisfy \eqref{eq : condition kappa}, \eqref{def : upper bound Y0}, \eqref{def : Z1 chapeau} and \eqref{def : Z2 chapeau} respectively.
\end{proof}

\begin{figure}[H]
\centering
\begin{minipage}{.9\textwidth}
  \centering
  \epsfig{figure=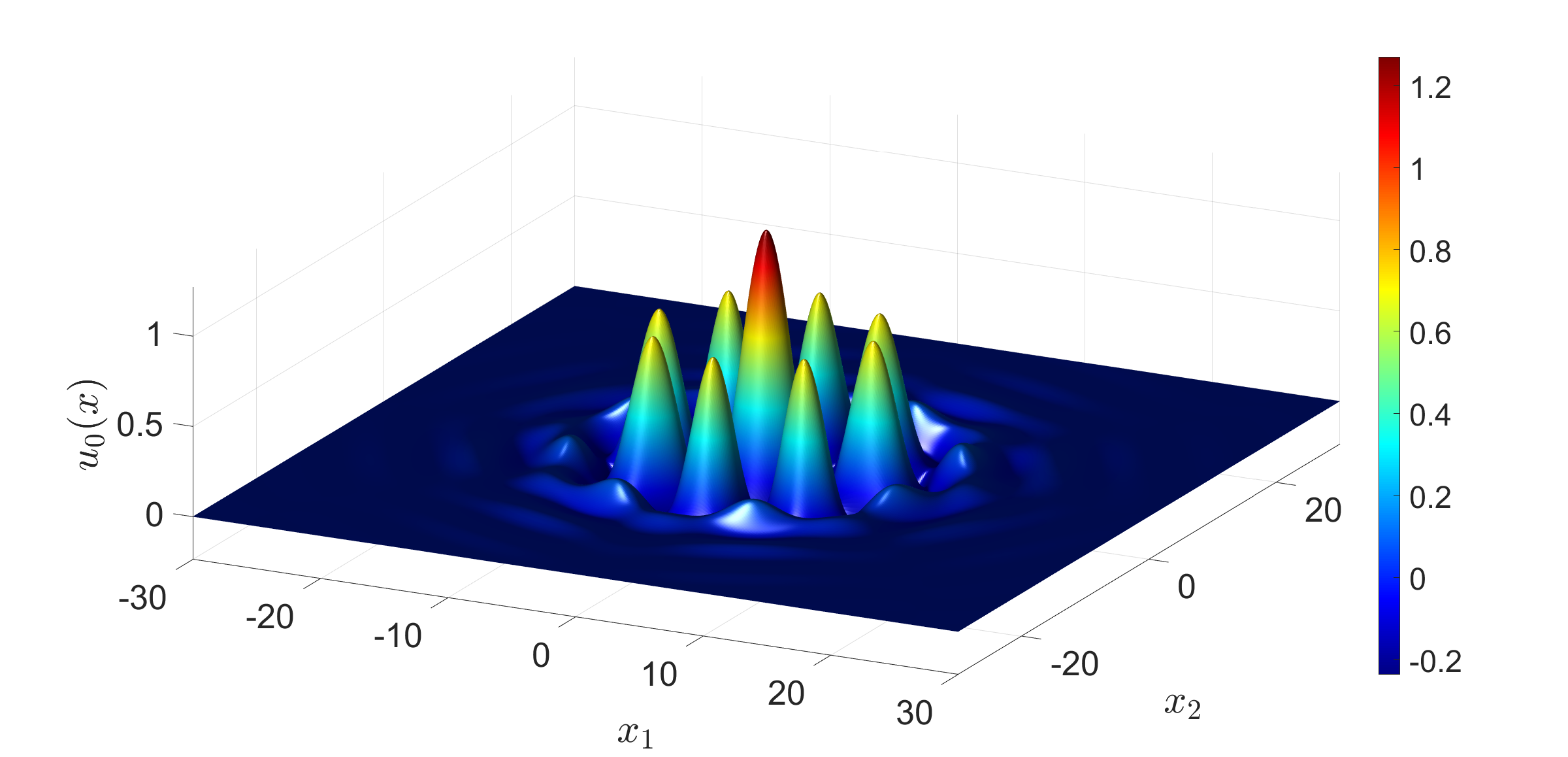,width=\linewidth}
\end{minipage}%
\caption{Octogonal pattern constructed on $(-76,76)^2$ and represented on $(-30,30)^2$}
 \label{fig : octogone}
\end{figure}

In practice, one has to optimize the numerical parameters in order to improve the quality of the  proofs. For instance, once the parameter $d$ is fixed, we chose $N_0$  big enough so that $r_0 \approx 10^{-5}$. One could obtain a better precision by increasing $N_0$, at the cost of increasing computation memory and time. On the other side, if one is simply after an existence result, then the value of $N_0$ can be decreased and optimized to be the smallest value for which the CAP succeeds. The value of $N$ was chosen in order for the bound $\widehat{\mathcal{Z}}_1$ to be around $10^{-1}$. A smaller value for $N$ would increase the value for $\widehat{\mathcal{Z}}_1$ and might cause the proof to fail.

Moreover, one needs to choose the parameter $d$ depending on the size of pattern to prove and its decay at infinity. More specifically, the trace of the approximate solution $u_0$ on the boundary $\partial \om$ needs to be small enough so that $\tilde{U}_0$ and $U_0$, introduced in Section \ref{sec : construction of u0}, are machine precision close. Thanks to Proposition \ref{prop : computation_of_f}, one can estimate the exponential decay of the pattern and hence obtain heuristics on the choice of $d$. Note, however that the bigger the value of $d$ is, the slower the decay of the Fourier coefficients of $U_0$ is. Consequently, proving the existence of large patterns or slowly decaying patterns at infinity becomes numerically challenging.




\section{Acknowledgments}

The authors wish to thank Prof. Jason Bramburger from Concordia University for his insights and fruitful discussions. JPL and JCN would like to acknowledge partial funding from the NSERC Discovery grant program.

\section{Appendix : Proof of Proposition \ref{prop : inequality depending on constants Ci}}\label{appendix}

We expose in the Appendix the proof of Proposition \ref{prop : inequality depending on constants Ci}. First, we present a preliminary result which will be useful for our computations.
\begin{lemma}\label{lem : exponential_computations}
Let $n_1 \geq 0$ and $x_1,z_1 \in (-d,d)$, then
\begin{align}\label{eq : exponential_d}
\nonumber
    &\hspace{-2cm} \int_{-d}^d e^{-a|y_1-x_1|}e^{-a|y_1-z_1-2dn_1|} dy_1\\
    = &\begin{cases}
 (d-x_1 + \frac{1}{2a})e^{-a(2dn_1 + z_1 -x_1)} - \frac{e^{-a(2d(n_1+1) + z_1 + x_1)}}{2a}  &\text{ if } n_1 >0\\
   ( |z_1-x_1| + \frac{1}{a})e^{-a|z_1-x_1|} - \frac{e^{-2ad}}{2a}(e^{-a(x_1+z_1)} + e^{a(x_1+z_1)}) ~ &\text{ if } n_1 =0.
\end{cases}
\end{align}

Similarly,
\begin{align}\label{eq : exponential_dn}
\nonumber
    &\hspace{-2cm} \int_{2dn_1-d}^{2dn_1+d} e^{-a|y_1-x_1|}e^{-a|y_1-z_1-2dn_1|} dy_1\\
    = &\begin{cases}
 (d+z_1 + \frac{1}{2a})e^{-a(2dn_1 + z_1 -x_1)} - \frac{e^{-a(2d(n_1+1)-z_1-x_1))}}{2a} &\text{ if } n_1 >0\\
   ( |z_1-x_1| + \frac{1}{a})e^{-a|z_1-x_1|} - \frac{e^{-2ad}}{2a}(e^{-a(x_1+z_1)} + e^{a(x_1+z_1)}) ~ &\text{ if } n_1 =0.
\end{cases}
\end{align}

Finally, for all $\alpha, \beta \in \mathbb{R}$, we have that
\begin{equation}\label{eq : exponential_inf}
    \int_{\mathbb{R}} e^{-a|y_1-\alpha|}e^{-a|y_1-\beta|} dy_1 = 
   ( |\alpha - \beta| + \frac{1}{a})e^{-a|\alpha-\beta|}.
\end{equation}
\end{lemma}

\begin{proof}
Let us first prove \eqref{eq : exponential_d}. Suppose first that $n_1>0$, then
\[
\int_{-d}^d e^{-a|y_1-x_1|}e^{-a|y_1-z_1-2dn_1|} dy_1 = \int_{-d}^d e^{-a|y_1-x_1|}e^{-a(2dn_1 + z_1 - y_1)} dy_1
\]
as $2dn_1 > y_1-z_1$ for all $y_1 \in (-d,d)$ as $z_1 \in (-d,d).$ Therefore, we get
\begin{align}\label{eq : Z12 term n neq 0}
\nonumber
    \int_{-d}^d e^{-a|y_1-x_1|}e^{-a|y_1-z_1-2dn_1|} dy_1  &= \int_{-d}^{x_1} e^{-a(x_1-y_1)}e^{-a(2dn_1 + z_1 - y_1)} dy_1  \\ \nonumber
    & \quad + \int_{x_1}^{d} e^{-a(y_1-x_1)}e^{-a(2dn_1 + z_1 - y_1)} dy_1\\ \nonumber
    &=  \frac{1}{2a}e^{-a(2dn_1 + z_1 + x_1)}\left(e^{2ax_1} - e^{-2ad}\right) + (d-x_1)e^{-a(2dn_1+z_1-x_1)}\\
    & = e^{-a(2dn_1+z_1-x_1)}\left(d-x_1 + \frac{1}{2a}\right) - \frac{e^{-a(2d(n_1+1)+z_1+x_1)}}{2a}.
\end{align}
Now if $n_1 = 0$, we first suppose that $z_1 \geq x_1$, then
\begin{align}\label{one_time_used}
\nonumber
    &\hspace{-.5cm} \int_{-d}^d e^{-a|y_1-x_1|}e^{-a|y_1-z_1-2dn_1|} dy_1 \\\nonumber
    = & ~ \int_{-d}^{x_1} e^{-a(x_1-y_1)}e^{-a(z_1 - y_1)} dy_1  + \int_{x_1}^{z_1} e^{-a(y_1-x_1)}e^{-a(z_1 - y_1)} dy_1 + \int_{z_1}^{d} e^{-a(y_1-x_1)}e^{-a( y_1 - z_1)} dy_1\\\nonumber
    = & ~  \frac{e^{-a(z_1 + x_1)}}{2a}\left(e^{2ax_1} - e^{-2ad}\right) + (z_1-x_1)e^{-a(z_1-x_1)} + \frac{e^{a(z_1 + x_1)}}{2a}\left(e^{-2az_1} - e^{-2ad}\right)  \\
    = & ~ e^{-a(z_1-x_1)}\left(z_1-x_1 + \frac{1}{a}\right) + \frac{e^{-2ad}}{2a}\left(e^{a(z_1+x_1) + e^{-a(z_1+x_1)}}\right).
\end{align}
A similar reasoning can be applied when $x_1 \geq z_1$ and we obtain \eqref{eq : exponential_d}.

In order to prove \eqref{eq : exponential_dn}, notice that 
\begin{align*}
    \int_{2dn_1-d}^{2dn_1+d} e^{-a|y_1-x_1|}e^{-a|y_1-z_1-2dn_1|} dy_1 &= \int_{-d}^{d} e^{-a|y_1 + 2dn_1 -x_1|}e^{-a|y_1-z_1|} dy_1 \\
    &= \int_{-d}^{d} e^{-a|-y_1 + 2dn_1 -x_1|}e^{-a|y_1+z_1|} dy_1 \\
    &= \int_{-d}^{d} e^{-a|y_1 + x_1 -  2dn_1|}e^{-a|y_1+z_1|} dy_1
\end{align*}
where we used the changes of variable $y_1 \mapsto y_1+2dn_1$ and $y_1 \mapsto -y_1$. In particular, the previous computations in \eqref{eq : Z12 term n neq 0} can be used where $x_1$ becomes $-z_1$ and $z_1$ becomes $-x_1$. This proves \eqref{eq : exponential_dn}.

Finally, to prove \eqref{eq : exponential_inf}, we use \eqref{one_time_used} and take the limit as $d \to \infty$. This concludes the proof.
\end{proof}

Now let $u \in L^2_{D_2}$ such that $\|u\|_2=1$ and define $v \bydef v_0^N u$ where $v_0^N$ is defined in \eqref{eq : v0N}. Then, recalling \eqref{eq : proven in the appendix}, we need to verify that the constants 
$C_1(d), C_{12}(d)$ and  $C_2(d) >0$ given in \eqref{def : definition of the constants Ci} satisfy
\begin{align}\label{eq : objective in appendix}
    &\hspace{-3cm} \sum_{n \in \mathbb{N}_0^2, n \neq 0} \alpha_n  \int_{\mathbb{R}^2\setminus (\om \cup (\om +2dn))} \int_\om \int_\om e^{-a|y-x|_1}e^{-a|y-2dn-z|_1}|v(x)v(z)|dxdzdy\\\nonumber
    \leq & ~  e^{-4ad}|\om| \left(V_0^N,V_0^N*\left[C_1(d)E_1+ C_{12}(d)E_{1,2} + C_{2}(d)E_2\right]\right)_2,
\end{align}
where  ${E_1}, {E_{1,2}}$ and ${E_2}$ are sequences in $\ell^2_{D_2}$ defined in \eqref{def : definition of the sequences Ei}.  
Using Fubini's theorem on the left-hand side of \eqref{eq : objective in appendix}, our goal is to compute an upper bound for 
\[
\int_{\om}\int_{\om} |v(x)v(z)|\left(\int_{\mathbb{R}^2\setminus (\om \cup (\om +2dn))} e^{-a|y-x|_1}e^{-a|y-2dn-z|_1}dy\right) dz dx
\]
for the cases $n_1, n_2>0$ (Section~\ref{sec: cases_for_n_1}), $n_1 >0, n_2 = 0$ (Section~\ref{sec : n_1>0, n_2 = 0}) and $n_1 = 0, n_2 >0$ (Section~\ref{sec : n_1=0, n_2 > 0}). Before presenting our analysis for each case, we first introduce some notations.
Let $n \in \mathbb{N}_0^2$, $x, z \in \Omega_0$, and denote
\begin{equation}\label{def : definition_Ii_Ji}
\begin{aligned}
I^1_n(x_1,z_1) &\bydef  \int_{\mathbb{R}}e^{-a|y_1-x_1|}e^{-a|y_1-z_1-2dn_1|}dy_1\\ 
    I^2_n(x_1,z_1) &\bydef \int_{-d}^d e^{-a|y_1-x_1|}e^{-a|y_1-z_1-2dn_1|} dy_1\\ 
    I^3_n(x_1,z_1) &\bydef \int_{2dn_1-d}^{2dn_1+d} e^{-a|y_1-x_1|}e^{-a|y_1-z_1-2dn_1|} dy_1\\ 
    J^1_n(x_2,z_2) &\bydef \int_{\mathbb{R}}e^{-a|y_2-x_2|}e^{-a|y_2-z_2-2dn_2|}dy_2\\ 
    J^2_n(x_2,z_2) &\bydef \int_{-d}^d e^{-a|y_2-x_2|}e^{-a|y_2-z_2-2dn_2|}dy_2\\ 
    J^3_n(x_2,z_2) &\bydef \int_{2dn_2-d}^{2dn_2+d} e^{-a|y_2-x_2|}e^{-a|y_2-z_2-2dn_2|}dy_2.
\end{aligned}
\end{equation}
To simplify notations, we will drop the dependency in $x$ and $z$ when no confusion arises. Moreover, we introduce three sequences $(p_n)_{n \in \mathbb{N}_0^2}, (q_n)_{n \in \mathbb{N}_0^2}$ and $(s_n)_{n \in \mathbb{N}_0^2}$ defined as 
\begin{equation}\label{eq : definition of pn qn and sn}
\begin{aligned}
    p_n &\bydef \begin{cases}
        4e^{-2ad(n_1+n_2-1)} \left(2d(n_2-1) + \frac{1}{a}\right)\left(2d(n_1-1)+ \frac{e^{-2ad}}{a}\right) &\text{ if } n_1, n_2 >0\\
         \frac{2e^{-2adn_1}}{a}\left(2d(n_1-1)+ \frac{e^{-2ad}}{a}\right) &\text{ if } n_1 >0, n_2 =0\\
         0 &\text{ otherwise}
    \end{cases}\\
    q_n &\bydef \begin{cases}
        8 e^{-2ad(n_1+n_2)} \left(2n_2d + \frac{1 +e^{-2ad}}{2a}\right)\left(2d + \frac{1}{2a}\right) &\text{ if } n_1, n_2 >0\\
        \frac{4e^{-2ad(n_1+1)}}{a}\left(2d+\frac{1}{a}\right)  &\text{ if } n_1 >0, n_2 =0\\
        q_n  =  \frac{4e^{-2ad(n_2+1)}}{a}\left(2d+\frac{1}{a}\right) &\text{ if } n_1 =0, n_2 >0\\
        0 &\text{ if } n_1 = n_2 = 0
    \end{cases}\\
    s_n &\bydef \begin{cases}
         \frac{2e^{-2adn_2}}{a}\left(2d(n_2-1)+ \frac{e^{-2ad}}{a}\right) &\text{ if } n_1 =0, n_2 >0\\
         0 &\text{ otherwise}.
    \end{cases}
\end{aligned}
\end{equation}

\subsection{Case \boldmath$n_1, n_2 >0 $} \label{sec: cases_for_n_1}
Let $n_1, n_2>0$, then using \eqref{def : definition_Ii_Ji}, we get
\begin{align*}
     \displaystyle \int_{\mathbb{R}^2\setminus (\om \cup (\om +2dn))} e^{-a|y-x|_1}e^{-a|y-z-2dn|_1}dy 
  &= I^1_nJ^1_n - I^2_nJ^2_n - I^3_nJ^3_n\\
  &= (I^1_n - I^2_n - I^3_n)J^1_n + (J^1_n-J^2_n)I^2_n + (J^1_n-J^3_n)I^3_n.
\end{align*}
Then, using Lemma \ref{lem : exponential_computations}, we obtain
\begin{equation}\label{eq : computation_I1_J1}
  I_n^1 = \left(2dn_1 + z_1-x_1 + \frac{1}{a}\right)e^{-a(2dn_1+z_1-x_1)} \text{ and } J_n^1 = \left(2dn_2 + z_2-x_2 + \frac{1}{a}\right)e^{-a(2dn_2+z_2-x_2)}  
\end{equation}

as $2dn_1 \geq x_1- z_1$ and $2dn_2 \geq x_2 - z_2$. Given a fixed $\alpha \geq 0$, let $h : \mathbb{R}^+ \to \mathbb{R}^+$ be defined as 
\[
h(r)  = \left( \alpha + r + \frac{1}{a}\right)e^{-a(\alpha + r)}.
\]
Then 
\[
h'(r) = e^{-a(\alpha +r)}\left(1  - a( \alpha + r + \frac{1}{a} )\right) = -a e^{-a(\alpha +r)}(\alpha + r) \leq  0
\]
for all $r \geq 0.$ In particular, $h$ has a global minimum at $r = 0$. Denoting $r = 2d + z_2-x_2$ and $\alpha = 2d(n_2-1)$, we obtain that
\begin{equation}\label{eq : inequality_J1}
    J_n^1 = h(r) \leq h(0) =  \left(2d(n_2-1) + \frac{1}{a}\right)e^{-2ad(n_2-1)}.
\end{equation}

Therefore, using Lemma \ref{lem : exponential_computations} and \eqref{eq : inequality_J1} we get
\begin{align}\label{eq : computation_Ii_Ji_n>0}
\nonumber
    I^1_n - I^2_n - I^3_n &= 2d(n_1-1)e^{-a(2dn_1+z_1-x_1)} + \frac{e^{-2ad(n_1+1)}}{2a}\left(e^{a(x_1+z_1)} + e^{-a(x_1+z_1)}\right)\\ \nonumber
    J^1_n &  \leq \left(2d(n_2-1) + \frac{1}{a}\right)e^{-2da(n_2-1)}\\ \nonumber
    J^1_n - J^2_n  
   &= \left(d(2n_2-1) + z_2 + \frac{1}{2a}\right)e^{-a(2dn_2+z_2-x_2)} + \frac{e^{-a(2d(n_2+1)+z_2+x_2)}}{2a}\\
       J^1_n - J^3_n  
   & = \left(d(2n_2-1) - x_2 + \frac{1}{2a}\right)e^{-a(2dn_2+z_2-x_2)} + \frac{e^{-a(2d(n_2+1)-z_2-x_2)}}{2a}.
\end{align}

Consequently, we use \eqref{eq : computation_Ii_Ji_n>0} to get
\begin{align}
\nonumber
 \int_{\om} \int_{\om} v(x)v(z)(I^1_n - I^2_n - I^3_n) dxdz
\leq  &~  2d(n_1-1)e^{-2adn_1}\int_{\om}v(x)e^{ax_1}dx\int_{\om}v(z)e^{-az_1} dz\\\nonumber
 +  &~\frac{e^{-2ad(n_1+1)}}{2a}\int_{\om}\int_{\om}v(x)v(z)\left(e^{a(x_1+z_1)} + e^{-a(x_1+z_1)}\right)dxdz
\end{align}
But using that $v \in L_{D_2}^2$, we simplify
\begin{align*}
    & \int_{\om} \int_{\om} v(x)v(z)(I^1_n - I^2_n - I^3_n) dxdz\\
\leq  &  ~2d(n_1-1)e^{-2adn_1}\left(\int_{\om}v(x)e^{ax_1}dx\right)^2
 +  \frac{e^{-2ad(n_1+1)}}{a}\left(\int_{\om}v(x)e^{ax_1}dx\right)^2,
\end{align*}
which implies that
\begin{equation}\label{before_cauchy_Schwarz_1}
    \int_{\om} \int_{\om} v(x)v(z)(I^1_n - I^2_n - I^3_n) dxdz
\leq  \left({2d(n_1-1)e^{-2adn_1}} + \frac{e^{-2ad(n_1+1)}}{a}\right)\left(\int_{\om}v(x)e^{ax_1}dx\right)^2.
\end{equation}

Now, using Cauchy-Schwarz inequality, we obtain
\begin{align}\label{after_cauchy_Schwarz_1}
   \left(\int_{\om}v(x)e^{ax_1}dx\right)^2 
    \leq \int_{\om} |u(x)|^2dx \int_{\om} v_0^N(x)^2e^{2ax_1}dx \leq \int_{\om} v_0^N(x)^2e^{2ax_1}dx
\end{align}
using that $v = v_0^Nu$ by definition and $\|u\|_2=1$. Moreover, as $v_0^N$ is $D_2$-symmetric, we have
\begin{align}\label{eq : exp is cosh}
   \int_{\om} v_0^N(x)^2e^{2ax_1}dx = \int_{\om} v_0^N(x)^2\cosh(2ax_1)dx = |\om|(V^N_0,V^N_0*E_1)_2,
\end{align}
where we used Parseval's identity for the last step. Therefore, combining \eqref{eq : inequality_J1}, \eqref{before_cauchy_Schwarz_1} and  \eqref{after_cauchy_Schwarz_1}, we get
\begin{equation}
    \int_{\om} \int_{\om} v(x)v(z)(I^1_n - I^2_n - I^3_n)J^1_n dxdz
\leq  \frac{p_n}{4}|\om|(V^N_0,V^N_0*E_1)_2
\end{equation}
where $p_n$ is defined in \eqref{eq : definition of pn qn and sn}. Now,  notice that $I^2_n \leq \left(d-x_1+\frac{1}{2a}\right)e^{-a(2dn_1+z_1-x_1)}$, then using \eqref{eq : computation_Ii_Ji_n>0} and $|z_2|,|x_1| \leq d$  for all $z_2, x_1 \in (-d,d)$, we get
\begin{align}\label{one_time_(J_1-J_2)}
\nonumber
    & \hspace{-1cm} \int_{\om} \int_{\om} v(x)v(z)(J^1_n - J^2_n)I^2_n dxdz\\\nonumber
\leq  &  ~ \int_{\om}\int_{\om} v(z)v(x)\left(2n_2d + \frac{1}{2a}\right)e^{-a(2dn_2+z_2-x_2)}\left(2d+\frac{1}{2a}\right)e^{-a(2dn_1+z_1-x_1)}dxdz\\\nonumber
& + \int_{\om}\int_{\om} v(z)v(x)\frac{e^{-a(2d(n_2+1)+z_2+x_2)}}{2a}\left(2d+\frac{1}{2a}\right)e^{-a(2dn_1+z_1-x_1)}dxdz\\\nonumber
=  &  ~ \left(2n_2d + \frac{1}{2a}\right)\left(2d + \frac{1}{2a}\right)e^{-2ad(n_1+n_2)}\int_{\om}v(z)e^{-a(z_2+z_1)}dz\int_{\om} v(x)e^{a(x_2+x_1)}dx\\
& + \left(2d + \frac{1}{2a}\right)\frac{e^{-2ad(n_1+n_2 +1)}}{2a}\int_{\om}v(z)e^{-a(z_2+z_1)}dz\int_{\om} v(x)e^{a(x_1-x_2)}dx.
\end{align}

Using again that $v \in L^2_{D_2}$ and Cauchy-Schwarz inequality, we obtain
\begin{align}\label{before_CS_2}
\nonumber
    & \hspace{-1cm} \int_{\om} \int_{\om} v(x)v(z)(J^1_n - J^2_n)I^2_n dxdz\\\nonumber
\leq ~ &e^{-2ad(n_1+n_2)} \left[ \left(2n_2d + \frac{1}{2a}\right)\left(2d + \frac{1}{2a}\right) + \frac{e^{-2ad}}{2a}\left(2d + \frac{1}{2a}\right)\right]\left(\int_{\om}v(z)e^{a(z_2+z_1)}dz\right)^2\\\nonumber
\leq ~ &e^{-2ad(n_1+n_2)} \left(2n_2d + \frac{1 +e^{-2ad}}{2a}\right)\left(2d + \frac{1}{2a}\right) \left(\int_{\om}v(z)e^{a(z_2+z_1)}dz\right)^2\\
\leq ~ &e^{-2ad(n_1+n_2)} \left(2n_2d + \frac{1 +e^{-2ad}}{2a}\right)\left(2d + \frac{1}{2a}\right) \int_{\om}v_0^N(z)^2e^{2a(z_2+z_1)}dz.
\end{align}


Moreover, as in \eqref{eq : exp is cosh}, we have
\begin{align*}
    \int_{\om}v_0^N(z)^2e^{2a(z_2+z_1)}dz = \int_{\om}v_0^N(z)^2\cosh(2az_1)\cosh(2az_2)dz = |\om| (V^N_0, V^N_0*E_{1,2})_2.
\end{align*}
Therefore we obtain
\begin{equation}\label{eq : result_J1-J2}
    \int_{\om} \int_{\om} v(x)v(z)(J^1_n - J^2_n)I^2_n dxdz \leq \frac{q_n}{8}|\om| (V^N_0, V^N_0*E_{1,2})_2
\end{equation}
where $q_n$ is defined in \eqref{eq : definition of pn qn and sn}. Similarly, using \eqref{eq : computation_Ii_Ji_n>0} we get
\[
\int_{\om} \int_{\om} v(x)v(z)(J^1_n - J^3_n)I^3_n dxdz \leq \frac{q_n}{8}|\om| (V^N_0, V^N_0*E_{1,2})_2
\]
using \eqref{eq : result_J1-J2}. Summarizing the results of the section, for all $n_1, n_2 >0$ we have that
\begin{align*} 
    &\hspace{-1cm} \int_{\om}\int_{\om} |v(x)v(z)|\left(\int_{\mathbb{R}^2\setminus (\om \cup (\om +2dn))} \hspace{-2cm} e^{-a|y-x|_1}e^{-a|y-2dn-x|_1}dy\right) dz dx \\ & \qquad  \leq \frac{p_n}{4}|\om|(V^N_0,V^N_0*E_1)_2 + \frac{q_n}{4}|\om| (V^N_0, V^N_0*E_{1,2})_2.
\end{align*}




\subsection{ Case \boldmath$n_1>0, n_2 = 0$}\label{sec : n_1>0, n_2 = 0}
Let $n_1 >0, n_2 = 0$, then $J^2_n = J^3_n$ and 
\begin{align}
\nonumber
      \displaystyle \int_{\mathbb{R}^2\setminus (\om \cup (\om +2dn))} e^{-a|y-x|_1}e^{-a|y-z-2dn|_1}dy& = I^1_nJ^1_n - I^2_nJ^2_n - I^3_nJ^2_n\\
      &= (I^1_n - I^2_n - I^3_n)J^1_n + (J^1_n-J^2_n)(I^2_n+I^3_n). 
\end{align}

First, notice that $J^1_n \leq \frac{1}{a}$, then 
\begin{align*}
\nonumber
 \int_{\om} \int_{\om} v(x)v(z)(I^1_n - I^2_n - I^3_n)J^1_n dxdz
\leq  \frac{1}{a}\int_{\om} \int_{\om} v(x)v(z)(I^1_n - I^2_n - I^3_n) dxdz.
\end{align*}

Then, using \eqref{before_cauchy_Schwarz_1} and  \eqref{after_cauchy_Schwarz_1}, we get
\begin{equation}
    \int_{\om} \int_{\om} v(x)v(z)(I^1_n - I^2_n - I^3_n)J^1_n dxdz \leq \frac{p_n}{2}|\om|(V^N_0,V^N_0*E_1)_2
\end{equation}
where $p_n$ is defined in \eqref{eq : definition of pn qn and sn}. Then, we use Lemma \ref{lem : exponential_computations} to get
\begin{align}\label{onetimZ12}
    J^1_n - J^2_n  = \frac{e^{-2ad}}{2a}\left(e^{-a(x_2+z_2)} + e^{a(x_2+z_2)}\right)
\end{align}
and 
\begin{align}\label{onetimZ11}
    I_n^2 \leq (d-x_1 + \frac{1}{2a})e^{-a(2dn_1+z_1-x_1)} \text{ and } I_n^3 \leq (d+z_1 + \frac{1}{2a})e^{-a(2dn_1+z_1-x_1)}.
\end{align}

Using \eqref{onetimZ12} and \eqref{onetimZ11}, we get
 \begin{align*}
     &\int_{\om} \int_{\om} v(x)v(z)(J^1_n - J^2_n)(I^2_n+I^3_n) dxdz\\
    \leq & \frac{e^{-2ad(n_1+1)}}{2a} \int_{\om} \int_{\om} v(x)v(z)\left(e^{-a(x_2+z_2)} + e^{a(x_2+z_2)}\right)(d+ z_1 + \frac{1}{2a})e^{-a( z_1 -x_1)}\\
    & + \frac{e^{-2ad(n_1+1)}}{2a} \int_{\om} \int_{\om} v(x)v(z)\left(e^{-a(x_2+z_2)} + e^{a(x_2+z_2)}\right)(d-x_1 + \frac{1}{2a})e^{-a(z_1 -x_1)}.
 \end{align*}

Therefore, similarly as in \eqref{one_time_(J_1-J_2)} and \eqref{before_CS_2}, we use that $v \in L^2_{D_2}$ and  the Cauchy-Schwarz inequality to obtain
\begin{equation*}
    \int_{\om} \int_{\om} v(x)v(z)(J^1_n - J^2_n)(I^2_n+I^3_n) dxdz \leq \frac{q_n}{2}|\om| (V^N_0, V^N_0*E_{1,2})_2
\end{equation*}
where $q_n$ is given in \eqref{eq : definition of pn qn and sn}.

\subsection{ Case \boldmath$n_1 = 0, n_2 > 0$} \label{sec : n_1=0, n_2 > 0}
Similarly, as in Section \ref{sec : n_1>0, n_2 = 0}, we have
\begin{align}
\nonumber
  \displaystyle \int_{\mathbb{R}^2\setminus (\om \cup (\om +2dn))} e^{-a|y-x|_1}e^{-a|y-z-2dn|_1}dy &= I^1_nJ^1_n - I^2_nJ^2_n - I^2_nJ^3_n\\ \nonumber
  &= (J^1_n - J^2_n - J^3_n)I^1_n + (I^1_n-I^2_n)(J^2_n+J^3_n).
\end{align}

In fact, the needed computations are identical as the ones  of Section \ref{sec : n_1>0, n_2 = 0} interchanging the subscripts 1 and 2. In particular, we obtain that
\begin{equation*}
    \int_{\om} \int_{\om} v(x)v(z)(J^1_n - J^2_n - J^3_n)I^1_n dxdz \leq \frac{s_n}{2}|\om| (V^N_0, E_2*V^N_0)_2
\end{equation*}
where $s_n$ is defined in \eqref{eq : definition of pn qn and sn}. Similarly,
\begin{equation*}
    \int_{\om} \int_{\om} v(x)v(z)(I^1_n - I^2_n)(J^2_n+J^3_n) dxdz \leq \frac{q_n}{2}|\om| (V^N_0, V^N_0*E_{1,2})_2.
\end{equation*}

\subsection{Summary and computation of \boldmath$C_{1}(d), C_{12}(d)$ and \boldmath$C_2(d)$}
Combining the results from Sections \ref{sec: cases_for_n_1}, \ref{sec : n_1>0, n_2 = 0} and  \ref{sec : n_1=0, n_2 > 0}, we get
\begin{align*}\label{summary proof Z1}
      &\hspace{-.4cm} \sum_{n \in \mathbb{N}^2_0, n \neq 0} \alpha_n \int_{\mathbb{R}^2\setminus (\om \cup (\om +2dn))} \mathbb{L}^{-1} v(y) \mathbb{L}^{-1}  v(y-2dn)   dy \\
      &\le C_0^2|\om|\left((V_0^N,E_1*V^N_0)_2\sum_{n \in \mathbb{N}^2_0}p_n  + (V_0^N,E_{1,2}*V^N_0)_2 \sum_{n \in \mathbb{N}^2_0}q_n +  (V_0^N,E_2*V^N_0)_2 \sum_{n \in \mathbb{N}^2_0}s_n\right) 
\end{align*}
where $p_n,~ q_n$ and $s_n$ are defined in \eqref{eq : definition of pn qn and sn}. Consequently, it remains to compute upper bounds for $\sum_{n \in \mathbb{N}^2_0}p_n$, $\sum_{n \in \mathbb{N}^2_0}q_n$ and $\sum_{n \in \mathbb{N}^2_0}s_n$. The following lemma provides upper bounds for such quantities. The bounds are explicit and decay in $e^{-4ad}.$ This allows us to define explicitly the constants $C_{1}(d), C_{12}(d)$ and $C_2(d)$ satisfying \eqref{eq : objective in appendix}.

\begin{lemma}\label{lem : computation of C(d)}
Let $C_{1}(d), C_{12}(d), C_2(d) >0$ be defined in \eqref{def : definition of the constants Ci}, then 
\begin{align}
     \sum_{n \in \mathbb{N}^2_0}p_n \leq  C_{1}(d)e^{-4ad}, ~~
    \sum_{n \in \mathbb{N}^2_0}q_n \leq  C_{12}(d)e^{-4ad}, ~~
     \sum_{n \in \mathbb{N}^2_0}s_n \leq  C_2(d)e^{-4ad}.
\end{align}
\end{lemma}

\begin{proof}
Recall that
\[
p_n = 4e^{-2ad(n_1+n_2-1)} \left(2d(n_2-1) + \frac{1}{a}\right)\left(2d(n_1-1)+ \frac{e^{-2ad}}{a}\right) 
\]
for all $n_1, n_2 >0.$
Let us define $\tilde{p}_n \bydef e^{-2ad(n_1+n_2-1)} \left(2d(n_2-1) + \frac{1}{a}\right)\left(2d(n_1-1)+ \frac{e^{-2ad}}{a}\right)$.
One can easily prove that 
\[
xe^{-ax} \leq \frac{e^{-1}}{a}
\]
for all $x\geq 0$. Therefore,
\begin{equation}\label{eq : estimate tilde a1}
 e^{-ad(n_2-1)}(2d(n_2-1)+\frac{1}{a}) \leq \frac{2e^{-1}+1}{a}   
\end{equation}
for all $n_2 \geq 1$. Similarly, 
\begin{equation}\label{eq : estimate tilde a2}
 e^{-adn_1}\left((2d(n_1-1)+\frac{1}{a}\right) \leq \frac{2e^{-1}+1}{a} \end{equation}

for all $n_1 \geq 1$, which implies that 
\begin{equation}
    \tp_n \leq \frac{(2e^{-1}+1)^2}{a^2}e^{-ad(n_1+n_2-1)}
\end{equation}
for all $n_1,n_2 \geq 1.$ Therefore, 
\begin{equation}
    \sum_{n_1>0, n_2>0}\tp_n 
    =  \tp_{1,1} + \tp_{2,1} + \tp_{3,1} +  \sum_{n_1=1}^{3}\sum_{n_2=2}^\infty \tp_n + \sum_{n_1=4}^{\infty}\sum_{n_2=1}^\infty \tp_n.
\end{equation}

We readily have 
\[
\tp_{1,1} + \tp_{2,1} + \tp_{3,1} = \frac{e^{-4ad}}{a}\left(2d+\frac{1+e^{-2ad}}{a} + e^{-2ad}\left(4d + \frac{e^{-2ad}}{a}\right)\right)
\]
and 
\[
\sum_{n_1=1}^{2}\sum_{n_2=2}^\infty \tp_n  = e^{-4ad}\left(\frac{1+e^{-2ad}}{a} + 2d
    \right)\sum_{n_2=1}^\infty e^{-2ad(n_2-1)}\left(2d(n_2-1)+\frac{1}{a}\right).
\]

Therefore, using \eqref{eq : estimate tilde a2}, we obtain
\begin{align*}
   \sum_{n_1=1}^{2}\sum_{n_2=2}^\infty \tp_n  &= e^{-4ad}\left(\frac{1+e^{-2ad}}{a} + 2d
    \right)\sum_{n_2=1}^\infty e^{-2ad(n_2-1)}\left(2d(n_2-1)+\frac{1}{a}\right) \\
    & \leq e^{-4ad}\left(\frac{1+e^{-2ad}}{a} + 2d
    \right)\frac{2e^{-1}+1}{a(1-e^{-ad})}.
\end{align*}

Similarly, using \eqref{eq : estimate tilde a1} and \eqref{eq : estimate tilde a2}, 
\begin{align*}
    \sum_{n_1=4}^{\infty}\sum_{n_2=1}^\infty \tp_n \leq  \frac{(2e^{-1}+1)^2}{a^2}\sum_{n_1=4}^{\infty}\sum_{n_2=1}^\infty e^{-ad(n_1+n_2-1)}
    =  e^{-4ad}\frac{(2e^{-1}+1)^2}{a^2(1-e^{-ad})^2}.
\end{align*}

Therefore we obtain
\begin{equation*}
     \sum_{n_1>0, n_2>0}p_n \leq 4P_{1}e^{-4ad} 
\end{equation*}
where
\begin{equation}\label{def : b^1_0}
    P_{1} \bydef  \frac{2ad+1+e^{-2ad}}{a^2} + e^{-2ad}\left(4d + \frac{e^{-2ad}}{a}\right) + \left(\frac{1+e^{-2ad}}{a} + 2d
    \right)\frac{2e^{-1}+1}{a(1-e^{-ad})} +  \frac{(2e^{-1}+1)^2}{a^2(1-e^{-ad})^2}.
\end{equation}

Now recall that
\[
    p_n  =  \frac{2e^{-2adn_1}}{a}\left(2d(n_1-1)+ \frac{e^{-2ad}}{a}\right)
\]
for all $n_1 >0 , n_2 = 0.$ 

Therefore, using \eqref{eq : estimate tilde a1},
\begin{align*}
    \sum_{n_1 =1}^\infty p_n &= p_{1,0} + p_{2,0} + p_{3,0} + \sum_{n_1 =4}^\infty p_n\\
    &\leq \frac{2e^{-4ad}}{a}\left( \frac{1+e^{-2ad}}{a} + 2d + e^{-2ad}(4d+\frac{e^{-2ad}}{a})+ \frac{2e+1}{a}\sum_{n_1=4}^\infty e^{-ad(n_1-4)}\right)\\
    & = \frac{2e^{-4ad}}{a}\left( \frac{1+e^{-2ad}}{a} + 2d + e^{-2ad}(4d+\frac{e^{-2ad}}{a})+ \frac{2e^{-1}+1}{a(1-e^{-ad})}\right).
\end{align*}

Finally, we obtain that
\begin{equation*}
     \sum_{n \in \mathbb{N}^2_0}p_n \leq \left(4P_{1} + P_2\right)e^{-4ad} \leq C_{1}(d) e^{-4ad}
\end{equation*}
where
\begin{align}\label{def : b_{1}}
    P_2 &\bydef \frac{2}{a}\left( \frac{1+e^{-2ad}}{a} + 2d + e^{-2ad}(4d+\frac{e^{-2ad}}{a})+ \frac{2e^{-1}+1}{a(1-e^{-ad})}\right)
\end{align}
and $P_1$ is defined in \eqref{def : b^1_0}. 
Now, recall that 
\begin{align*}
    q_n &\bydef 8 e^{-2ad(n_1+n_2)} \left(2n_2d + \frac{1 +e^{-2ad}}{2a}\right)\left(2d + \frac{1}{2a}\right)
\end{align*}
for all $n_1,n_2 >0$. Using \eqref{eq : estimate tilde a1}, we obtain that 
\begin{align*}
    q_n \leq 8 \left(2d + \frac{1}{2a}\right) e^{-ad(2n_1+n_2)}\frac{4e^{-1} + 1 +e^{-2ad}}{2a} 
\end{align*}
for all $n_1,n_2 >0$. Therefore,
\begin{align*}
    \sum_{n_1>0, n_2>0}q_n 
    &= c_{1,1} +  \sum_{n_1=2}^\infty c_{n_1,1} + \sum_{n_1=1}^{\infty}\sum_{n_2=2}^\infty q_n\\
    &\leq   8 \left(2d + \frac{1}{2a}\right) e^{-4ad}\left[2d + \frac{1+e^{-2ad}}{2a} + (2d + \frac{1+e^{-2ad}}{2a})\frac{1}{1-e^{-ad}} +  \frac{4e^{-1} + 1 +e^{-2ad}}{2a(1-e^{-ad})^2}  \right].
\end{align*}

This implies that 
\begin{equation*}
    \sum_{n_1>0, n_2>0}q_n \leq 8e^{-4ad}Q_1\left(2d + \frac{1}{2a}\right)
\end{equation*}
where 
\begin{equation}\label{def : b_{12}^0}
Q_1 \bydef 2d + \frac{1+e^{-2ad}}{2a} + (2d + \frac{1+e^{-2ad}}{2a})\frac{1}{1-e^{-ad}} +  \frac{4e^{-1} + 1 +e^{-2ad}}{2a(1-e^{-ad})^2}.
\end{equation}

Moreover, recall that
\begin{align*}
q_n  =  \frac{4e^{-2ad(n_1+1)}}{a}\left(2d+\frac{1}{a}\right)
\end{align*}
for all $n_1 >0, n_2=0$ and similarly,
\begin{align*}
q_n  =  \frac{4e^{-2ad(n_2+1)}}{a}\left(2d+\frac{1}{a}\right)
\end{align*}
for all $n_1 =0, n_2>0.$
Therefore,
\begin{equation*}
    \sum_{n_1=1}^{\infty }q_{n_1,0} +  \sum_{n_2=1}^{\infty }q_{0,n_2} \leq Q_2e^{-4ad}
\end{equation*}
where 
\begin{equation}\label{def : b_{12}^1}
    Q_2 \bydef \left(2d + \frac{1}{a}\right)\frac{8}{a(1-e^{-2ad})}.
\end{equation}

Finally, we obtain
\begin{equation*}
    \sum_{n \in \mathbb{N}^2_0}q_n \leq 8Q_1\left(2d + \frac{1}{2a}\right) + Q_2e^{-4ad} \leq C_{12}(d) e^{-4ad}
\end{equation*} 
where  $Q_1, Q_2$ are defined in \eqref{def : b_{12}^0} and \eqref{def : b_{12}^1} respectively. To conclude, we recall that
\begin{align*}
    s_n  &\bydef  \frac{2e^{-2adn_2}}{a}\left(2d(n_2-1)+ \frac{e^{-2ad}}{a}\right)
\end{align*}
for all $n_2>0, n_1 = 0$ and $s_n =0$ otherwise.
Therefore, using \eqref{eq : estimate tilde a1}, 
\begin{align*}
    \sum_{n_2=1}^\infty d_{0,n_2}& = s_{0,1} + s_{0,2} + s_{0,3} + \sum_{n_2=4}^\infty s_{0,n_2}\\
    &\leq   \frac{2e^{-4ad}}{a} \left[\frac{1+e^{-2ad}}{a} + 2d + e^{-2ad}(4d + \frac{e^{-2ad}}{a}) + \frac{(2e^{-1}+e^{-2ad})}{a(1-e^{-ad})} \right].
\end{align*}
This implies that
\begin{equation*}
    \sum_{n \in \mathbb{N}^2_0}s_n \leq C_2(d)e^{-4ad}. \qedhere
\end{equation*}
\end{proof}




\bibliographystyle{unsrt}
\bibliography{biblio}

\begin{thebibliography}{10}

\bibitem{julia_cadiot}
Matthieu Cadiot.
\newblock Localizedpatternsh.jl.
\newblock 2024.
\newblock \url{ https://github.com/matthieucadiot/LocalizedPatternSH.jl}.

\bibitem{Swift_original}
J.~Swift and P.~C. Hohenberg.
\newblock Hydrodynamic fluctuations at the convective instability.
\newblock {\em Phys. Rev. A}, 15:319--328, Jan 1977.

\bibitem{phase-crystal}
Lukas Ophaus, Svetlana~V. Gurevich, and Uwe Thiele.
\newblock Resting and traveling localized states in an active phase-field-crystal model.
\newblock {\em Phys. Rev. E}, 98(2):022608, 16, 2018.

\bibitem{GROVES20171}
M.~D. Groves, D.~J.~B. Lloyd, and A.~Stylianou.
\newblock Pattern formation on the free surface of a ferrofluid: spatial dynamics and homoclinic bifurcation.
\newblock {\em Phys. D}, 350:1--12, 2017.

\bibitem{Odent:2016gli}
V.~Odent, M.~Tlidi, M.~G. Clerc, and E.~Louvergneaux.
\newblock {Experimental Observation of Front Propagation in Lugiato-Lefever Equation in a Negative Diffractive Regime and Inhomogeneous Kerr Cavity}.
\newblock {\em Springer Proc. Phys.}, 173:71--85, 2016.

\bibitem{Knobloch2008open_problems}
E.~Knobloch.
\newblock Spatially localized structures in dissipative systems: open problems.
\newblock {\em Nonlinearity}, 21(4):T45--T60, 2008.

\bibitem{Knobloch2015spatial}
E.~Knobloch.
\newblock Spatial localization in dissipative systems.
\newblock {\em Annual Review of Condensed Matter Physics}, 6(1):325--359, 2015.

\bibitem{burke2007snakes}
John Burke and Edgar Knobloch.
\newblock Snakes and ladders: localized states in the {S}wift--{H}ohenberg equation.
\newblock {\em Physics Letters A}, 360(6):681--688, 2007.

\bibitem{avitabile2010snake}
Daniele Avitabile, David J.~B. Lloyd, John Burke, Edgar Knobloch, and Bj\"{o}rn Sandstede.
\newblock To snake or not to snake in the planar {S}wift-{H}ohenberg equation.
\newblock {\em SIAM J. Appl. Dyn. Syst.}, 9(3):704--733, 2010.

\bibitem{radial2019Bramburger}
Jason~J. Bramburger, Dylan Altschuler, Chloe~I. Avery, Tharathep Sangsawang, Margaret Beck, Paul Carter, and Bj\"{o}rn Sandstede.
\newblock Localized radial roll patterns in higher space dimensions.
\newblock {\em SIAM J. Appl. Dyn. Syst.}, 18(3):1420--1453, 2019.

\bibitem{burke2006localized}
John Burke and Edgar Knobloch.
\newblock Localized states in the generalized {S}wift-{H}ohenberg equation.
\newblock {\em Phys. Rev. E (3)}, 73(5):056211, 15, 2006.

\bibitem{budd2005localized}
C.~J. Budd and R.~Kuske.
\newblock Localized periodic patterns for the non-symmetric generalized {S}wift-{H}ohenberg equation.
\newblock {\em Phys. D}, 208(1-2):73--95, 2005.

\bibitem{hexagon2008}
David J.~B. Lloyd, Bj\"{o}rn Sandstede, Daniele Avitabile, and Alan~R. Champneys.
\newblock Localized hexagon patterns of the planar {S}wift-{H}ohenberg equation.
\newblock {\em SIAM J. Appl. Dyn. Syst.}, 7(3):1049--1100, 2008.

\bibitem{hexagon2021lloyd}
David~J. Lloyd.
\newblock Hexagon invasion fronts outside the homoclinic snaking region in the planar {S}wift-{H}ohenberg equation.
\newblock {\em SIAM J. Appl. Dyn. Syst.}, 20(2):671--700, 2021.

\bibitem{radial2009}
David Lloyd and Bj\"{o}rn Sandstede.
\newblock Localized radial solutions of the {S}wift-{H}ohenberg equation.
\newblock {\em Nonlinearity}, 22(2):485--524, 2009.

\bibitem{mccalla2013spots}
S.~G. McCalla and B.~Sandstede.
\newblock Spots in the {S}wift-{H}ohenberg equation.
\newblock {\em SIAM J. Appl. Dyn. Syst.}, 12(2):831--877, 2013.

\bibitem{squareSakaguchi_1997}
H.~Sakaguchi and H.~R. Brand.
\newblock Stable localized squares in pattern-forming nonequilibrium systems.
\newblock {\em Europhysics Letters}, 38(5):341, may 1997.

\bibitem{ladder2009Beck}
Margaret Beck, J\"{u}rgen Knobloch, David J.~B. Lloyd, Bj\"{o}rn Sandstede, and Thomas Wagenknecht.
\newblock Snakes, ladders, and isolas of localized patterns.
\newblock {\em SIAM J. Math. Anal.}, 41(3):936--972, 2009.

\bibitem{existence2019sandstede}
Elizabeth Makrides and Bj\"{o}rn Sandstede.
\newblock Existence and stability of spatially localized patterns.
\newblock {\em J. Differential Equations}, 266(2-3):1073--1120, 2019.

\bibitem{stability1997Mielke}
Alexander Mielke.
\newblock Instability and stability of rolls in the {S}wift-{H}ohenberg equation.
\newblock {\em Comm. Math. Phys.}, 189(3):829--853, 1997.

\bibitem{radial2023lessard}
Jan~Bouwe van~den Berg, Olivier H\'{e}not, and Jean-Philippe Lessard.
\newblock Constructive proofs for localised radial solutions of semilinear elliptic systems on {$\Bbb R^d$}.
\newblock {\em Nonlinearity}, 36(12):6476--6512, 2023.

\bibitem{unbounded_domain_cadiot}
M.~Cadiot, J.-P. Lessard, and J.-C. Nave.
\newblock Rigorous computation of solutions of semi-linear {PDE}s on unbounded domains via spectral methods.
\newblock {\em arXiv:2302.12877}, 2023.

\bibitem{Hill_2023}
Dan~J. Hill, Jason~J. Bramburger, and David J.~B. Lloyd.
\newblock Approximate localised dihedral patterns near a {T}uring instability.
\newblock {\em Nonlinearity}, 36(5):2567--2630, 2023.

\bibitem{champneys_2021}
{Alan R} Champneys and Nicolas {Verschueren van Rees}.
\newblock Dissecting the snake: Transition from localized patterns to spike solutions.
\newblock {\em Physica D: Nonlinear Phenomena}, 419, May 2021.

\bibitem{localized-rigorous}
Yasuaki Hiraoka and Toshiyuki Ogawa.
\newblock Rigorous numerics for localized patterns to the quintic {S}wift-{H}ohenberg equation.
\newblock {\em Japan J. Indust. Appl. Math.}, 22(1):57--75, 2005.

\bibitem{nakao_numerical}
Mitsuhiro~T. Nakao.
\newblock Numerical verification methods for solutions of ordinary and partial differential equations.
\newblock volume~22, pages 321--356. 2001.
\newblock International Workshops on Numerical Methods and Verification of Solutions, and on Numerical Function Analysis (Ehime/Shimane, 1999).

\bibitem{gomez_cap}
Javier G\'{o}mez-Serrano.
\newblock Computer-assisted proofs in {PDE}: a survey.
\newblock {\em SeMA J.}, 76(3):459--484, 2019.

\bibitem{jb_rigorous_dynmamics}
Jan~Bouwe van~den Berg and Jean-Philippe Lessard.
\newblock Rigorous numerics in dynamics.
\newblock {\em Notices Amer. Math. Soc.}, 62(9):1057--1061, 2015.

\bibitem{koch_computer_assisted}
Hans Koch, Alain Schenkel, and Peter Wittwer.
\newblock Computer-assisted proofs in analysis and programming in logic: a case study.
\newblock {\em SIAM Rev.}, 38(4):565--604, 1996.

\bibitem{plum_numerical_verif}
Mitsuhiro~T. Nakao, Michael Plum, and Yoshitaka Watanabe.
\newblock {\em Numerical verification methods and computer-assisted proofs for partial differential equations}, volume~53 of {\em Springer Series in Computational Mathematics}.
\newblock Springer, Singapore, [2019] \copyright 2019.

\bibitem{rigorous_global}
Sarah Day, Yasuaki Hiraoka, Konstantin Mischaikow, and Toshiyuki Ogawa.
\newblock Rigorous numerics for global dynamics: a study of the {S}wift-{H}ohenberg equation.
\newblock {\em SIAM J. Appl. Dyn. Syst.}, 4(1):1--31, 2005.

\bibitem{lessard-chaotic-rigorous}
Jan~Bouwe van~den Berg and Jean-Philippe Lessard.
\newblock Chaotic braided solutions via rigorous numerics: chaos in the {S}wift-{H}ohenberg equation.
\newblock {\em SIAM J. Appl. Dyn. Syst.}, 7(3):988--1031, 2008.

\bibitem{MR2718657}
Marcio Gameiro and Jean-Philippe Lessard.
\newblock Analytic estimates and rigorous continuation for equilibria of higher-dimensional {PDE}s.
\newblock {\em J. Differential Equations}, 249(9):2237--2268, 2010.

\bibitem{MR3077902}
Marcio Gameiro and Jean-Philippe Lessard.
\newblock Efficient rigorous numerics for higher-dimensional {PDE}s via one-dimensional estimates.
\newblock {\em SIAM J. Numer. Anal.}, 51(4):2063--2087, 2013.

\bibitem{MR4668616}
Jan~Bouwe van~den Berg, Jonathan Jaquette, and J.~D. Mireles~James.
\newblock Validated numerical approximation of stable manifolds for parabolic partial differential equations.
\newblock {\em J. Dynam. Differential Equations}, 35(4):3589--3649, 2023.

\bibitem{MR4379799}
Jacek Cyranka and Jean-Philippe Lessard.
\newblock Validated forward integration scheme for parabolic {PDE}s via {C}hebyshev series.
\newblock {\em Commun. Nonlinear Sci. Numer. Simul.}, 109:Paper No. 106304, 32, 2022.

\bibitem{berg2023validated}
Jan~Bouwe van~den Berg, Maxime Breden, and Ray Sheombarsing.
\newblock Validated integration of semilinear parabolic pdes, 2023.

\bibitem{IVP_Takayasu}
G.W. Duchesne, J.-P. Lessard, and A.~Takayasu.
\newblock A rigorous integrator and global existence for higher-dimensional semilinear parabolic pdes via semigroup theory.
\newblock {\em arXiv:2402.00406}, 2024.

\bibitem{MR1976079}
Xavier Cabr\'{e}, Ernest Fontich, and Rafael de~la Llave.
\newblock The parameterization method for invariant manifolds. {I}. {M}anifolds associated to non-resonant subspaces.
\newblock {\em Indiana Univ. Math. J.}, 52(2):283--328, 2003.

\bibitem{MR1976080}
Xavier Cabr\'{e}, Ernest Fontich, and Rafael de~la Llave.
\newblock The parameterization method for invariant manifolds. {II}. {R}egularity with respect to parameters.
\newblock {\em Indiana Univ. Math. J.}, 52(2):329--360, 2003.

\bibitem{MR2821596}
Jan~Bouwe van~den Berg, J.D. Mireles~James, Jean-Philippe Lessard, and Konstantin Mischaikow.
\newblock Rigorous numerics for symmetric connecting orbits: even homoclinics of the {G}ray-{S}cott equation.
\newblock {\em SIAM J. Math. Anal.}, 43(4):1557--1594, 2011.

\bibitem{MR3741385}
Jan~Bouwe van~den Berg, Maxime Breden, Jean-Philippe Lessard, and Maxime Murray.
\newblock Continuation of homoclinic orbits in the suspension bridge equation: a computer-assisted proof.
\newblock {\em J. Differential Equations}, 264(5):3086--3130, 2018.

\bibitem{plum_thesis_navierstokes}
Jonathan~Matthias Wunderlich.
\newblock {\em Computer-assisted Existence Proofs for {N}avier-{S}tokes Equations on an Unbounded Strip with Obstacle}.
\newblock PhD thesis, Karlsruher Institut für Technologie (KIT), 2022.

\bibitem{gradshteyn2014table}
I.~S. Gradshteyn and I.~M. Ryzhik.
\newblock {\em Table of integrals, series, and products}.
\newblock Elsevier/Academic Press, Amsterdam, eighth edition, 2015.
\newblock Translated from the Russian, Translation edited and with a preface by Daniel Zwillinger and Victor Moll.

\bibitem{Lessard2021conserved}
Jaime Burgos-Garc\'{\i}a, Jean-Philippe Lessard, and J.~D. Mireles~James.
\newblock Spatial periodic orbits in the equilateral circular restricted four-body problem: computer-assisted proofs of existence.
\newblock {\em Celestial Mech. Dynam. Astronom.}, 131(1):Paper No. 2, 36, 2019.

\bibitem{Lessard2021conserved_nbody}
Renato Calleja, Carlos Garc\'{\i}a-Azpeitia, Jean-Philippe Lessard, and J.~D. Mireles~James.
\newblock Torus knot choreographies in the {$n$}-body problem.
\newblock {\em Nonlinearity}, 34(1):313--349, 2021.

\bibitem{trace_polynomials}
Christine Bernardi, Monique Dauge, and Yvon Maday.
\newblock Polynomials in the {S}obolev world.
\newblock 2007.

\bibitem{Moore_interval_analysis}
Ramon~E. Moore.
\newblock {\em Interval analysis}.
\newblock Prentice-Hall, Inc., Englewood Cliffs, NJ, 1966.

\bibitem{julia_interval}
L.~Benet and D.P. Sanders.
\newblock Intervalarithmetic.jl.
\newblock 2022.
\newblock \url{ https://github.com/JuliaIntervals/IntervalArithmetic.jl}.

\bibitem{van2021spontaneous}
Jan~Bouwe van~den Berg, Maxime Breden, Jean-Philippe Lessard, and Lennaert van Veen.
\newblock Spontaneous periodic orbits in the {N}avier-{S}tokes flow.
\newblock {\em J. Nonlinear Sci.}, 31(2):Paper No. 41, 64, 2021.

\bibitem{poularikas2018transforms}
Alexander~D. Poularikas, editor.
\newblock {\em Transforms and applications handbook}.
\newblock The Electrical Engineering Handbook Series. CRC Press, Boca Raton, FL, third edition, 2010.

\bibitem{watson_bessel}
G.~N. Watson.
\newblock {\em A {T}reatise on the {T}heory of {B}essel {F}unctions}.
\newblock Cambridge University Press, Cambridge; The Macmillan Company, New York, 1944.

\bibitem{gaunt2017inequalities}
Robert~E Gaunt.
\newblock Inequalities for the modified bessel function of the second kind and the kernel of the {K}r\"atzel integral transformation.
\newblock {\em Math. Inequal. Appl.}, 20(4):987--990, 2017.

\bibitem{julia_fresh_approach_bezanson}
Jeff Bezanson, Alan Edelman, Stefan Karpinski, and Viral~B. Shah.
\newblock Julia: A fresh approach to numerical computing.
\newblock {\em SIAM Review}, 59(1):65--98, 2017.

\bibitem{julia_olivier}
Olivier Hénot.
\newblock Radiipolynomial.jl.
\newblock 2022.
\newblock \url{ https://github.com/OlivierHnt/RadiiPolynomial.jl}.

\bibitem{MR2807595}
Warwick Tucker.
\newblock {\em Validated numerics}.
\newblock Princeton University Press, Princeton, NJ, 2011.
\newblock A short introduction to rigorous computations.

\end{thebibliography}

\end{document}